\documentclass[runningheads]{llncs}
\usepackage[T1]{fontenc}
\usepackage{mathrsfs}
\usepackage{amsmath,amssymb,amsfonts}
\usepackage{tikz}
\usepackage{pifont}
\usepackage{mathtools}
\usepackage{bm}
\usepackage{enumitem}
\usepackage{forest}
\usepackage{float}
\usepackage{multirow}
\usepackage{enumitem}
\usepackage{subcaption}
\usepackage{hyperref}
\usepackage[titletoc]{appendix}
\usepackage{imakeidx}
\makeindex[columns=2, title=Index, intoc]
\usepackage{lscape}
\usepackage{graphicx}
\usepackage{thmtools}
\usepackage{thm-restate}
\newcounter{teller}

\newenvironment{tabel}{\begin{list}%
{\rm  (\alph{teller})\hfill}{\usecounter{teller} \leftmargin=1.1cm
\labelwidth=1.1cm \labelsep=0cm \parsep=0cm}
                      }{\end{list}}

\newcounter{tellerr}

\newcommand{\ca}{{\mathcal A}}
\newcommand{\cb}{{\mathcal B}}
\newcommand{\cc}{{\mathcal C}}
\newcommand{\cd}{{\mathcal D}}

\newcommand{\cg}{{\mathcal G}}

\newcommand{\cl}{{\mathcal L}}

\begin{document}
\title{Equivalence of Connected and Peak-Pit Maximal Condorcet Domains}
%
%
\author{Guanhao Li\orcidID{0009-0003-4732-415X}}
\authorrunning{G. Li}
%
\institute{Department of Mathematics, University of Auckland, New Zealand
\email{gli103@aucklanduni.ac.nz}
}
\maketitle              
\begin{abstract}
This paper provides a combinatorial proof to show that, in the study of maximal Condorcet domains, the class of peak-pit Condorcet domains, the class of connected Condorcet domains, and the class of directly connected Condorcet domains are all equivalent.
\keywords{Combinatorics \and Condorcet Domains \and Social Choice Theory}
\end{abstract}
\section{Introduction}
%
Condorcet domains are sets of linear orders on a given set of alternatives such that any majority relation induced by these orders is acyclic;
see \cite{monjardetb.2009}  for seven different yet equivalent definitions.
These domains hold a crucial place in combinatorics and are especially useful in social choice theory, where they help address voting paradoxes and aim for more rational ways to aggregate preferences \cite{arrowk.j.1951}.
Structurally, Condorcet domains can be represented as graphs, and research often focuses on their connectedness properties, particularly in peak-pit Condorcet domains. 
From a computational perspective, these domains present significant challenges, which have led to the development of efficient algorithms from combinatorial optimization to handle real-world voting scenarios \cite{brandtf.conitzerv.endrissu.langj.procacciaa2016}.
Additionally, counting and listing preference orderings within these domains helps classify different types of domains \cite{gehrleinw.2006}. 
One of the major research questions in the literature focuses on finding the largest size of Condorcet domains for a given number of alternatives \cite{monjardetb.2009}. A central research question involves characterising maximal Condorcet domains by combinatorical objects.

The class of peak-pit maximal Condorcet domains has been a primary object of investigation in the search for large Condorcet domains. Using combinatorial tools and computational methods, the largest sizes of this class have been proven to match the maximal cardinality of Condorcet domains for $n\leq 8$ alternatives \cite{fishburnp.c.1997,fishburnp.c.2002,galambosa.v.reinerv.2008,leedhamgreen.cmarkström.kriiss.2024}. Consequently, it has been shown that all such peak-pit maximal Condorcet domains are connected, where any two linear orders in the domain can be transformed into one another through a sequence of adjacent swaps within the domain. 
\cite{brandtf.ledererp.tauschs.2022} demonstrated that connected Condorcet domains can support strategyproof social decision schemes. 
This means that within such connected domains, voters have no incentive to misrepresent their preferences, leading to more honest and reliable outcomes. 
Thus, understanding connected Condorcet domains is crucial for designing voting systems that are resistant to paradoxes and manipulation, ensuring fairer decision-making processes.


Previous research has shown some connections between these two subclasses of Condorcet domain, yet the question of their equivalence remains open.
\cite{danilovv.i.karzanova.v.koshevoyg.a.2012} proved that all peak-pit maximal Condorcet domains of maximal width are connected. \cite{lig.2020} first used a mathematical method and proved that even without the restriction of maximal width, connectedness and peak-pittedness are equivalent on four alternatives. Later, \cite{lig.2023} attempted to show this equivalence is valid for any arbitrary number of alternatives. However, \cite{puppec.slinkoa.2024a} provided a counter-example showing that the work of \cite{lig.2023} is incomplete. Since their counter-example does not correspond to a maximal Condorcet domain, the problem remains open. \cite{puppec.slinkoa.2024b} conjectured this equivalence in their survey. In this paper, we aim to work from a different perspective and show the equivalence of connected and peak-pit maximal Condorcet domains.

Furthermore, within the class of connected Condorcet domains, the subclass of directly connected Condorcet domains consists of those domains where any two orders can be connected by a shortest path in the permutahedron, with all intermediate orders lying within the domain. 
This is equivalent to the { non‐restoration property} identified by \cite{satos.2013}, according to which any two linear orders can be linked by a path along which no pair of alternatives is swapped more than once.  \cite{satos.2013} showed that direct connectedness is exactly the feature underpinning the equivalence between local and global strategy‐proofness.  Moreover, \cite{satos.2013} proved that directly connected single‐peaked domains admit concise characterisations of strategy‐proof rules.
Hence, it is practically important to characterise directly connected domains in the context of Condorcet domains.
\cite{lig.puppec.slinkoa.2021} showed that all peak-pit maximal Condorcet domains of maximal width are not just connected but in fact directly connected. In this paper, we will show that the connectedness property is equivalent to the directly connectedness property within the class of maximal Condorcet domains.

Finally, our findings establish that the class of {\em maximal peak-pit Condorcet domains} is the same as the class of {\em peak-pit maximal Condorcet domains}. While there is a belief in the literature that these two concepts are equivalent, and the terms are often used interchangeably. However, no formal mathematical definitions have been provided to clarify their relationship, which leads to confusion. This paper resolves this ambiguity by providing a formal definition and proving their equivalence.
\section{Condorcet domains}


Suppose $A$ is a finite set of alternatives. Any set of linear orders on $A$ is called {\bf a domain of linear orders}, or shortly {\bf a domain}. Let $\cl(A)$ denote the set of all linear orders on $A$. Let $\cd\subseteq \cl(A)$ be a domain of linear orders on $A$. 
We denote the {\bf restriction} of $\cd$ to any subset $S$ of $A$ by $\cd_S$, where $\cd_S$ is also called the {\bf restricted domain} of $\cd$ to $S$. The restriction of $\cd$ to $A\setminus \{a\}$ is denoted by $\cd_{-a}$, where $a\in A$. The restriction of $\cd$ to $A\setminus X$ is denoted by $\cd_{-X}$, where $X\subset A$. 
For every linear order $R\in \cd$, the {\bf restriction} of $R$ to $A\setminus\{a\}$ is denoted by $R_{-a}$, where $R_{-a}$ is also called the {\bf restricted linear order} of $R$ to $A\setminus\{a\}$. The restriction of $R$ to a subset $S$ of $A$ is denoted by $R_S$. 

\begin{definition}
Let $A$ be a finite set of alternatives and $\cd\subseteq \cl(A)$ a domain of linear orders. Let $a,b,c\in A$ be a triple of distinct alternatives. Let $x\in\{a,b,c\}$ and $k\in\{1,2,3\}$. Then we say that the restricted domain $\cd_{\{a,b,c\}}$ satisfies the {\bf never-condition} $xN_{\{a,b,c\}}k$ if for all $R=x_1x_2x_3\in \cd_{\{a,b,c\}}$, one has $x\ne x_k$.
We say that $\cd_{\{a,b,c\}}$ satisfies a {\bf never-top condition} if there exists an $x\in\{a,b,c\}$ such that $\cd_{\{a,b,c\}}$ satisfies the never-condition $xN_{\{a,b,c\}}1$. 
Similarly, we say that $\cd_{\{a,b,c\}}$ satisfies a {\bf never-bottom condition} if there exists an $x\in\{a,b,c\}$ such that $\cd_{\{a,b,c\}}$ satisfies the never-condition $xN_{\{a,b,c\}}3$. 
The definition for $\cd_{\{a,b,c\}}$ satisfying a {\bf never-middle condition} is similar. 

We say that $\cd$ is a {\bf never-top domain} if for every distinct triple $a,b,c\in A$, the restricted domain $\cd_{\{a,b,c\}}$ satisfies a never-top condition. 
Similarly, we say that $\cd$ is a {\bf never-bottom domain} if for every distinct triple $a,b,c\in A$, the restricted domain $\cd_{\{a,b,c\}}$ satisfies a never-bottom condition. 
The definition for a {\bf never-middle domain} is similar.
We say that $\cd$ is a {\bf Condorcet domain} if for every distinct triple $a,b,c\in A$, the restricted domain $\cd_{\{a,b,c\}}$ is a never-top domain or never-bottom domain or never-middle domain.
A domain $\cd$ is called a {\bf peak-pit Condorcet domain} if for every distinct triple $a,b,c\in A$, the restricted domain $\cd_{\{a,b,c\}}$ is a never-top domain or never-bottom domain.
A never-condition that is a never-top or never-bottom condition is called a {\bf peak-pit condition}.
\end{definition}

\begin{remark}
    A domain $\cd$ is a Condorcet domain if and only if for every distinct $a,b,c\in A$, the restricted domain $\cd_{\{a,b,c\}}$ is a Condorcet domain.
\end{remark}



Given a Condorcet domain $\cd$,  the {\bf set of never-conditions satisfied by $\cd$} is denoted by $N(\cd)$, and the {\bf set of peak-pit-conditions satisfied by $\cd$} is denoted by $N_p(\cd)$.

One can order Condorcet domains by set inclusion. A Condorcet domain $\mathcal{D}\subseteq\mathcal{L}(A)$ is called \textbf{maximal}\index{maximal Condorcet domain} if for every Condorcet domain $\mathcal{D}'$ from $\mathcal{L}(A)$ with $\mathcal{D}\subseteq\mathcal{D}'$, it follows that $\mathcal{D}=\mathcal{D}'$. 
Equivalently, it means that a Condorcet domain $\mathcal{D}$ is maximal if for every linear order $R\notin \mathcal{D}$,  the union $\mathcal{D}\cup\{R\}$ is not a Condorcet domain. We are interested in maximal Condorcet domains which are peak-pit Condorcet domains.  
The maximal Condorcet domains for three alternatives are well-known (see, e.g., \cite{fishburnp.c.1997}). It will still be useful for us to have the following lemma.

\begin{lemma}\label{3max}
    Let $a, b, c$ be three distinct alternatives. Define
    \begin{align*}
\cd_{3,t} &= \{abc, acb, cab, cba\}, \\
\cd_{3,m} &= \{abc, bca, acb, cba\} \text{ and } \\
\cd_{3,b} &= \{abc, bac, bca, cba\}.
\end{align*}
    Let $\cd\subseteq\cl(\{a,b,c\})$ be a Condorcet domain. Then the following are equivalent:
    \begin{description}
            \item[(i) ] $\cd$ is a maximal Condorcet domain;
            \item[(ii) ] $|\cd|=4$.
            \item[(iii) ] $\cd$ is isomorphic to $\cd_{3,t}$ or $\cd_{3,m}$ or $\cd_{3,b}$.
            \end{description}
\end{lemma}
\begin{proof}
Note that since $\cd$ is a Condorcet domain, there are $x\in{\{a,b,c\}}$ and $i\in\{1,2,3\}$ such that $\cd$ satisfies a never-condition $xN_{\{a,b,c\}}i$.
    Write $\{y,z\}=\{a,b,c\}\setminus\{x\}$.
    There are three cases for $\cd$.
Suppose $\cd$ satisfies the never-top condition $xN_{\{a,b,c\}}1$.
Then $\cd\subseteq\{yxz,yzx,zxy,zyx\}$.
Suppose $\cd$ satisfies the never-middle condition $xN_{\{a,b,c\}}2$.
Then $\cd\subseteq\{xyz,xzy,yzx,zyx\}$.
Suppose $\cd$ satisfies the never-bottom condition $xN_{\{a,b,c\}}3$.
Then $\cd\subseteq\{xyz,xzy,yxz,zxy\}$.
Note that $\{yxz,yzx,zxy,zyx\}$, $\{xyz,xzy,yzx,zyx\}$, and $\{xyz,xzy,yxz,zxy\}$ are Condorcet domains.

{\bf \text{(i)} $\Leftrightarrow$ \text{(ii)}:} Suppose $\cd$ is a maximal Condorcet domain. Then maximality implies that $\cd$ equals $\{yxz$, $yzx$, $zxy$, $zyx\}$ or $\{xyz$, $xzy$, $yzx$, $zyx\}$ or $\{xyz$, $xzy$, $yzx$, $zyx\}$. Hence, $|\cd|=4$.

On the other hand, suppose $|\cd|=4$. Then $\cd$ equals $\{yxz$, $yzx$, $zxy$, $zyx\}$ or $\{xyz$, $xzy$, $yzx$, $zyx\}$ or $\{xyz, xzy, yzx, zyx\}$.
Since each of these three possibilities is a maximal Condorcet domain, $\cd$ is a maximal Condorcet domain.

{\bf \text{(ii)} $\Leftrightarrow$ \text{(iii)}:} Suppose $|\cd|=4$. Then $\cd$ equals $\{yxz, yzx, zxy, zyx\}$ or $\{xyz$, $xzy$, $yzx$, $zyx\}$ or $\{xyz, xzy, yzx, zyx\}$.
Now $\{yxz,yzx,zxy,zyx\}$ is isomorphic to $\cd_{3,t}=\{abc,acb,cab,cba\}$, and $\{xyz,xzy,yzx,zyx\}$ is isomorphic to $\cd_{3,m}=\{abc,bca,acb,cba\}$, and $\{xyz,xzy,yzx,zyx\}$ is isomorphic to $\cd_{3,b}=\{abc,bac,bca,cba\}.$
Hence, $\cd$ is isomorphic to $\cd_{3,t}$ or $\cd_{3,m}$ or $\cd_{3,b}$.

On the other hand, suppose $\cd$ is isomorphic to $\cd_{3,t}$ or $\cd_{3,m}$ or $\cd_{3,b}$. Since $\cd_{3,t}$, $\cd_{3,m}$ and $\cd_{3,b}$ all have four linear orders, $|\cd|=4$.
\end{proof}

We note that there are only three maximal Condorcet domains on the set $A=\{a,b,c\}$ up to an isomorphism, as listed in Lemma~\ref{3max}. 
While the only never-condition that $\cd_{3,t}$ satisfies is the never-top condition $bN_{\{a,b,c\}}1$, the domain $\cd_{3,m}$ only satisfies the never-middle condition $aN_{\{a,b,c\}}2$, and $\cd_{3,b}$ only satisfies the never-bottom condition $bN_{\{a,b,c\}}3$.
Both $\cd_{3,t}$ and $\cd_{3,b}$ are peak-pit maximal Condorcet domains.\\

There are various definitions of connected Condorcet domains defined over the permutahedron. We will use the concept of paths of alike linear orders, which will be defined in next section.

\section{Paths of alike linear orders}
Suppose $R=r_1\cdots r_n$ and $T=t_1\cdots t_n$ are linear orders of $\cl(A)$. We say that $R$ and $T$ are {\bf alike} if they differ by a swap of adjacent alternatives in them. Precisely, there is an $i\in[n-1]$ such that $r_i=t_{i+1}$, $r_{i+1}=t_i$, and $r_j=t_j$ for all $j\in [n]\setminus\{i,i+1\}$. So, $T=r_1\cdots r_{i-1}r_{i+1}r_ir_{i+2}\cdots r_n$ in this case.
We call $(r_i, r_{i+1})$ the {\bf switching pair of alternatives} between $R$ and $T$, or, shortly, switching pair between $R$ and $T$. We do not specify the order of the alternatives in a switching pair, so the switching pair between $R$ and $T$ can also be written as $(r_{i+1},r_i)$.

\begin{definition}
   Let $R$ and $T$ be two linear order in $\cl(A)$. We say that $\ca=(R_1,\ldots, R_k)$ is a {\bf path of alike linear orders} connecting $R$ and $T$, or shortly a {\bf path connecting} $R$ and $T$, if $R_1=R, R_k=T$, and $R_i$ and $R_{i+1}$ are alike for all $i\in[k-1]$. The number $k\in\mathbb{N}$ is called the {\bf length} of the path $\ca$ connecting $R$ and $T$. 
   We say that $\ca=(R_1,\ldots, R_k)$ is a {\bf geodesic of alike linear orders} connecting $R$ and $T$, or shortly a {\bf geodesic connecting} $R$ and $T$, if 
   \begin{equation*}
\begin{split}
k &= \min\bigl\{\,m\in\mathbb{N} : (S_1,\dots,S_m)\text{ is a path of alike linear orders}\\
   &\qquad\qquad\text{connecting }R\text{ and }T\bigr\}.
\end{split}
\end{equation*}
\end{definition}

Note that the length of a geodesic connecting linear orders $R$ and $T$ is one more than the length of the reduced decomposition of a permutation sending $R$ to $T$ (see, e.g., \cite{orientedmatroids.1999}, Sec.~6.4). The Kendall tau distance between $R$ and $T$ is the number of pairs of alternatives that are ranked differently by $R$ and $T$. The length of a geodesic connecting $R$ and $T$ is one more than the Kendall tau distance between $R$ and $T$.

\begin{example}
Note that there are only six linear orders in $\cl({\{a,b,c\}})$, namely $$\cl({\{a,b,c\}})=\{abc,bac,bca,cba,cab,acb\}.$$ 
If an edge is used to connect any two alike linear orders in $\cl({\{a,b,c\}})$, we obtain a diagram that shows all paths connecting any two alike linear orders in $\cl({\{a,b,c\}})$, as shown in Figure~\ref{g3paths}, where the switching pair between two adjacent linear orders is labeled on the edge that connects them.

\begin{figure}[H]
 \centering
\begin{tikzpicture}[scale=2]

\draw[fill=black] (-1,0) circle (2pt);
\draw[fill=black] (0.5,0.866) circle (2pt);
\draw[fill=black] (-0.5,-0.866) circle (2pt);
\draw[fill=black] (1,0) circle (2pt);
\draw[fill=black] (-0.5,0.866) circle (2pt);
\draw[fill=black] (0.5,-0.866) circle (2pt);
 
\node at (-1.3,0) {$bac$};
\node at (-0.5,1.1) {$bca$};
\node at (0.5,1.1) {$cba$};
\node at (1.3,0) {$cab$};
\node at (-0.5,-1.1) {$abc$};
\node at (0.5,-1.1) {$acb$};

\node at (-1,-0.5) {$(a,b)$};
\node at (-1,0.5) {$(a,c)$};
\node at (0,1) {$(b,c)$};
\node at (1,0.5) {$(a,b)$};
\node at (1,-0.5) {$(a,c)$};
\node at (0,-1) {$(b,c)$};
 
\draw[thick] (1,0) -- (0.5,0.866)-- (-0.5,0.866)--(-1,0)--(-0.5,-0.866) --(0.5,-0.866)--(1,0);

\end{tikzpicture}
    \caption{All paths connecting any two linear orders in $\cl(\{a,b,c\})$.} \label{g3paths}
\end{figure}
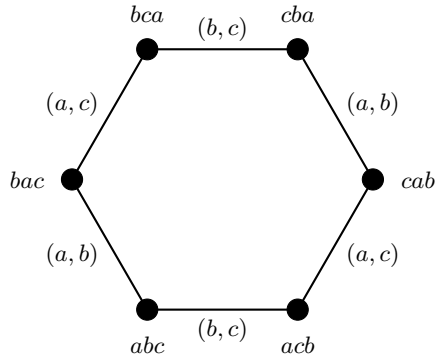 

While $\ca=(abc,bac,bca)$ and $\cb=(abc,acb,cab,cba,bca)$ are both paths connecting the same linear orders, $abc$ and $bca$, the path $\ca$ is a geodesic but $\cb$ is not since the number of linear orders in the latter one is not minimal. 
Note that $abc$ and $bca$ differ by three swaps of adjacent alternatives, and and this is also the number of swaps in the shortest path (geodesic) connecting $abc$ and $bca$. This tells us some obvious facts about paths of alike linear orders.
\end{example}

Let $\ca=(R_1,\ldots, R_k)$ be a path of alike linear orders, 
we use $K(\ca)=\{R_1,\ldots, R_k\}$ to represent the {\bf set of linear orders from $\ca$}.
We denote $S(\ca)$ to be its {\bf sequence of switching pairs of alternatives}. Observe that each path of alike linear orders is uniquely determined by its sequence of switching pairs of alternatives. The number of switching pairs in $S(\ca)$ is called the {\bf length} of $S(\ca)$.
Note that the length of $S(\ca)$ is always one less than the length of $\ca$.

Let $\ca=(R_1,\ldots, R_k)$ be a path of alike linear orders on $\cl(A)$. Let $a\in A$ and $B\subseteq A$. Then the {\bf restricted path} $\ca_B$ is the path of alike linear orders on $\cl(B)$ obtained from $(R_1)_B,\ldots, (R_k)_B$ and removing duplicate equal linear orders which are next to each other.
Similarly, the {\bf restricted path} $\ca_{-a}$ is the path of alike linear orders on $\cl(A_{-a})$ obtained from $(R_1)_{-a},\ldots, (R_k)_{-a}$ and removing duplicate equal linear orders which are next to each other.
Note that the length of $\ca_B$ can be less than the length of $\ca$. We also note that switching pairs of alternatives in a path are defined over a pair of alike linear orders in the path. Hence, if $(a,b)$ is a switching pair in a path, then it must occur at least once in it.

  Let $(a,b)$ and $(c,d)$ be two switching pairs in a geodesic $\ca$. If $(a,b)$ occurs before $(c,d)$ in $S(\ca)$, we denote that by $$(a,b)\vartriangleleft (c,d).$$ We similarly define $\trianglelefteq$.
  

We will now investigate the restriction of any geodesic to a triple of alternatives that differ by three swaps of adjacent alternatives.
Let $R=abc$ and $T=cba$. Then $\ca_1=(abc,bac,bca,cba)$ and $\ca_2=(abc,acb,cab,cba)$ are the only possible geodesics connecting them, as shown in Figure~\ref{g3paths}. 
Note that both geodesics rise to peak-pit Condorcet domains, since they connect $abc$ and $cba$, the property in Lemma~\ref{wex0}\ref{wex0b} holds for them.
These two geodesics can be represented by the wiring diagrams shown in Figure~\ref{twogeodesics}.

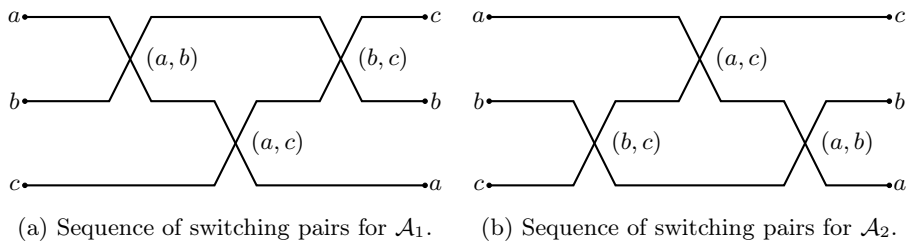
\begin{figure}[H]
 \begin{subfigure}[b]{0.5\textwidth}
 \centering
\begin{tikzpicture}[scale=0.28]
\draw[fill=black] (0,0) circle (3pt);
\draw[fill=black] (0,4) circle (3pt);
\draw[fill=black] (0,8) circle (3pt);
\draw[fill=black] (19,0) circle (3pt);
\draw[fill=black] (19,4) circle (3pt);
\draw[fill=black] (19,8) circle (3pt);
 
\node at (-0.5,0) {$c$};
\node at (-0.5,4) {$b$};
\node at (-0.5,8) {$a$};
\node at (19.5,0) {$a$};
\node at (19.5,4) {$b$};
\node at (19.5,8) {$c$};

\node at (12,2) {$(a,c)$};
\node at (17,6) {$(b,c)$};
\node at (7,6) {$(a,b)$};

\draw[thick] (0,0)-- (9,0) --(11,4) --(14,4) --(16,8) -- (19,8);
\draw[thick] (0,4) -- (4,4) --(6,8) --(14,8)--(16,4) --(19,4);
\draw[thick] (0,8) -- (4,8) --(6,4)-- (9,4) --(11,0) -- (19,0);
\end{tikzpicture}
   \caption{Sequence of switching pairs for $\ca_1$.}\label{}
\end{subfigure}
\begin{subfigure}[b]{0.5\textwidth}
 \centering 
\begin{tikzpicture}[scale=0.28]
\draw[fill=black] (0,0) circle (3pt);
\draw[fill=black] (0,4) circle (3pt);
\draw[fill=black] (0,8) circle (3pt);
\draw[fill=black] (19,0) circle (3pt);
\draw[fill=black] (19,4) circle (3pt);
\draw[fill=black] (19,8) circle (3pt);
 
\node at (-0.5,0) {$c$};
\node at (-0.5,4) {$b$};
\node at (-0.5,8) {$a$};
\node at (19.5,0) {$a$};
\node at (19.5,4) {$b$};
\node at (19.5,8) {$c$};

\node at (12,6) {$(a,c)$};
\node at (17,2) {$(a,b)$};
\node at (7,2) {$(b,c)$};

\draw[thick] (0,0)-- (4,0)--(6,4)--(9,4) --(11,8) -- (19,8);
\draw[thick] (0,4) -- (4,4) --(6,0) --(14,0)--(16,4) --(19,4);
\draw[thick] (0,8) -- (9,8) --(11,4)--(14,4)--(16,0) -- (19,0);
\end{tikzpicture}
   \caption{Sequence of switching pairs for $\ca_2$.}\label{}
 
\end{subfigure}
\caption{Sequences of switching pairs for two geodesics connecting $abc$ and $cba$.}\label{twogeodesics}
\end{figure}

Note that the sequences of switching pairs for $\ca_1$ and $\ca_2$ are in reversed order. While the set of linear orders from $\ca_1$ only satisfies one never-condition, which is the never-bottom condition $bN_{\{a,b,c\}}3$, the set from $\ca_2$ also only satisfies one never-condition, which is the never-top condition $bN_{\{a,b,c\}}1$. 
As a result, we refer to the first geodesic $\ca_1$ as a {\bf never-bottom geodesic}, while the second geodesic $\ca_2$ is referred to as a {\bf never-top geodesic}.
It is important to note that the switching pairs occurring in $\ca_1$ involve first switching the top two alternatives, i.e., $(a,b)$, while those in $\ca_2$ involve first switching the bottom two ranked alternatives, i.e., $(b,c)$. This dichotomy can be generalised into the following lemma.

\begin{lemma}\label{2geodesics}
    Let $a,b,c$ be distinct alternatives. Let $T\in\cl(\{a,b,c\})$. Let $\ca=(R_1,\ldots, R_k)$ be a geodesic on $\cl(\{a,b,c\})$ connecting $abc$ and $T$. Then the following are equivalent:
    \begin{description}
            \item[(I) ] $T=cba$;
            \item[(II) ] $ (a,b)$, $(a,c)$ and $(b,c)$ are switching pairs in $S(\ca)$.
    \end{description}
    Now suppose $T=cba$.
    Then $\{R_1,\ldots, R_k\}$ satisfies a unique never-condition which is either $bN_{\{a,b,c\}}3$ or $bN_{\{a,b,c\}}1$. Moreover,
    \begin{tabel}
        \item\label{2geodesicsa} The following are equivalent: \begin{description}
            \item[(i) ] $N(\{R_1,\ldots, R_k\})=\{bN_{\{a,b,c\}}3\}$;
            \item[(ii) ] $\ca=(abc,bac,bca,cba)$, that is it is a never-bottom geodesic;
            \item[(iii) ] $S(\ca):  (a,b)\vartriangleleft(a,c)\vartriangleleft(b,c)$.
            \end{description}
        \item\label{2geodesicsb} The following are equivalent: \begin{description}
        \item[(i) ] $N(\{R_1,\ldots, R_k\})=\{bN_{\{a,b,c\}}1\}$;
            \item[(ii) ] $\ca=(abc,acb,cab,cba)$, that is it is  a never-top geodesic;
            \item[(iii) ] $S(\ca):  (b,c)\vartriangleleft(a,c)\vartriangleleft(a,b)$.
                    \end{description}
    \end{tabel}
\end{lemma}
\begin{proof}
For (II) $\Rightarrow$ (I), since $(a,b)$, $(a,c)$ and $(b,c)$ are switching pairs in $S(\ca)$, we have that $abc$ and $T$ differ by three swaps of adjacent alternatives. Hence, $T=cba$. 
For (I) $\Rightarrow$ (II), suppose $T=cba$. Then $\ca$ be a geodesic connecting $abc$ and $cba$. Inspecting Figure~\ref{g3paths}, we note that there are exactly 2 geodesics on $\cl(\{a,b,c\})$ connecting $abc$ and $cba$, which are
$\ca_1=(abc,bac,bca,cba)$ and $\ca_2=(abc,acb,cab,cba)$.
In particular, (II) is valid. 
Note that $\ca_1$ gives the domain $\{abc$, $bac$, $bca$, $cba\}$ satisfying a unique never-condition $bN_{\{a,b,c\}}3$, while $\ca_2$ gives the domain $\{abc,acb,cab,cba\}$ satisfying a unique never-condition $bN_{\{a,b,c\}}1$.

(a) Suppose $\{R_1,\ldots, R_k\}$ satisfies the never-top condition $bN_{\{a,b,c\}}3$. Since $\ca$ connecting $abc$ and $cba$, the linear order $bac$ or $bca$ that ranks $b$ first needs to be in $\{R_1,\ldots, R_k\}$. Hence, $\ca=\ca_1=(abc,bac,bca,cba)$. In particular, the sequence of switching pairs $S(\ca)$ is $(a,b),(a,c),(b,c)$, as shown in Figure~\ref{twogeodesics}(a).
On the other hand, suppose $S(\ca)$ is $(a,b),(a,c),(b,c)$. Since $\ca$ connecting $abc$ and $cba$, the geodesic $\ca=\ca_1=(abc,bac,bca,cba)$. Thus, the domain $\{R_1,\ldots, R_k\}=\{abc,bac,bca,cba\}$ satisfies the never-top condition $bN_{\{a,b,c\}}3$.

(b) Suppose $\{R_1,\ldots, R_k\}$ satisfies the never-top condition $bN_{\{a,b,c\}}1$. Since $\ca$ connecting $abc$ and $cba$, the linear order $acb$ or $cab$ that ranks $b$ last needs to be in $\{R_1,\ldots, R_k\}$. Hence, $\ca=\ca_2=(abc,acb,cab,cba)$. In particular, the sequence of switching pairs $S(\ca)$ is $(b,c),(a,c),(a,b)$, as shown in Figure~\ref{twogeodesics}(b).
On the other hand, suppose $S(\ca)$ is $(b,c), (a,c),(a,b)$. Since $\ca$ connecting $abc$ and $cba$, the geodesic $\ca=\ca_2=(abc,acb,cab,cba)$. Thus, the domain $\{R_1,\ldots, R_k\}=\{abc,acb,cab,cba\}$ satisfies the never-top condition $bN_{\{a,b,c\}}1$.
\end{proof}



We will use paths of alike linear orders to define connectedness for a domain.

\begin{definition}Let $\cd \subseteq \cl(A)$ be a domain and $R,T\in \cd$. We say that $R$ and $T$ are {\bf connected} in $\cd$ if there exists a path of alike linear orders $\ca=(R_1,\ldots, R_k)$ connecting $R$ and $T$ such that  $R_i\in\cd$ for all $i\in[k]$. We say that $R$ and $T$ are {\bf directly connected} in $\cd$ if there exists a geodesic of alike linear orders $\ca=(R_1,\ldots, R_k)$ connecting $R$ and $T$ such that  $R_i\in\cd$ for all $i\in[k]$. We then call $\ca$ a {\bf geodesic on $\cd$ connecting} $R$ and $T$.
A domain $\cd$ is called {\bf connected} if for all $R,T\in \cd$, the linear orders $R$ and $T$ are connected in $\cd$. 
A domain $\cd$ is called {\bf directly connected} if for all $R,T\in \cd$, the linear orders $R$ and $T$ are directly connected in $\cd$. 
\end{definition}



Obviously, if a domain $\cd$ is directly connected, then it is connected by definition.

\begin{example}
    Note that there are only two peak-pit maximal Condorcet domains for three alternatives $a,b,c$, and which contain $abc$, namely,
    $$\cd_{3,b}=\{abc, bac,bca,cba\}\text{ and }\cd_{3,t}=\{abc,acb,cab,cba\}.$$
    While the only never-condition that $\cd_{3,b}$ satisfies is the never-bottom condition $bN_{\{a,b,c\}}3$,  the domain $\cd_{3,t}$ only satisfies the never-top condition $bN_{\{a,b,c\}}1$. However, both domains are connected and directly connected as shown in Figure~\ref{g3paths}.
\end{example}

\begin{remark}
Note that our definition of connectedness is equivalent to that given by \cite{satos.2013}. Moreover, by Lemmas~\ref{wex1.1}\ref{wex1.1a} and \ref{wex1.1}\ref{wex1.1b}, the directly connected domains defined here are equivalent to those satisfying the no-restoration property from \cite{satos.2013}.
\end{remark}

\section{Equivalence of Connectedness and peak-pittedness}

The following proposition is central to establishing the equivalence between connectedness and peak-pittedness. Due to its length and technical nature, the proof of Proposition~\ref{conex2} is provided in Appendix A.

\begin{proposition}\label{conex2}
    Let $\cd\subseteq\cl(A)$ be a peak-pit Condorcet domain. Let $R, T\in \cd$ be linear orders and $B\subseteq A$ be a subset such that $|B|\geq 3$. Then there exists a geodesic $\ca$  connecting $R_{B}$ and $T_{B}$ such that $\cd_{B}\cup K(\ca)$ is a peak-pit Condorcet domain. 
\end{proposition}

We need one more lemma before showing our main theorems.

\begin{lemma}\label{wex1.9}
    Let $A$ be a set of at least four alternatives. Let
    $\cd\subseteq \cl(A)$ be a domain, and $a\in A$. Then \begin{tabel} 
\item\label{wex1.9a} If $\cd$ is connected, then $\cd_{-a}$ is connected. \item\label{wex1.9b} If $\cd$ is directly connected, then $\cd_{-a}$ is directly connected.
    \end{tabel}
\end{lemma}
\begin{proof}
    (a) Suppose $\cd\subseteq \cl(A)$ is connected, and $a\in A$.
    Let $R, T\in \cd_{-a}$ be any two linear orders, then there are linear orders $R^*, T^*\in\cd$ such that $R^*_{-a}=R$ and $T^*_{-a}=T$. Since $\cd$ is connected, there is a path $\ca=(R^1,\ldots, R^k)$ on $\cd$ connecting $R^*$ and $T^*$.
    Hence, the restricted path $\ca_{-a}=(R^1_{-a},\ldots, R^{k'}_{-a})$ is a path connecting $R$ and $T$. Since $R^i\in\cd$, we have $R^i_{-a}\in\cd_{-a}$, for all $i\in [k]$. Hence, $\ca_{-a}$ is a path on $\cd_{-a}$ and so $\cd_{-a}$ is connected.
    
    (b) suppose $\cd\subseteq \cl(A)$ is directly connected, and $a\in A$. Let $R, T\in \cd_{-a}$ be any two linear orders, then there are linear orders $R^*, T^*\in\cd$ such that $R^*_{-a}=R$ and $T^*_{-a}=T$. Since $\cd$ is directly connected, there is a geodesic $\ca=(R^1,\ldots, R^k)$ on $\cd$ connecting $R^*$ and $T^*$.
    Hence, the restricted path $\ca_{-a}=(R^1_{-a},\ldots, R^{k'}_{-a})$ is a path connecting $R$ and $T$. Moreover, since $\ca$ is a geodesic and $S(\ca_{-a})$ is obtained by removing all switching pairs that contain $a$ from $S(\ca)$, we have that $\ca_{-a}$ is also a geodesic.
    Since $R^i\in\cd$, we have $R^i_{-a}\in\cd_{-a}$, for all $i\in [k]$. Hence, $\ca_{-a}$ is a geodesic on $\cd_{-a}$ and so $\cd_{-a}$ is directly connected.
\end{proof}

Although our primary interest lies in maximal Condorcet domains that satisfy the peak-pit condition, it is still useful to understand the structure of the peak-pit condition itself, rather than focusing solely on the maximality of the domain in terms of size. 

\begin{definition}
Let $\mathcal{D}\subseteq\mathcal{L}(A)$ be a peak-pit Condorcet domain. 
If for every Condorcet domain  $\mathcal{D}'\subseteq\mathcal{L}(A)$ with $\mathcal{D}\subseteq\mathcal{D}'$, one has $\mathcal{D}=\mathcal{D}'$, then $\cd$ is called a {\bf peak-pit maximal Condorcet domain}.
If for every peak-pit Condorcet domain $\mathcal{D}'\subseteq\mathcal{L}(A)$ with $\mathcal{D}\subseteq\mathcal{D}'$, one has $\mathcal{D}=\mathcal{D}'$, then $\cd$ is called a {\bf maximal 
peak-pit Condorcet domain}.
\end{definition}
Using the concept of connectedness, we establish that the class of maximal peak-pit Condorcet domains is equivalent to the class of peak-pit maximal Condorcet domains. Before showing this, let us first see an important property of maximal peak-pit Condorcet domains.

\begin{theorem}\label{mpp1}
    Every maximal peak-pit Condorcet domain is directly connected.
\end{theorem}
\begin{proof}
    Let $\cd\subseteq\cl(A)$ be a maximal peak-pit Condorcet domain. If we consider $B=A$ as in Proposition~\ref{conex2}, then for every pair of linear orders $R, T\in \cd$, there exists a geodesic $\ca$  connecting $R$ and $T$ such that $\cd\cup K(\ca)$ is a peak-pit Condorcet domain.
Since $\cd$ is a maximal peak-pit Condorcet domain, we have $K(\ca)\subseteq \cd$. Thus, $\ca$ is a geodesic on $\cd$ and so $\cd$ is directly connected.
\end{proof}

Consequently, we have the following theorem.

\begin{theorem}\label{mpp2}
    Let $\cd$ be a Condorcet domain. Then $\cd$ is a peak-pit maximal Condorcet domain if and only if $\cd$ is a maximal peak-pit Condorcet domain.
\end{theorem}
\begin{proof}
Suppose $\cd$ is a peak-pit maximal Condorcet domain.
Since every peak-pit Condorcet domain is a Condorcet domain, $\cd$ is a maximal peak-pit Condorcet domain by definition.

Conversely, suppose $\cd\subseteq\cl(A)$ is a maximal peak-pit Condorcet domain. 
By Theorem~\ref{mpp1}, the domain $\cd$ is directly connected.
Now we will show that $\cd$ is a maximal Condorcet domain. Suppose not.
Then there is a linear order $R\in\cl(A)$ such that $\cd\cup\{R\}$ is a Condorcet domain but not a peak-pit Condorcet domain.
Hence, there is a triple $a,b,c\in A$ such that $(\cd\cup\{R\})_{\{a,b,c\}}$ only satisfies a never-middle condition. Without loss of generality,  
suppose $(\cd\cup\{R\})_{\{a,b,c\}}$  satisfies the never-middle condition $cN_{\{a,b,c\}} 2$, namely, $(\cd\cup\{R\})_{\{a,b,c\}}\subseteq\{abc,bac,cab, cba\}$.
Because $(\cd\cup\{R\})_{\{a,b,c\}}$ does not satisfy any never-top or never-bottom condition and any proper subset of $\{abc,bac,cab, cba\}$ satisfies a never-top or never-bottom condition, we have $(\cd\cup\{R\})_{\{a,b,c\}}=\{abc,bac,cab, cba\}$, as shown in Figure~\ref{g3paths2}.
Moreover, since $\cd$ is a peak-pit Condorcet domain, we have $R_{\{a,b,c\}}\in\{abc,bac,cab, cba\}$.
Hence, $\cd_{\{a,b,c\}}=\{abc,bac,cab, cba\}\setminus\{R_{\{a,b,c\}}\}$.
Note that removing any one element from $\{abc,bac,cab,cba\}$ will cause the set of remaining elements to be disconnected.
In particular, if $R_{\{a,b,c\}}$ is $abc$ or $cab$, then $bac$ and $cba$ is not directly connected in $\cd_{\{a,b,c\}}$.
Similarly, if $R_{\{a,b,c\}}$ is $bac$ or $cba$, then $abc$ and $cab$ is not directly connected in $\cd_{\{a,b,c\}}$. Since $\cd$ is directly connected, that contradicts Lemma~\ref{wex1.9}\ref{wex1.9b}.
Hence, $\cd$ is a maximal Condorcet domain.
\end{proof}

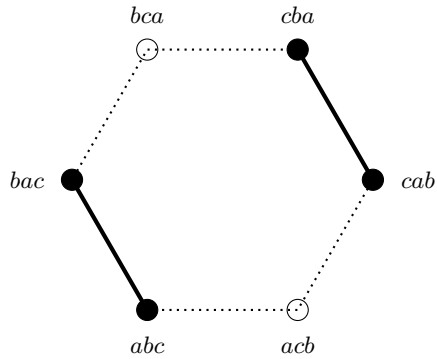
\begin{figure}[h]
 \centering
\begin{tikzpicture}[scale=2]

\draw[fill=black] (-1,0) circle (2pt);
\draw[fill=black] (0.5,0.866) circle (2pt);
\draw[fill=black] (-0.5,-0.866) circle (2pt);
\draw[fill=black] (1,0) circle (2pt);
\draw[fill=white] (-0.5,0.866) circle (2pt);
\draw[fill=white] (0.5,-0.866) circle (2pt);
 
\node at (-1.3,0) {$bac$};
\node at (-0.5,1.1) {$bca$};
\node at (0.5,1.1) {$cba$};
\node at (1.3,0) {$cab$};
\node at (-0.5,-1.1) {$abc$};
\node at (0.5,-1.1) {$acb$};
 

\draw[ultra thick] (1,0) -- (0.5,0.866);
\draw[ultra thick] (-1,0)--(-0.5,-0.866);
\draw[thick, dotted] (-0.5,-0.866) --(0.5,-0.866)--(1,0);
\draw[thick, dotted] (0.5,0.866)-- (-0.5,0.866)--(-1,0);

\end{tikzpicture}
    \caption{A never-middle domain that is not connected in $\cl(\{a,b,c\})$.} \label{g3paths2}
\end{figure} 

Next we will see that, in the class of maximal Condorcet domains, peak-pittedness, direct connectedness, and connectedness are actually equivalent.

\begin{theorem}\label{wex2}
    Let $\cd$ be a maximal Condorcet domain. Then the following are equivalent: 
    \begin{description}
            \item[(i) ] $\cd$ is a peak-pit Condorcet domain;
            \item[(ii) ] $\cd$ is directly connected;
            \item[(iii) ] $\cd$ is connected.
            \end{description}
\end{theorem}
\begin{proof}

{\bf \text{(i)} $\Rightarrow$ \text{(ii)}:}
Suppose $\cd\subseteq\cl(A)$ is a peak-pit maximal Condorcet domain. 
By Theorem~\ref{mpp2}, $\cd$ is a maximal peak-pit Condorcet domain.  Hence by Theorem~\ref{mpp1}, it is directly connected.

{\bf \text{(ii)} $\Rightarrow$ \text{(iii)}:}
Suppose $\cd$ is directly connected, then every two linear orders in it are connected by a geodesic on $\cd$. Since all geodesics are paths, $\cd$ is connected.

{\bf \text{(iii)} $\Rightarrow$ \text{(i)}:}
    Suppose $\cd$ is a connected maximal Condorcet domain. Suppose $\cd$ is not a peak-pit Condorcet domain for a contradiction. Then there is a triple of distinct alternatives $a,b,c$ such that $\cd_{\{a,b,c\}}$ does not satisfy a never-top or never-bottom condition. Since $\cd$ is a Condorcet domain, $\cd_{\{a,b,c\}}$  satisfies a never-middle condition. Without loss of generality, suppose $\cd_{\{a,b,c\}}$  satisfies the never-middle condition $cN_{\{a,b,c\}} 2$, namely, $\cd_{\{a,b,c\}}\subseteq\{abc,bac,cab, cba\}$.
However, since $\cd_{\{a,b,c\}}$ does not satisfy any never-top or never-bottom condition and any proper subset of $\{abc,bac,cab, cba\}$ would satisfy a never-top or never-bottom condition, we have $\cd_{\{a,b,c\}}=\{abc,bac,cab, cba\}$, as shown in Figure~\ref{g3paths2}. Note that $\cd_{\{a,b,c\}}$ is not connected, because no path on $\cd_{\{a,b,c\}}$ connects $abc$ and $cba$.
However, since $\cd$ is connected, $\cd_{\{a,b,c\}}$ is also connected by Lemma~\ref{wex1.9}\ref{wex1.9a}. This is a contradiction and so $\cd$ is a peak-pit Condorcet domain. 
\end{proof}
    






\section{Conclusion}
We have established the equivalence of three fundamental classes of maximal Condorcet domains: the class of connected Condorcet domains, the class of peak-pit Condorcet domains, and the class of directly connected Condorcet domains. By examining the structure of geodesics of alike linear orders, we resolved a long-standing open question and unified these previously distinct perspectives under the maximality condition. Our results not only advance the theoretical understanding of Condorcet domains but also clarify the relationship between maximal peak-pit Condorcet domains and peak-pit maximal Condorcet domains, providing formal definitions to resolve ambiguity in the literature.

With these new findings, the work of \cite{lig.2023} can be carried forward. They introduced combinatorial tools, such as weakly separated ideals and generalised arrangements of pseudolines, to explore the combinatorial structure of peak-pit maximal Condorcet domains. Our results provide a solid foundation for future research, opening new avenues to investigate the deeper combinatorial properties of these domains.

\begin{credits}
\subsubsection{\ackname} 
I thank Tom ter Elst and Dominik Peters for their thorough review and invaluable feedback on the manuscript. I also thank the anonymous reviewers for their insightful comments.
I gratefully acknowledge financial support from the Department of Mathematics at the University of Auckland.

\subsubsection{\discintname}
The author has no competing interests to declare that are
relevant to the content of this article. 
\end{credits}

\section*{\Large Appendix: Proof of Proposition \ref{conex2}}
The idea behind proving Proposition~\ref{conex2} is as follows: Given any two linear orders in a peak-pit Condorcet domain, we can construct a geodesic connecting these linear orders such that the peak-pit condition arising from this geodesic is compatible with the ones satisfied by the domain containing these two linear orders.

To achieve this construction, we need to examine lemmas from different perspectives, including lemmas related to maximal Condorcet domains and lemmas related to connected Condorcet domains, which will be presented through geodesics. The latter case involves further examination of the equivalence class of a geodesic.
Using these lemmas, we can enrich peak-pit Condorcet domains in terms of their progress toward maximality. Additionally, we investigate a special property called swap closure of equivalent geodesics, which eliminates some of the obstacles and enables the proof to proceed.
Finally, we will list some technical lemmas related to peak-pit conditions before proceeding with the proof of Proposition~\ref{conex2}.

\setcounter{section}{0}
\section{Lemmas for Condorcet domains}
We first look at some special cases of a peak-pit Condorcet domain, which will be useful for our later analysis.

\begin{lemma}\label{wex0} Let $\cd\subseteq \cl(A)$ be a peak-pit Condorcet domain and let $a,b,c\in A$ be distinct alternatives. Suppose $abc,cba\in\cd_{\{a,b,c\}}$. Then
\begin{tabel}
    \item\label{wex0a} $\emptyset\ne N_p(\cd_{\{a,b,c\}})\subseteq \{bN_{\{a,b,c\}}1$, $bN_{\{a,b,c\}}3\}$.  
    \item\label{wex0b}  $N_p(\cd_{\{a,b,c\}})= \{bN_{\{a,b,c\}}1$, $bN_{\{a,b,c\}}3\}$ if and only if $\cd_{\{a,b,c\}}=\{abc,cba\}$. 
\end{tabel}
\end{lemma}
\begin{proof}
    (a) Since $\cd$ is a peak-pit Condorcet domain, there are six possible peak-pit conditions for $\cd_{\{a,b,c\}}$ of which four are excluded by the conditions $abc,cba\in\cd_{\{a,b,c\}}$.

    (b) Suppose $\cd_{\{a,b,c\}}$ satisfies both $bN_{\{a,b,c\}}1$ and $bN_{\{a,b,c\}}3$. Now if $R\in\cd_{\{a,b,c\}}$, then alternatives $b$ can never appear first or last in $R$. Hence, $R$ can only be $abc$ or $cba$. That is, $\cd_{\{a,b,c\}}\subseteq\{abc,cba\}$. Since both $abc$ and $cba$ are in $\cd_{\{a,b,c\}}$, the equality holds.
    On the other hand, if $\cd_{\{a,b,c\}}=\{abc,cba\}$, then $\cd_{\{a,b,c\}}$ satisfies both $bN_{\{a,b,c\}}1$ and $bN_{\{a,b,c\}}3$. Hence, the only peak-pit-conditions that $\cd_{\{a,b,c\}}$ satisfies are these two conditions by Statement~\ref{wex0a}.
\end{proof}

\section{Lemmas for paths of alike linear orders}

We investigate several useful properties of paths of alike linear orders.

\begin{lemma}\label{wex1.1}
\begin{tabel} 
\item\label{wex1.1a} Any switching pair of alternatives in a geodesic $\ca$ occurs exactly once in $\ca$.
\item\label{wex1.1b} Let $\ca=(R_1,\ldots, R_k)$ be a path of alike linear orders for which any switching pair of alternatives occurs exactly once. Then $\ca$ is a geodesic connecting $R_1$ and $R_k$.
\item\label{wex1.1c} Let $\ca=(R_1,\ldots, R_k)$ be a geodesic of alike linear orders on $\cl(A)$ and $B\subseteq A$. Then $\ca_B$ is also a geodesic on $\cl(B)$.
\item\label{wex1.1d} If $\ca=(R_1,\ldots, R_i)$ and $\cb=(T_1,\ldots, T_j)$ are geodesics such that $R_i=T_1$ and the switching pairs of alternatives in $\ca$ and $\cb$ are pairwise distinct, then $\cc=(R_1,\ldots, R_i, T_2,\ldots, T_j)$ is a geodesic connecting $R_1$ and $T_j$.
\end{tabel}
\end{lemma}
\begin{proof}
    (a) Let $\ca=(R_1,\ldots, R_k)$ be a geodesic connecting $R$ and $T$. The switching pairs of alternatives that occur in $\ca$ are precisely the pairs of alternatives that are ranked differently by $R$ and $T$. Each of the latter pairs occurs at most once, and so any switching pair of alternatives occurs at most once. Since switching pairs of alternatives in $\ca$ are defined over a pair of alike linear orders in $\ca$. If $(a,b)$ is a switching pair in $\ca$, then it must occur at least once. Hence, every switching pair in $\ca$ occurs exactly once.

    (b) Let $\ca=(R_1,\ldots, R_k)$ be a path of alike linear orders for which any switching pair of alternatives occurs exactly once. Then all switching pairs of alternatives in $\ca$ are pairwise distinct, and so there are $k-1$ pairs of alternatives that are ranked differently by $R_1$ and $R_k$. Hence, the length of $\ca$ is one more than such number and so the length of $\ca$ is minimal for a path connecting $R_1$ and $R_k$. Thus, $\ca$ is a geodesic connecting $R_1$ and $R_k$.

    (c)  Let $\ca=(R_1,\ldots, R_k)$ be a geodesic of alike linear orders on $\cl(A)$ and $B\subseteq A$. Then any switching pair of alternatives in $\ca$ occurs exactly once in $\ca$. Hence, any switching pair of alternatives in $\ca_B$ also occurs exactly once in $\ca_B$, and so $\ca_B$ is also a geodesic due to Statement~\ref{wex1.1b}.
    
    (d) Suppose $\ca=(R_1,\ldots, R_i)$ and $\cb=(T_1,\ldots, T_j)$ are geodesics such that the switching pairs of alternatives in them are pairwise distinct. Then any switching pair of alternatives in $\cc=(R_1,\ldots, R_i, T_2,\ldots, T_j)$ occurs exactly once. Hence, the proof follows immediately from Statements~\ref{wex1.1a} and \ref{wex1.1b}.
   \end{proof} 

On a triple of alternatives, we have observed a dichotomy in Lemma~\ref{2geodesics} when a geodesic includes three switching pairs. 
However, when we are only aware of two of the switching pairs from a geodesic, there will be six distinct cases to consider. 
If we know the ordering of two switching pairs in a geodesic on a triple of alternatives, we can also determine in four cases that the third switching pair must exist in the geodesic and in which order the three occur. In the two remaining cases the third switching pair may or may not exist and if it exists, then it must be in a certain order.

\begin{lemma}\label{2geodesics3}
    Let $a,b,c$ be distinct alternatives. Let $T\in\cl(\{a,b,c\})$ and let $\ca$ be a geodesic on $\cl(\{a,b,c\})$ connecting $abc$ and $T$. Then
    \begin{tabel}
    \item\label{2geodesics3a} 
    If $(a,b)$ and $(a,c)$ are switching pairs in $S(\ca)$ such that $(a,b)\vartriangleleft (a,c)$,
    then $(a,b)$ appears first in $S(\ca)$.
    Moreover, if $(b,c)$ is a switching pair in $S(\ca)$, then $(a,b)\vartriangleleft (a,c)\vartriangleleft(b,c)$ and $T=cba$.
\item\label{2geodesics3b}
        If $(a,b)$ and $(a,c)$ are switching pairs in $S(\ca)$ such that $(a,c)\vartriangleleft (a,b)$,
    then $(b,c)$ is a switching pair is in $S(\ca)$.
    Moreover, $S(\ca)$ is $(b,c)\vartriangleleft(a,c)\vartriangleleft(a,b)$ and $T=cba$. 
            \item\label{2geodesics3c} If $(a,b)$ and $(b,c)$ are switching pairs in $S(\ca)$ such that
        $(a,b)\vartriangleleft (b,c)$, then $(a,c)$ is a switching pair in $S(\ca)$. Moreover,        
        $S(\ca)$ is $(a,b) \vartriangleleft(a,c)\vartriangleleft(b,c)$ and $T=cba$.
        \item\label{2geodesics3d}
        If $(a,b)$ and $(b,c)$ are switching pairs in $S(\ca)$ such that $(b,c)\vartriangleleft(a,b)$, then $(a,c)$ is a switching pair in $S(\ca)$. Moreover,  $S(\ca)$ is $(b,c)\vartriangleleft(a,c)\vartriangleleft(a,b)$ and $T=cba$.
            \item\label{2geodesics3e}
        If $(a,c)$ and $(b,c)$ are switching pairs in $S(\ca)$ such that $(a,c)\vartriangleleft (b,c)$,
    then $(a,b)$ is a switching pair is in $S(\ca)$.
    Moreover, $S(\ca)$ is $(a,b)\vartriangleleft(a,c)\vartriangleleft(b,c)$ and $T=cba$.
            \item\label{2geodesics3f}
        If $(a,c)$ and $(b,c)$ are switching pairs in $S(\ca)$ such that $(b,c)\vartriangleleft (a,c)$,
    then $(b,c)$ appears first in $S(\ca)$.
    Moreover, if $(a,b)$ is a switching pair in $S(\ca)$, then $(b,c)\vartriangleleft(a,c)\vartriangleleft(a,b)$ and $T=cba$.
    \end{tabel}
\end{lemma}
\begin{proof}
    Note that $\ca$ is a geodesic connecting $abc$ and $T$. Since there are only two pairs of alternatives that are adjacent in $abc$,
    the first switching pair in $S(\ca)$ can only be either $(a,b)$ or $(b,c)$.

(a) Suppose $(b,c)$ does not occur in $S(\ca)$.
Then $\ca=(abc,bac,bca)$ is possible with $(a,b)\vartriangleleft(a,c)$ and the switching pair $(a,b)$ appears first in $S(\ca)$.
        Suppose $(b,c)$ occurs in $S(\ca)$.
        Then $S(\ca)$ contains three switching pairs. Then Lemma~\ref{2geodesics} (III) provides a dichotomy.
        Since $(a,b)\vartriangleleft(a,c)$, we have a never-bottom geodesic $(a,b)\vartriangleleft(a,c)\vartriangleleft(b,c)$. In particular, $(a,b)$ appears first in $S(\ca)$.

    (b) Since $(a,c)$ appears in $S(\cg)$ and alternatives $a$ and $c$ are not adjacent in $abc$, we need to either first swap $a$ and $b$, or first swap $b$ and $c$, before we can swap $a$ and $c$.
    Now $(a,c)\vartriangleleft (a,b)$ by assumption. Hence, $(b,c)$ is in $S(\ca)$ and $(b,c)\vartriangleleft(a,c)$ in $S(\ca)$.
    Together, we have $(b,c)\vartriangleleft(a,c)\vartriangleleft(a,b)$ and $T=cba$.
    
    (c) Suppose $(a,b)\vartriangleleft (b,c)$ in $S(\ca)$. Then $(a,b)$ is the first switching pair in $S(\ca)$. The linear order immediately after the swap $(a,b)$ in $\ca$ is $bac$.
    Now since $(b,c)$ is a switching pair in $S(\ca)$ and the alternatives $b$ and $c$ are not adjacent in $bac$, we need to swap $a$ and $c$ before we can swap $b$ and $c$.
    Hence, $(a,c)$ is a switching pair in $S(\ca)$ and $(a,b) \vartriangleleft(a,c)\vartriangleleft(b,c)$ in $S(\ca)$. Then $T=cba$

    (d) Suppose $(b,c)\vartriangleleft (a,b)$ in $S(\ca)$. Then $(b,c)$ is the first switching pair in $S(\ca)$. The linear order immediately after the swap $(b,c)$ in $\ca$ is $acb$.
    Now since $(a,b)$ is a switching pair in $S(\ca)$ and the alternatives $a$ and $b$ are not adjacent in $acb$, we need to swap $a$ and $c$ before we can swap $a$ and $b$. Then the statement follows as \ref{2geodesics3c}.
 
(e) Since $(a,c)$ appears in $S(\cg)$ and the alternatives $a$ and $c$ are not adjacent in $abc$, we need to either first swap $a$ and $b$, or first swap $b$ and $c$, before we can swap $a$ and $c$.
    Now $(a,c)\vartriangleleft (b,c)$ by assumption. Hence, $(a,b)$ is in $S(\ca)$ and $(a,b)\vartriangleleft(a,c)$ in $S(\ca)$.
    Together, we have $(a,b)\vartriangleleft(a,c)\vartriangleleft(b,c)$ and $T=cba$.

  (f) Suppose $(a,b)$ does not occur in $S(\ca)$.
Then $\ca=(abc,acb,cab)$ is possible with $(b,c)\vartriangleleft(a,c)$ and the switching pair $(b,c)$ appears first in $S(\ca)$.
        Suppose $(a,b)$ occurs in $S(\ca)$.
        Then $S(\ca)$ contains three switching pairs. Then Lemma~\ref{2geodesics} (III) provides a dichotomy.
        Since $(b,c)\vartriangleleft(a,c)$, we have a never-top geodesic $(b,c)\vartriangleleft(a,c)\vartriangleleft(a,b)$. In particular, $(b,c)$ appears first in $S(\ca)$.
\end{proof}

Now we investigate a special case for a geodesic on a quadruple of alternatives.

\begin{lemma}\label{2geodesics4}
    Let $a,b,c,d$ be distinct alternatives. Let $T\in\cl(\{a,b,c,d\})$ and let $\ca$ be a geodesic on $\cl(\{a,b,c,d\})$ connecting $R=abcd$ and $T$. Then 
    \begin{tabel}
        \item\label{2geodesics4a} If $(a,c)\vartriangleleft (b,c)\vartriangleleft (b,d)$ in $S(\ca)$, then $(a,b)$ and $(a,d)$ are in $S(\ca)$, and         
        $(a,b) \vartriangleleft(a,c)\vartriangleleft (a,d)\vartriangleleft(b,d)$ in $S(\ca)$.
        \item\label{2geodesics4b}
        If $(b,d)\vartriangleleft(a,b)\vartriangleleft (a,c)$ in $S(\ca)$, then $(c,d)$ and $(a,d)$ are in $S(\ca)$, and $(c,d)\vartriangleleft(b,d)\vartriangleleft(a,d)\vartriangleleft (a,c)$ in $S(\ca)$.
    \end{tabel}
\end{lemma}
\begin{proof}
    (a) Note that $R=abcd$, and so we have $R_{\{a,b,c\}}=abc$.
    Since $(a,c)\vartriangleleft (b,c)$ by assumption,
    we have $(a,b)$ is in $S(\ca)$ and $(a,b)\vartriangleleft(a,c)$ in $S(\ca)$ by Lemma~\ref{2geodesics3}\ref{2geodesics3e}.
    Recall $(a,c)\vartriangleleft (b,c)\vartriangleleft (b,d)$ by assumption, we have $(a,b)\vartriangleleft(b,d)$.
    Therefore, by Lemma~\ref{2geodesics3}\ref{2geodesics3c}, $(a,d)$ is in $S(\ca)$ and $(a,b)\vartriangleleft(a,d)\vartriangleleft(b,d)$ .

    Now if $(a,c)\vartriangleleft(a,d)$, then $(a,b) \vartriangleleft(a,c)\vartriangleleft (a,d)\vartriangleleft(b,d)$ as desired.
    Suppose $(a,d)\vartriangleleft(a,c)$ for a contradiction.
    Now $R_{\{a,c,d\}}=acd$. By Lemma~\ref{2geodesics3}\ref{2geodesics3b}, the swap $(c,d)$ is in $S(\ca)$ and $(c,d)\vartriangleleft(a,d)$.
    This gives $(c,d)\vartriangleleft(a,d)\vartriangleleft(a,c)\vartriangleleft(b,c)\vartriangleleft(b,d)$.
    In particular, $S(\cg_{\{b,c,d\}})$ is $(c,d)\vartriangleleft(b,c)\vartriangleleft(b,d)$, which contradicts the dichotomy of Lemma~\ref{2geodesics} for $\ca_{\{b,c,d\}}$.

(b) Since $(b,d)\vartriangleleft(a,b)$ in $\cg_{\{a,b,d\}}$, we have $(b,d)\vartriangleleft(a,d)\vartriangleleft(a,b)$ in $S(\cg_{\{a,b,d\}})$ by Lemma~\ref{2geodesics3}\ref{2geodesics3d}.
Since $(b,d)\vartriangleleft(a,b)\vartriangleleft (a,c)$ in $S(\ca)$ by assumption, we have
$(b,d)\vartriangleleft(a,d)\vartriangleleft(a,b)\vartriangleleft (a,c)$. It remains to show that $(c,d)$ is in $S(\ca)$ and $(c,d)\vartriangleleft(b,d)$.

Now $R=abcd$ and so $R_{\{a,c,d\}}=acd$.
Since $(a,d)$ appears in $S(\cg)$ and $(a,d)\vartriangleleft (a,c)$, we have $(c,d)$ is a switching pair in $S(\ca)$ and $(c,d)\vartriangleleft(a,d)$ in $S(\ca)$ by Lemma~\ref{2geodesics3}\ref{2geodesics3b}.
  Now suppose $(b,d)\vartriangleleft(c,d)$ for a contradiction.
Then since $R_{\{b,c,d\}}=bcd$  and $(b,d)\vartriangleleft (c,d)$, we have $(b,c)$ is in $S(\ca)$ and $(b,c)\vartriangleleft(b,d)$ in $S(\ca)$
by Lemma~\ref{2geodesics3}\ref{2geodesics3e}.
Moreover, $(b,d)\vartriangleleft(a,b)\vartriangleleft (a,c)$ by assumption. We have  $(b,c)\vartriangleleft(b,d)\vartriangleleft(a,b)\vartriangleleft (a,c)$.
    In particular, $S(\cg_{\{a,b,c\}})$ is $(b,c)\vartriangleleft(a,b)\vartriangleleft(a,c)$, which contradicts the dichotomy of Lemma~\ref{2geodesics} for $\ca_{\{a,b,c\}}$.  
\end{proof}

\section{Lemmas for equivalent paths of alike linear orders}

To understand paths of alike linear orders and bring forward further lemmas about them, we need to define the equivalence classes of paths of alike linear orders.

One way to represent a path of alike linear orders is to use wiring diagrams, a tool first introduced by \cite{goodmanj.e.1980} [p.32], which consists of piecewise linear wires. The wires (lines) are horizontal except in the small neighbourhoods around their crossings with other wires. For example, consider the following path of alike linear orders on $\cl(\{a,b,c,d\})$ can be represented by the wiring diagram as shown in Figure~\ref{geo1}.
$$\ca^1=(abcd, bacd,bcad,cbad,cbda,cdba).$$
The sequence of switching pairs from $\ca$ is
$$S(\ca^1):(a,b),(a,c),(b,c),(a,d),(b,d).$$
The intersection of two wires in the wiring diagram indicates a switching pair of alternatives from $S(\ca)$. Each vertical dotted line that intersects all four wires corresponds to a linear order in $\ca$, as it is labelled in downward order. For instance, the leftmost points labelled $a,b,c,d$ in downward order represent the initial linear order $abcd$ in $\ca$.

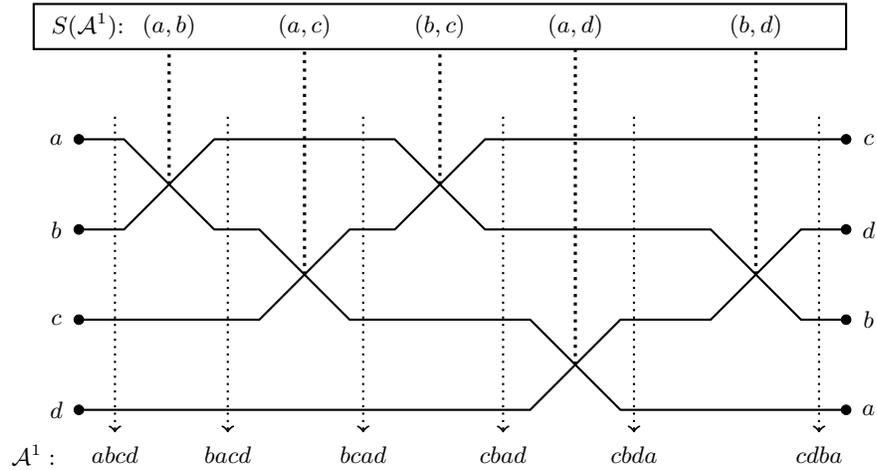
\begin{figure}[H]
    \centering
\begin{tikzpicture}[scale=0.60]
\draw[fill=black] (0,0) circle (3pt);
\draw[fill=black] (0,2) circle (3pt);
\draw[fill=black] (0,4) circle (3pt);
\draw[fill=black] (0,6) circle (3pt);
\draw[fill=black] (17,0) circle (3pt);
\draw[fill=black] (17,2) circle (3pt);
\draw[fill=black] (17,4) circle (3pt);
\draw[fill=black] (17,6) circle (3pt);
 
\node at (-0.5,6) {$a$};
\node at (-0.5,0) {$d$};
\node at (-0.5,2) {$c$};
\node at (-0.5,4) {$b$};
\node at (17.5,0) {$a$};
\node at (17.5,2) {$b$};
\node at (17.5,4) {$d$};
\node at (17.5,6) {$c$};

\draw[thick] (0,0)-- (10,0) --(12,2) --(14,2) --(15,3) --(16,4)--(17,4);
\draw[thick] (0,2) -- (4,2) --(6,4)--(7,4)--(9,6)--
(13,6) --(17,6);
\draw[thick] (0,4) -- (1,4) -- (3,6)--(7,6)--(9,4)--(11,4)--(13,4) --(14,4)--(16,2) --(17,2);
\draw[thick] (0,6)--(1,6)--(3,4)--(4,4)--(6,2) -- (10,2) -- (12,0) -- (17,0);

\draw[very thick,dotted] (2,5)--(2,8);
\draw[very thick,dotted] (5,3)--(5,8);
\draw[very thick,dotted]  (8,5)--(8,8);
\draw[very thick,dotted] (11,1)--(11,8);
\draw[very thick,dotted]  (15,3)--(15,8);

\node at (0.2,8.5) {$S(\ca^1$):};
\node at (2,8.5) {$(a,b)$};
\node at (5,8.5) {$(a,c)$};
\node at (8,8.5) {$(b,c)$};
\node at (11,8.5) {$(a,d)$};
\node at (15,8.5) {$(b,d)$};

\draw[thick](-1,8) rectangle  (17,9);

\node at (-1,-1) {$\ca^1:$};

\draw[thick,dotted][->] (0.8,6.5)--(0.8,-0.5);
\node at (0.8,-1) {$abcd$};

\draw[thick,dotted][->] (3.3,6.5)--(3.3,-0.5);
\node at (3.3,-1) { $bacd$};

\draw[thick,dotted][->] (6.3,6.5)--(6.3,-0.5);
\node at (6.3,-1) {$bcad$};

\draw[thick,dotted][->] (9.4,6.5)--(9.4,-0.5);
\node at (9.4,-1) {$cbad$};

\draw[thick,dotted][->] (12.3,6.5)--(12.3,-0.5);
\node at (12.3,-1) {$cbda$};

\draw[thick,dotted][->] (16.4,6.5)--(16.4,-0.5);
\node at (16.4,-1) {$cdba$};

\end{tikzpicture}
    \caption{A wiring diagram representation for a geodesic on $\cl(\{a,b,c,d\})$ 
    }\label{geo1}
\end{figure}

Observe that the third and fourth switching pairs in $S(\ca^1)$, namely $(b,c)$ and $(a,d)$, are an adjacent disjoint pair of switching pairs in $S(\ca^1)$.
Precisely, given a sequence of switching pairs
    $$S:(a_1,b_1),\ldots, (a_p,b_p),$$
and $i\in[p-1]$, then the pair $(a_i,b_i)$ and $(a_{i+1},b_{i+1})$ is called {\bf an adjacent disjoint pair of switching pairs} in $S$ if $\{a_i,b_i\}\cap\{a_{i+1},b_{i+1}\}=\emptyset$. 

\begin{definition}
Let $R,T$ be two linear orders in $\cl(A)$  and let $\ca$ and $\cb$ be paths of alike linear orders on connecting $R$ and $T$.
    Then we say $S(\ca)$ and $S(\cb)$ are {\bf equivalent} if they differ only by swaps of adjacent disjoint pairs of switching pairs of alternatives. We say that $\ca$ and $\cb$ are {\bf equivalent} if $S(\ca)$ and $S(\cb)$ are equivalent.
\end{definition} 

Note that if $\ca$ and $\cb$ are equivalent paths that differ by $n$ swaps of adjacent disjoint pairs, then there is a sequence $\ca, \ca_1, \ldots, \ca_{n-1}, \cb$ of paths that are all equivalent and such that each path differs by exactly one swap of adjacent disjoint pairs from its predecessor in the sequence. 
Moreover, all equivalent paths of $\ca$ form an equivalence class, and the transitivity relation holds for them. 
Hence, we may frequently assume that any two equivalent paths differ by one swap of adjacent disjoint pairs. 
This is useful when we talk about a portion of a geodesic.

\begin{lemma}\label{equismall}
    Let $\ca=(R_1,\ldots, R_k)$ be a geodesic. Let $\ca'=(R'_1,\ldots, R'_k)$ be a geodesic which is equivalent to $\ca$. Let $n\in[k-1]$. Define the geodesic $\cb=(R'_1,\ldots, R'_n)$. Let $\cb'=(R''_1,\ldots, R''_n)$ be a geodesic which is equivalent to $\cb$. Define the sequence of linear orders $\ca''=(R''_1,\ldots, R''_n, R'_{n+1},\ldots,R'_k)$.
    Then $\ca''$ is a geodesic and it is equivalent to $\ca$.
\end{lemma}
\begin{proof}
    Note that $\cb=(R'_1,\ldots, R'_n)$ and $\cb'=(R''_1,\ldots, R''_n)$ are equivalent, one has $S(\cb)$ and $S(\cb')$ contains the same switching pairs and $R'_n=R''_n$.
    Since $R'_n$ and $R'_{n+1}$ are alike, one has $R''_n$ and $R'_{n+1}$ are alike.
    Hence, $\ca''$ defines a path of alike linear orders.
Since $S(\cb)$ and $S(\cb')$ contains the same switching pairs, 
$S(\ca'')$ and $S(\ca')$ contains the same switching pairs.    
    Now $\ca'$ is a geodesic and so every switching pair in $S(\ca')$ occurs exactly once by Lemma~\ref{wex1.1}\ref{wex1.1a}.
    Hence, every switching pair in $S(\ca')$ occurs exactly once and so $\ca''$ is a geodesic by Lemma~\ref{wex1.1}\ref{wex1.1b}.
    
    Note that $S(\cb)$ and $S(\cb')$ are equivalent. Hence, they differ by swaps of adjacent disjoint pairs of switching pairs. Since the first $n-1$ switching pairs of $S(\ca')$ and $S(\ca'')$ are the switching pairs in $S(\cb)$ and $S(\cb')$, respectively,
    $S(\ca')$ and $S(\ca'')$ also differ by swaps of adjacent disjoint pairs of switching pairs.
    Therefore, $\ca'$ and $\ca''$ are equivalent.
    Since $\ca'$ and $\ca$ are equivalent, we have $\ca''$ and $\ca$ are equivalent.
\end{proof}

\begin{example}
    Recall the path of alike linear orders in Figure~\ref{geo1}:
$$\ca^1=(abcd, bacd,bcad,cbad,cbda,cdba).$$
The sequence of switching pairs is
$$S(\ca^1):(a,b),(a,c),(b,c),(a,d),(b,d).$$
Then $(b,c)$ and $(a,d)$ are an adjacent disjoint pair of switching pairs.
Hence, $\ca^1$ is equivalent to the following path of alike linear orders, as shown in Figure~\ref{geo2},
$$\ca^2=(abcd, bacd,bcad,bcda,cbda,cdba).$$
Actually, this is the only path of alike linear orders that is equivalent to $\ca^1$.
\end{example}

\begin{figure}[H]
    \centering
\begin{tikzpicture}[scale=0.60]
\draw[fill=black] (0,0) circle (3pt);
\draw[fill=black] (0,2) circle (3pt);
\draw[fill=black] (0,4) circle (3pt);
\draw[fill=black] (0,6) circle (3pt);
\draw[fill=black] (17,0) circle (3pt);
\draw[fill=black] (17,2) circle (3pt);
\draw[fill=black] (17,4) circle (3pt);
\draw[fill=black] (17,6) circle (3pt);
 
\node at (-0.5,6) {$a$};
\node at (-0.5,0) {$d$};
\node at (-0.5,2) {$c$};
\node at (-0.5,4) {$b$};
\node at (17.5,0) {$a$};
\node at (17.5,2) {$b$};
\node at (17.5,4) {$d$};
\node at (17.5,6) {$c$};

\draw[thick] (0,0)-- (7,0) --(9,2) --(14,2) --(15,3) --(16,4)--(17,4);
\draw[thick] (0,2) -- (4,2) --(6,4)--(10,4)--(12,6)--
(13,6) --(17,6);
\draw[thick] (0,4) -- (1,4) -- (3,6)--(10,6)--(12,4)--(13,4) --(14,4)--(16,2) --(17,2);
\draw[thick] (0,6)--(1,6)--(3,4)--(4,4)--(6,2) -- (7,2) -- (9,0) -- (17,0);

\draw[very thick,dotted] (2,5)--(2,8);
\draw[very thick,dotted] (5,3)--(5,8);
\draw[very thick,dotted]  (8,1)--(8,8);
\draw[very thick,dotted] (11,5)--(11,8);
\draw[very thick,dotted]  (15,3)--(15,8);

\node at (0.2,8.5) {$S(\ca^2$):};
\node at (2,8.5) {$(a,b)$};
\node at (5,8.5) {$(a,c)$};
\node at (8,8.5) {$(a,d)$};
\node at (11,8.5) {$(b,c)$};
\node at (15,8.5) {$(b,d)$};

\draw[thick](-1,8) rectangle  (17,9);

\node at (-1,-1) {$\ca^2:$};

\draw[thick,dotted][->] (0.8,6.5)--(0.8,-0.5);
\node at (0.8,-1) {$abcd$};

\draw[thick,dotted][->] (3.3,6.5)--(3.3,-0.5);
\node at (3.3,-1) { $bacd$};

\draw[thick,dotted][->] (6.3,6.5)--(6.3,-0.5);
\node at (6.3,-1) {$bcad$};

\draw[thick,dotted][->] (9.4,6.5)--(9.4,-0.5);
\node at (9.4,-1) {$bcda$};

\draw[thick,dotted][->] (12.3,6.5)--(12.3,-0.5);
\node at (12.3,-1) {$cbda$};

\draw[thick,dotted][->] (16.4,6.5)--(16.4,-0.5);
\node at (16.4,-1) {$cdba$};

\end{tikzpicture}
    \caption{A wiring diagram representation for a geodesic on $\cl(\{a,b,c,d\})$ 
    }\label{geo2}
\end{figure}
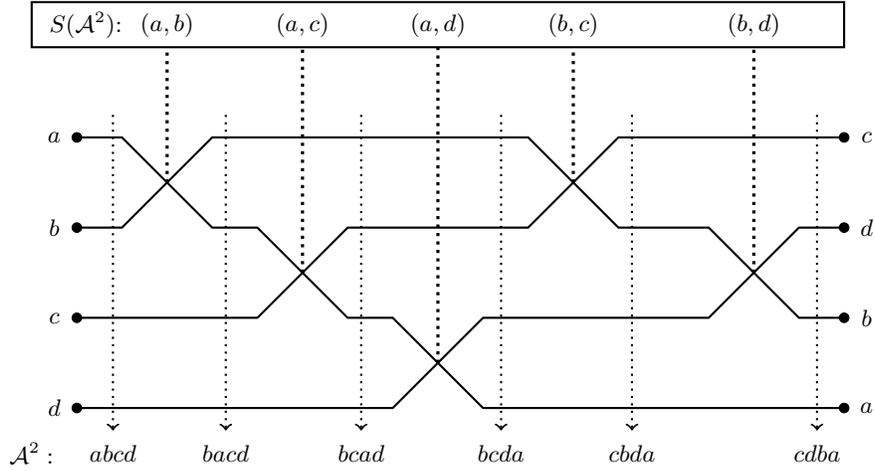

Note that the wiring diagram in Figure~\ref{geo2} would be the same as the one in Figure~\ref{geo1} if we stretch the intersection $(b,c)$ in Figure~\ref{geo2} toward left after passing $(a,d)$, or stretch the intersection $(a,d)$ to the right after passing $(b,c)$. In particular, any three wires of the four wires in it will still intersect in the same manner. For example, $$\ca^1_{\{a,b,c\}}=(abc, bac,bca,cba)$$ after removing duplication $cba$ twice, and $$\ca^2_{\{a,b,c\}}=(abc, bac,bca,cba)$$ after removing duplication $bca$ and $cba$. We emphasise that $\ca^1_{\{a,b,c\}}=\ca^2_{\{a,b,c\}}$.
Moreover, $$S(\ca^2_{\{a,b,c\}}):(a,b),(a,c),(b,c).$$ 
Now $S(\ca^2_{\{a,b,c\}})$ is the sequence of switching pairs in Figure~\ref{geo2} by removing all switching pairs that contain at least one alternative outside the set $\{a,b,c\}$. This immediately gives the following lemma.

\begin{lemma}\label{conex0}
    Let $\ca$ be a path of alike linear orders on $\cl(A)$ and $a,b,c\in A$ be distinct alternatives.
    Then the number of occurrences of the switching pair $(a,b)$ in $S(\ca)$ is the same as the number of occurrences of the switching pair $(a,b)$ in $S(\ca_{\{a,b,c\}})$.
\end{lemma}
\begin{proof}
Since $S(\ca_{\{a,b,c\}})$ is obtained from $S(\ca)$, the number of occurrences of $(a,b)$ in $S(\ca_{\{a,b,c\}})$ is at least as many as the number of times it appears in $S(\ca)$. Moreover, $S(\ca_{\{a,b,c\}})$ is obtained from $S(\ca)$ by removing all switching pairs that contain at least one alternative outside the set $\{a,b,c\}$. 
If $\ca=(R_1,\ldots,R_k)$, then $\ca_{\{a,b,c\}}$ is obtained from $(R_1)_{\{a,b,c\}}$,\ldots, $(R_k)_{\{a,b,c\}}$ by removing duplicate equal linear orders which are next to each other. This happens precisely if the corresponding linear orders in $\ca$ have a switching pair for which at least one of the alternatives is not in the set $\{a,b,c\}$.
Since $\{a,b\} \subset \{a,b,c\}$, all $(a,b)$ pairs that are in $S(\ca)$ will remain in $S(\ca_{\{a,b,c\}})$.
\end{proof}

The above example of the equivalent paths $\ca^1$ and $\ca^2$ provides insight into determining equivalent paths.

\begin{lemma}\label{conex1}
Let $R,T$ be two linear orders in $\cl(A)$  and let $\ca$ and $\cb$ be paths of alike linear orders connecting $R$ and $T$. Then $\ca$ and $\cb$ are equivalent if and only if $\ca_{\{a,b,c\}}=\cb_{\{a,b,c\}}$ for every distinct $a,b,c\in A$. 
\end{lemma}
\begin{proof}
$(\Rightarrow)$
Suppose $\ca$ and $\cb$ are equivalent paths on $\cl(A)$.  Then $S(\ca)$ and $S(\cb)$ differ by $n$ swaps of adjacent disjoint pairs of switching pairs. 
We claim that it suffices to prove the implication if $S(\ca)$ and $S(\cb)$ differ by only one swap of such a pair.
This can be done by induction on the number of swaps. Suppose $S(\ca)$ and $S(\cb)$ differ by $n-1$ swaps of adjacent disjoint pairs of switching pairs implies that $\ca_{\{a,b,c\}}=\cb_{\{a,b,c\}}$ for every distinct $a,b,c\in A$.  
Let $\cc$ be an equivalent path of $\ca$ and $\cb$ such that $S(\cc)$ and $S(\ca)$ differ by one swap of adjacent disjoint pair, and $S(\cc)$ and $S(\cb)$ differ by $n-1$ swaps.
Then $\cc_{\{a,b,c\}}=\ca_{\{a,b,c\}}$ for every distinct $a,b,c\in A$ by the base case, and $\cc_{\{a,b,c\}}=\cb_{\{a,b,c\}}$ by the induction hypothesis. That implies $\ca_{\{a,b,c\}}=\cb_{\{a,b,c\}}$.

Now let $S(\ca)$ and $S(\cb)$ differ by only one swap of adjacent disjoint pairs of switching pairs and their representation are
$$S(\ca):(a_1,b_1),\ldots, (a_i,b_i),(a_{i+1},b_{i+1}),\ldots,(a_p,b_p),$$
    $$S(\cb):(c_1,d_1),\ldots,(c_i,d_i),(c_{i+1},d_{i+1}),\ldots, (c_p,d_p),$$
such that $(a_j, b_j)=(c_j,d_j)$ for all $j\notin\{i,i+1\}$, 
and $(a_i,b_i)=(c_{i+1},d_{i+1})$ and $(a_{i+1},b_{i+1})=(c_{i},d_{i})$. Let $a,b,c\in A$ be any distinct alternatives.
We will show that $S(\ca_{\{a,b,c\}})=S(\cb_{\{a,b,c\}})$, which implies $\ca_{\{a,b,c\}}=\cb_{\{a,b,c\}}$.

{\bf Case 1.} Suppose $\{a,b,c\}\cap\{a_i, b_i, a_{i+1},b_{i+1}\}=\emptyset$. Then since $(a_j, b_j)=(c_j,d_j)$ for all $j\notin\{i,i+1\}$, we have $S(\ca_{\{a,b,c\}})=S(\cb_{\{a,b,c\}})$.

{\bf Case 2.} Suppose $|\{a,b,c\}\cap\{a_i, b_i, a_{i+1},b_{i+1}\}|=1$. Then since $S(\ca_{\{a,b,c\}})$ is obtained by removing all switching pairs in $S(\ca)$ that contain at least one alternative outside the set $\{a,b,c\}$, all switching pairs in $S(\ca_{\{a,b,c\}})$ are from switching pairs $(a_j,b_j)$ for some $j\notin\{i,i+1\}$. The same reasoning applies to the switching pairs in $S(\cb_{\{a,b,c\}})$ and so $S(\ca_{\{a,b,c\}})=S(\cb_{\{a,b,c\}})$.

{\bf Case 3.} Suppose $|\{a,b,c\}\cap\{a_i, b_i, a_{i+1},b_{i+1}\}|=2$. There are two subcases.

{\bf Case 3.1.} Suppose $\{a_i, b_i\}\subset\{a,b,c\}$ or $\{a_{i+1}, b_{i+1}\}\subset\{a,b,c\}$.
Without loss of generality, suppose $(a_i,b_i)=(a,b)$. Then since $\{a_{i+1},b_{i+1}\}$ is disjoint from $\{a_i,b_i\}$, it contains at least one alternative outside the set $\{a,b,c\}$ and so it will be removed from $S(\ca)$ to obtain $S(\ca_{\{a,b,c\}})$. Likewise, $(c_{i},d_{i})$ will be removed from $S(\cb)$ to obtain $S(\cb_{\{a,b,c\}})$. 
Since $(a_{i+1},b_{i+1})=(c_{i},d_{i})$, we have $S(\ca_{\{a,b,c\}})=S(\cb_{\{a,b,c\}})$.

{\bf Case 3.2.} Suppose $|\{a_i, b_i\}\cap\{a,b,c\}|=1$.
Then both $(a_{i},b_{i})$ and $(a_{i+1},b_{i+1})$ contain one alternative outside the set $\{a,b,c\}$ and so they will be removed from $S(\ca)$ to obtain $S(\ca_{\{a,b,c\}})$. Likewise, both $(c_{i},d_{i})$ and $(c_{i+1},d_{i+1})$ will be removed from $S(\cb)$ to obtain $S(\cb_{\{a,b,c\}})$. 
Hence $S(\ca_{\{a,b,c\}})=S(\cb_{\{a,b,c\}})$.

{\bf Case 4.} Suppose $\{a,b,c\}\subseteq\{a_i, b_i, a_{i+1},b_{i+1}\}$. This is the same as Case 3.1.

Therefore, $\ca_{\{a,b,c\}}=\cb_{\{a,b,c\}}$ for every distinct $a,b,c\in A$.\\


$(\Leftarrow)$  Suppose $\ca_{\{a,b,c\}}=\cb_{\{a,b,c\}}$ for every triple of distinct alternatives $a,b,c\in A$. Then $S(\ca_{\{a,b,c\}})=S(\cb_{\{a,b,c\}})$. 
Let $u,v\in A$ different, we shall show that the number of switching pairs $(u,v)$ in $S(\ca)$ is the same as the one in $S(\cb)$. Choose $w\in A\setminus\{u,v\}$. Then by Lemma 11, the number of switching pairs $(u,v)$ in $S(\ca)$ is the same as the one in $S(\ca_{\{u,v,w\}})$. By assumption this is the same as in $S(\cb_{\{u,v,w\}})$. Using again Lemma 11, this is the same as in $S(\cb)$. 
Hence, the multiset of switching pairs from $S(\ca)$ is the same as the one from $S(\cb)$, and so $S(\ca)$ and $S(\cb)$ have the same length.
%

Let $p$ be the length of $S(\ca)$. For all $i\in\{0,1\ldots,p\}$, let $P(i)$ be the hypothesis that the path $\ca$ is equivalent to a path $\ca^*$ such that the first $i$ switching pairs of $\ca^*$ are equal to the first $i$ switching pairs of $\cb$. Then the lemma follows once we show that $P(p)$ is valid. The proof is by induction. Clearly $P(0)$ is valid by choosing $\ca^*=\ca$. Let $i\in [p]$ and suppose $P(i-1)$ is true. Then $\ca$ is equivalent to a path $\ca^*$ such that the first $i-1$ switching pairs of $\ca^*$ are equal to the first $i-1$ switching pairs of $\cb$. 
By the implication $(\Rightarrow)$, it follows that $\ca^*_{\{a,b,c\}}=\ca_{\{a,b,c\}}=\cb_{\{a,b,c\}}$ for all distinct $a,b,c\in A$.
Let the sequences of switching pairs for $\ca^*$ and $\cb$ be
    $$S(\ca^*):(a_1,b_1),\ldots, (a_p,b_p),$$
    $$S(\cb):(c_1,d_1),\ldots, (c_p,d_p).$$
If $\{a_i,b_i\}= \{c_i,d_i\}$, then obviously $P(i)$ is true. Suppose $\{a_i,b_i\}\ne \{c_i,d_i\}$. There are two cases.

    

    {\bf Case 1.} Suppose $\{a_i,b_i\}\cap \{c_i,d_i\}=\emptyset$. We will continuously check the switching pairs of alternatives that are after $(a_{i},b_{i})$ in $S(\ca^*)$. Note that if $\{a_{q},b_{q}\}\cap\{c_i,d_i\}=\emptyset$ for all $q>i$, then the number of occurrences of $(c_i,d_i)$ in $S(\ca^*)$ would not be the same as it in $S(\cb)$, which contradicts the earlier fact that the multisets of switching pairs from $S(\ca^*)$ and $S(\cb)$ the same.  
    Let
    $$j=min\{r\in\{i+1,\ldots,p\}:\{a_r,b_r\}\cap\{c_i,d_i\}\ne\emptyset\}.$$

    {\bf Case 1.1.} Suppose $|\{a_{j},b_{j}\}\cap\{c_i,d_i\}|=1$, write $\{a_{j},b_{j}\}\cup\{c_i,d_i\}=\{x,y,z\}$. 
    Note that the contribution of $S(\ca^*)$ from the first $i-1$ pairs to $S(\ca^*_{\{x,y,z\}})$ is the same as the contribution of $S(\cb)$ from the first $i-1$ pairs to $S(\cb_{\{x,y,z\}})$ by assumption. However, the contribution of $S(\ca^*)$ between the $i$-th pair and the $j$-th pair (including the ends) to $S(\ca^*_{\{x,y,z\}})$ is starting with $(a_j,b_j)$, but the contribution of $S(\cb)$ between the $i$-th pair and the $j$-th pair to $S(\cb_{\{x,y,z\}})$ is starting with $(c_i,d_i)$.
    Hence, $S(\ca^*_{\{x,y,z\}})\ne S(\cb_{\{x,y,z\}})$. This is a contradiction, and so this case is not possible.
    
    {\bf Case 1.2.} Suppose $\{a_{j},b_{j}\}=\{c_i,d_i\}$. Then since  $\{c_i,d_i\}\cap \{a_r,b_r\}=\emptyset$ for all $i\leq r<j$, we have $\{a_j,b_j\}\cap \{a_r,b_r\}=\emptyset$. Hence, we can move $(a_j,b_j)$ to the position immediately before $(a_i,b_i)$ in $S(\ca^*)$. 
    This means that $\ca^*$ is equivalent to a path $\ca^{ }$ such that the first $i$ switching pairs of $\ca^{ }$ are equal to the first $i$ switching pairs of $\cb$. Therefore, $P(i)$ is valid in this case.
    
    {\bf Case 2.} Suppose $|\{a_i,b_i\}\cap \{c_i,d_i\}|=1$. This is similar to Case 1.1 and this case is not possible.

    Thus, the implication $(\Leftarrow)$ follows by induction.
\end{proof}

\section{Lemmas for enriching peak-pit Condorcet domains}

We will show that for every two linear orders $R$ and $T$ in a peak-pit Condorcet domain, the union of the domain and the set of linear orders from a geodesic connecting $R$ and $T$ will remain a peak-pit Condorcet domain. In particular, we will show that this is true for peak-pit Condorcet domains restricted to a triple or a quadruple of alternatives.


\begin{lemma}\label{wex1.2} Suppose $\cd\subseteq \cl(A)$ is a peak-pit Condorcet domain. 
\begin{tabel}
    \item\label{wex1.2a} Let $a,b,c\in A$ be distinct alternatives. If the linear orders $abc$ and $cab$ are both in $\cd_{\{a,b,c\}}$, then $\cd_{\{a,b,c\}}\cup\{acb\}$ is a peak-pit Condorcet domain. 
If the linear orders $abc$ and $bca$ are both in $\cd_{\{a,b,c\}}$, then $\cd_{\{a,b,c\}}\cup\{bac\}$ is a peak-pit Condorcet domain. 
\item\label{wex1.2b} 
Let $a,b,c\in A$ be distinct and $R,T\in \cd_{\{a,b,c\}}$. Suppose $R$ and $T$  differ by at most two swaps of adjacent alternatives. Then there is exactly one geodesic $\ca$ on $\cl(\{a,b,c\})$ connecting them. In addition, $\cd_{\{a,b,c\}}\cup K(\ca)$ is a peak-pit Condorcet domain. 
\item\label{wex1.2c} Let $a,b,c\in A$ be distinct and $R,T\in \cd_{\{a,b,c\}}$. Suppose $R$ and $T$  differ by three swaps of adjacent alternatives. Then there are exactly two geodesics on $\cl(\{a,b,c\})$ connecting them, and for at least one of these two geodesics, $\ca$, we have  $\cd_{\{a,b,c\}}\cup K(\ca)$ is a peak-pit Condorcet domain.
\item\label{wex1.2d} For all distinct $a,b,c\in A$ and $R,T\in \cd_{\{a,b,c\}}$, there is a geodesic $\ca$ connecting $R$ and $T$ such that $\cd_{\{a,b,c\}}\cup\ K(\ca)$ is a peak-pit Condorcet domain.
\end{tabel}
\end{lemma}
\begin{proof}
(a) Suppose $abc$ and $cab$ are linear orders in $\cd_{\{a,b,c\}}$. 
Since $\cd$ is a peak-pit Condorcet domain, the restricted domain $\cd_{\{a,b,c\}}$ is also a peak-pit Condorcet domain. Hence, $\cd_{\{a,b,c\}}$ is a never-top domain or a never-bottom domain.

{\bf Case 1.} Suppose $\cd_{\{a,b,c\}}$ is a never-top domain. Then there is an $x\in \{a,b,c\}$ such that $\cd_{\{a,b,c\}}$ satisfies the never-condition $xN_{\{a,b,c\}}1$. Since $abc\in \cd_{\{a,b,c\}}$, it follows that $x\ne a$.
Then $\cd_{\{a,b,c\}}\cup\{acb\}$ also satisfies the never-condition $xN_{\{a,b,c\}}1$, and so $\cd_{\{a,b,c\}}\cup\{acb\}$ is a peak-pit Condorcet domain. 

{\bf Case 2.} Suppose $\cd_{\{a,b,c\}}$ is a never-bottom domain. Then there is an $x\in \{a,b,c\}$ such that $\cd_{\{a,b,c\}}$ satisfies the never-condition $xN_{\{a,b,c\}}3$. Since $cab\in \cd_{\{a,b,c\}}$, it follows that $x\ne b$.
Then $\cd_{\{a,b,c\}}\cup\{acb\}$ also satisfies the never-condition $xN_{\{a,b,c\}}3$, and so $\cd_{\{a,b,c\}}\cup\{acb\}$ is a peak-pit Condorcet domain. 

The second statement follows from the first and the substitution $a\mapsto b,b\mapsto c,c\mapsto a$.\\

    (b) Let $a,b,c$ be distinct alternatives in $A$. Note that there are only six linear orders in $\cl({\{a,b,c\}})$, as shown in Figure~\ref{g3paths}.
$$\cl({\{a,b,c\}})=\{abc,bac,bca,cba,cab,acb\}.$$ 

Without loss of generality, suppose $R=abc \in \cd_{\{a,b,c\}}$.
We consider three cases.

{\bf Case 1.} Suppose $T=R=abc$. Then $(R)$ is the only geodesic on $\cl(\{a,b,c\})$ connecting $R$ and $T$.

{\bf Case 2.} Suppose $T\in\{bac,acb\}$. Note that $R$ and $T$ differ by one swap of adjacent alternatives. Then $(R,T)$ is the only geodesic on $\cl(\{a,b,c\})$ connecting $R$ and $T$.

{\bf Case 3.} Suppose $T\in\{bca,cab\}$. Note that $R$ and $T$ differ by two swaps of adjacent alternatives. 
Suppose $T=bca$. Then $(abc,bac,bca)$ is a path connecting $R$ and $T$ and obviously it is the only geodesic on $\cl(\{a,b,c\})$ connecting $R$ and $T$. 
In addition, $\cd_{\{a,b,c\}}\cup\{bac\}$ is a peak-pit Condorcet domain by Statement~\ref{wex1.2a}.
The case for $T=cab$ is similar.\\

(c) Without loss of generality, suppose $R=abc \in \cd_{\{a,b,c\}}$. Since $R$ and $T$ differ by three swaps of adjacent alternatives, we have $T=cba$. 
Inspecting Figure~\ref{g3paths}, we see that there are exactly 2 geodesics on $\cl(\{a,b,c\})$ connecting $abc$ and $cba$, which are
$\ca_1=(abc,acb,cab,cba)$ and $\ca_2=(abc,bac,bca,cba)$.
Since $abc, cba\in \cd_{\{a,b,c\}}$ and $\cd$ is a peak-pit Condorcet domain, the only peak-pit conditions that $\cd_{\{a,b,c\}}$ satisfies are $bN_{\{a,b,c\}}1$ or $bN_{\{a,b,c\}}3$ by Lemma~\ref{wex0}\ref{wex0a}.
Suppose $\cd_{\{a,b,c\}}$ satisfies $bN_{\{a,b,c\}}1$. Then $\ca_1=(abc,acb,cab,cba)$ is a geodesic connecting $abc$ and $cba$ such that $\cd_{\{a,b,c\}}\cup \{acb,cab\}$ is a peak-pit Condorcet domain because it still satisfies $bN_{\{a,b,c\}}1$ but no longer $bN_{\{a,b,c\}}3$.
Similarly, if $\cd_{\{a,b,c\}}$ satisfies $bN_{\{a,b,c\}}3$, then $\ca_2=(abc,bac,bca,cba)$ is a geodesic connecting $abc$ and $cba$ such that $\cd_{\{a,b,c\}}\cup \{bac,bca\}$ is a peak-pit Condorcet domain because it still satisfies $bN_{\{a,b,c\}}3$. It does not satisfy, however, the condition $bN_{\{a,b,c\}}1$.\\

(d) This follows immediately from Statements~\ref{wex1.2b} and \ref{wex1.2c}.
\end{proof}

The second statement of the above lemma states that if two linear orders for three alternatives differ by at most two swaps of adjacent alternatives, then there is a unique geodesic connecting them. 
However, the third statement states that if they differ by exactly three swaps of adjacent alternatives, there could be potentially two geodesics connecting them, each satisfying different peak-pit conditions.
Hence, the construction of a geodesic connecting $R$ and $T$ to ensure the directly connectedness of $\cd$ will not rely on using different restricted geodesics. 
The problem of inconstructibility arises when trying to construct a geodesic connecting $R$ and $T$ from restricted geodesics connecting $R_B$ and $T_B$ but such restricted geodesics could satisfy different peak-pit conditions.

\begin{lemma}\label{wex1.3} Suppose $\cd\subseteq \cl(A)$ is a peak-pit Condorcet domain. 
\begin{tabel}
\item\label{wex1.3a} Let $a,b,c\in A$ be distinct alternatives, and $R=\cdots abc \cdots$ and $T=\cdots cab \cdots$ be linear orders in $\cd$ with $R_{A\setminus\{a,b\}}=T_{A\setminus\{a,b\}}$. Suppose $S=\cdots acb \cdots$ is a linear order in $\cl(A)$ with $S_{A\setminus\{a,b\}}=R_{A\setminus\{a,b\}}$.
Then $\cd\cup\{S\}$ is a peak-pit Condorcet domain.
\item\label{wex1.3b} Let $a,b,c,d\in A$ be distinct alternatives. For then let $R=\cdots ab \cdots cd\cdots$ and $T=\cdots ba \cdots dc\cdots$ be linear orders in $\cd$ with  $R_{A\setminus\{b,d\}}=T_{A\setminus\{b,d\}}$. Suppose $S^1=\cdots ba \cdots cd\cdots$ and $S^2=\cdots ab \cdots dc\cdots$ are linear orders in $\cl(A)$ with $R_{A\setminus\{b,d\}}=T_{A\setminus\{b,d\}}=S^1_{A\setminus\{b,d\}}=S^2_{A\setminus\{b,d\}}$. Then $\cd\cup\{S_1, S_2\}$ is also a peak-pit Condorcet domain and satisfies the same set of never-conditions as $\cd$.
In particular, $\cd$, $\cd\cup\{S_1\}$ and $\cd\cup\{S_2\}$ are three peak-pit Condorcet domains which satisfy the same never-conditions.
\item\label{wex1.3c} Let $\ca$ and $\cb$ be equivalent paths such that $\cd\cup K(\ca)$ is a peak-pit Condorcet domain. Then $\cd\cup K(\cb)$ is also a peak-pit Condorcet domain and satisfies the same set of never-conditions as $\cd\cup K(\ca)$.
\end{tabel}
\end{lemma}

\begin{proof}

(a) Let $R=\cdots abc \cdots$ and $T=\cdots cab \cdots$ be linear orders in $\cd$ with $R_{A\setminus\{a,b\}}=T_{A\setminus\{a,b\}}$. Let $S=\cdots acb \cdots$ be such that $S_{A\setminus\{a,b\}}=R_{A\setminus\{a,b\}}$. 
Let $p,q,r$ be a triple of distinct alternatives. We shall show that $(\cd\cup \{S\})_{\{p,q,r\}}$ is a peak-pit Condorcet domain.

{\bf Case 1.} Suppose $\{p,q,r\}\cap \{a,b,c\}=\emptyset$.  Since $S_{A\setminus\{a,b\}}=R_{A\setminus\{a,b\}}$ by assumption, we have $S_{\{p,q,r\}}=(S_{A\setminus\{a,b\}})_{\{p,q,r\}}=(R_{A\setminus\{a,b\}})_{\{p,q,r\}}=R_{\{p,q,r\}}$. Hence, $(\cd\cup \{S\})_{\{p,q,r\}}=\cd_{\{p,q,r\}}$ is a peak-pit Condorcet domain.

{\bf Case 2.} Suppose $|\{p,q,r\}\cap \{a,b,c\}|=1$. Since $S_{A\setminus\{a,b\}}=R_{A\setminus\{a,b\}}$ by assumption and $a,b,c$ are next to each other in $R$ and $S$, we have $S_{A\setminus\{a,c\}}=R_{A\setminus\{a,c\}}$ and 
$S_{A\setminus\{b,c\}}=R_{A\setminus\{b,c\}}$.
Let $\{x,y\}=\{a,b,c\}\setminus \{p,q,r\}$. Then
$S_{\{p,q,r\}}=(S_{A\setminus\{x,y\}})_{\{p,q,r\}}=(R_{A\setminus\{x,y\}})_{\{p,q,r\}}=R_{\{p,q,r\}}$.
Hence, $(\cd\cup \{S\})_{\{p,q,r\}}=\cd_{\{p,q,r\}}$ is a peak-pit Condorcet domain.

{\bf Case 3.} Suppose $|\{p,q,r\}\cap \{a,b,c\}|=2$. Let $\{x\}=\{p,q,r\}\setminus \{a,b,c\}$. 
There are three subcases for $\{p,q,r\}$. If $\{p,q,r\}=\{a,b,x\}$, then since ${S}_{\{a,b,x\}}={R}_{\{a,b,x\}}$, we have $(\cd\cup \{S\})_{\{a,b,x\}}=\cd_{\{a,b,x\}}$.
If $\{p,q,r\}=\{a,c,x\}$, then since ${S}_{\{a,c,x\}}={R}_{\{a,c,x\}}$, we have $(\cd\cup \{S\})_{\{a,c,x\}}=\cd_{\{a,c,x\}}$.
If $\{p,q,r\}=\{b,c,x\}$, then since ${S}_{\{b,c,x\}}={T}_{\{b,c,x\}}$, we have $(\cd\cup \{S\})_{\{b,c,x\}}=\cd_{\{b,c,x\}}$.
Thus, $(\cd\cup \{S\})_{\{p,q,r\}}=\cd_{\{p,q,r\}}$ is a peak-pit Condorcet domain.

{\bf Case 4.} Suppose $\{p,q,r\}=\{a,b,c\}$. This was shown earlier in Lemma~\ref{wex1.2}\ref{wex1.2a} that  $(\cd\cup \{S\})_{\{a,b,c\}}$ is a peak-pit Condorcet domain. 

Hence, $\cd\cup \{S\}$ is a peak-pit Condorcet domain.\\

(b) Let $R=\cdots ab \cdots cd\cdots$ and $T=\cdots ba \cdots dc\cdots$ be linear orders in $\cd$. Let $S^1=\cdots ba \cdots cd\cdots$ and $S^2=\cdots ab \cdots dc\cdots$ be such that $R_{A\setminus\{b,d\}}=T_{A\setminus\{b,d\}}=S^1_{A\setminus\{b,d\}}=S^2_{A\setminus\{b,d\}}$.
Let $p,q,r$ be a triple of distinct alternatives. We shall show that $S^i_{\{p,q,r\}}$ is either equal to $R_{\{p,q,r\}}$ or $T_{\{p,q,r\}}$ for both $i=1$ and $i=2$. That implies both $(\cd\cup \{S^1\})_{\{p,q,r\}}$ and $(\cd\cup \{S^2\})_{\{p,q,r\}}$ are peak-pit Condorcet domains satisfying the same set of never-conditions as $\cd_{\{p,q,r\}}$. 

{\bf Case 1.} Suppose $\{p,q,r\}\cap \{a,b,c,d\}=\emptyset$.  Since $R_{A\setminus\{b,d\}}=S^i_{A\setminus\{b,d\}}$ by assumption, $S^i_{\{p,q,r\}}=(S^i_{A\setminus\{b,d\}})_{\{p,q,r\}}=(R_{A\setminus\{b,d\}})_{\{p,q,r\}}=R_{\{p,q,r\}}$ for both $i=1$ and $i=2$.
Hence, both $(\cd\cup \{S^1\})_{\{p,q,r\}}$ and $(\cd\cup \{S^2\})_{\{p,q,r\}}$ are peak-pit Condorcet domains satisfying the same set of never-conditions as $\cd_{\{p,q,r\}}$.

{\bf Case 2.} Suppose $|\{p,q,r\}\cap \{a,b,c,d\}|=1$.  Let $\{x,y,z\}=\{a,b,c,d\}\setminus \{p,q,r\}$. Since $R_{A\setminus\{b,d\}}=S^i_{A\setminus\{b,d\}}$ and each pair of $\{a,b\}$ and $\{c,d\}$ is next to each other in $R$ and $S^i$,  we have $S^i_{A\setminus\{a,c\}}=R_{A\setminus\{a,c\}}$, $S^i_{A\setminus\{a,d\}}=R_{A\setminus\{a,d\}}$ and 
$S^i_{A\setminus\{b,c\}}=R_{A\setminus\{b,c\}}$ for both $i=1$ and $i=2$.
Thus, we have $ S^i_{A\setminus\{x,y,z\}}=R_{A\setminus\{x,y,z\}}$ for both $i=1$ and $i=2$.
Then $S^i_{\{p,q,r\}}=(S^i_{A\setminus\{x,y,z\}})_{\{p,q,r\}}=(R_{A\setminus\{x,y,z\}})_{\{p,q,r\}}=R_{\{p,q,r\}}$ for both $i=1$ and $i=2$.
Hence, both $(\cd\cup \{S^1\})_{\{p,q,r\}}$ and $(\cd\cup \{S^2\})_{\{p,q,r\}}$ are peak-pit Condorcet domains satisfying the same set of never-conditions as $\cd_{\{p,q,r\}}$.

{\bf Case 3.} Suppose $|\{p,q,r\}\cap \{a,b,c,d\}|=2$. Let $\{x\}=\{p,q,r\}\setminus \{a,b,c,d\}$. If $\{p,q,r\}=\{a,b,x\}$, 
then $S^1_{\{p,q,r\}}=S^1_{\{a,b,x\}}={T}_{\{a,b,x\}}={T}_{\{p,q,r\}}$ and $S^2_{\{p,q,r\}}=S^2_{\{a,b,x\}}={R}_{\{a,b,x\}}={R}_{\{p,q,r\}}$.
If $\{p,q,r\}=\{c,d,x\}$, then $S^1_{\{p,q,r\}}=S^1_{\{c,d,x\}}={R}_{\{c,d,x\}}={R}_{\{p,q,r\}}$ and $S^2_{\{p,q,r\}}=S^2_{\{c,d,x\}}={T}_{\{c,d,x\}}={T}_{\{p,q,r\}}$. The other subcases here are covered by the assumption that $R_{A\setminus\{b,d\}}=S^1_{A\setminus\{b,d\}}=S^2_{A\setminus\{b,d\}}$. Precisely, we write $\{a,b\}\setminus\{p,q,r\}=\{m\}$ and $\{c,d\}\setminus\{p,q,r\}=\{n\}$.
Then $R_{A\setminus\{m,n\}}=S^i_{A\setminus\{m,n\}}$ and so
$S^i_{\{p,q,r\}}=(S^i_{A\setminus\{m,n\}})_{\{p,q,r\}}=(R_{A\setminus\{m,n\}})_{\{p,q,r\}}=R_{\{p,q,r\}}$ for both $i=1$ and $i=2$.
Hence, both $(\cd\cup \{S^1\})_{\{p,q,r\}}$ and $(\cd\cup \{S^2\})_{\{p,q,r\}}$ are peak-pit Condorcet domains satisfying the same set of never-conditions as $\cd_{\{p,q,r\}}$.

{\bf Case 4.} Suppose $\{p,q,r\} \subset \{a,b,c,d\}$. 
Suppose $\{a,b\}\subset \{p,q,r\}$, and we write $\{p,q,r\}=\{a,b,x\}$.
Then $S^1_{\{p,q,r\}}=S^1_{\{a,b,x\}}=T_{\{a,b,x\}}=T_{\{p,q,r\}}$ and
$S^2_{\{p,q,r\}}=S^2_{\{a,b,x\}}=R_{\{a,b,x\}}=R_{\{p,q,r\}}$.
Alternatively, if $\{c,d\}\subset \{p,q,r\}$, write $\{p,q,r\}=\{c,d,y\}$.
Then $S^1_{\{p,q,r\}}=S^1_{\{c,d,y\}}=R_{\{c,d,y\}}=R_{\{c,d,r\}}$ and
$S^2_{\{p,q,r\}}=S^2_{\{c,d,y\}}=T_{\{c,d,y\}}=T_{\{p,q,r\}}$.
Hence, both $(\cd\cup \{S^1\})_{\{p,q,r\}}$ and $(\cd\cup \{S^2\})_{\{p,q,r\}}$ are peak-pit Condorcet domains satisfying the same set of never-conditions as $\cd_{\{p,q,r\}}$.

Therefore, since $S^i_{\{p,q,r\}}$ is either equal to $R_{\{p,q,r\}}$ or $T_{\{p,q,r\}}$ for both $i=1$ and $i=2$, both $\cd\cup \{S^1\}$ and $\cd\cup \{S^2\}$ are peak-pit Condorcet domains satisfying the same set of never-conditions as $\cd$.\\

(c) Let $\ca=(R_1,\ldots, R_k)$ and $\cb=(T_1,\ldots, T_k)$ be two equivalent paths such that $\cd\cup\{R_1,\ldots,R_k\}$ is a peak-pit Condorcet domain. 
Under transitivity of equivalent paths, we may assume that $S(\ca)$ and $S(\cb)$ only differ by one swap of adjacent disjoint switching pairs, say $(a,b)$ and $(c,d)$. Then all linear orders in the same position in $\ca$ and $\cb$ are the same except for one position, say $R_i$ and $T_i$ for some $i\in[k]$. 
Write $R_{i-1}=T_{i-1}=\cdots ab \cdots cd\cdots$ and $R_{i+1}=T_{i+1}=\cdots ba \cdots dc\cdots$.
Without loss of generality, suppose $R_i=\cdots ba \cdots cd\cdots$ and $T_i=\cdots ab \cdots dc\cdots$.
By assumption, $\cd\cup\{R_1,\ldots, R_{i-1}, R_i,R_{i+1}\ldots, R_k\}$ is a peak-pit Condorcet domain.
Then $\cd^*=\cd\cup\{R_1,\ldots, R_{i-1},R_{i+1}\ldots, R_k\}$ is also a peak-pit Condorcet domain.
Using Statement~\ref{wex1.3b} with $\cd^*$ instead of $\cd$, it follows that $\cd^*$, $\cd^*\cup\{R_i\}$ and $\cd^*\cup\{T_i\}$ are three peak-pit Condorcet domains which satisfy the same set of never-conditions. 
Hence, $\cd\cup\{T_1,\ldots, T_{k}\}=\cd^*\cup\{T_i\}$ is a peak-pit Condorcet domain and it satisfies the same set of never-conditions as $\cd^*\cup\{R_i\}=\cd\cup\{R_1,\ldots,R_k\}$.
\end{proof}

\section{Lemmas for swap closure on a geodesic}
Now we investigate some properties of equivalent geodesics, specifically how one can be transferred to another when they differ by more than one swap of adjacent disjoint pairs.
Consider the two equivalent geodesics shown in Figure~\ref{2longswap2}, where we swap the disjoint switching pairs \((h,w)\) and \((d,e)\). However, since $(d,e)$ and $(h,w)$ are not adjacent, we only just swap \((d,e)\) with \((h,w)\) but also all switching pairs between them, namely \((h,a)\), \((h,b)\), and \((a,b)\).
In order for the swapping to be possible, we need \((d,e)\) to be disjoint from all of these pairs. Moreover, note that while the switching pairs \((h,a)\) and  \((h,b)\) are not disjoint from \((h,w)\), the switching pair \((a,b)\) is not disjoint from these switching pairs that are not disjoint from \((h,w)\).

\begin{figure}[H]
\begin{subfigure}[b]{0.5\textwidth}
     \centering
\begin{tikzpicture}[scale=0.28]

\draw[fill=black] (0,0) circle (3pt);
\draw[fill=black] (0,2) circle (3pt);
\draw[fill=black] (0,4) circle (3pt);
\draw[fill=black] (0,6) circle (3pt);
\draw[fill=black] (0,8) circle (3pt);
\draw[fill=black] (0,10) circle (3pt);
\draw[fill=black] (19,0) circle (3pt);
\draw[fill=black] (19,2) circle (3pt);
\draw[fill=black] (19,4) circle (3pt);
\draw[fill=black] (19,6) circle (3pt);
\draw[fill=black] (19,8) circle (3pt);
\draw[fill=black] (19,10) circle (3pt);
 
\node at (-0.5,0) {$e$};
\node at (-0.5,2) {$d$};
\node at (-0.5,4) {$w$};
\node at (-0.5,6) {$b$};
\node at (-0.5,8) {$a$};
\node at (-0.5,10) {$h$};
\node at (19.5,0) {$d$};
\node at (19.5,2) {$h$};
\node at (19.5,4) {$e$};
\node at (19.5,6) {$w$};
\node at (19.5,8) {$a$};
\node at (19.5,10) {$b$};

\node at (4,1) {$(d,e)$};
\node at (9,7) {$(h,a)$};
\node at (3,7) {$(a,b)$};
\node at (16,5) {$(h,w)$};
\node at (6,9) {$(h,b)$};

\draw[thick] (0,0) -- (1,0)--(3,2)--(16,2)--(18,4) --(19,4);
\draw[thick] (0,2)--(1,2)--(3,0)--(6,0)--(19,0);
\draw[ultra thick] (0,4)--(13,4)--(15,6)--(16,6)--(18,6)--(19,6);
\draw[thick] (0,6)--(4,6)--(6,8)--(7,8)--(9,10)--(19,10);
\draw[thick] ((0,8)--(4,8)--(6,6)--(10,6)--(12,8)--(16,8)--(18,8)--(19,8);
\draw[thick] (0,10)--(7,10)--(9,8)--(10,8)--(12,6)--(13,6)--(15,4)--(16,4)--(18,2)--(19,2);

\end{tikzpicture}
\caption{$(d,e)$ appears before $(h,w)$ in $\cg$.}\label{}
 \end{subfigure}
 \begin{subfigure}[b]{0.5\textwidth}
     \centering
\begin{tikzpicture}[scale=0.28]

\draw[fill=black] (0,0) circle (3pt);
\draw[fill=black] (0,2) circle (3pt);
\draw[fill=black] (0,4) circle (3pt);
\draw[fill=black] (0,6) circle (3pt);
\draw[fill=black] (0,8) circle (3pt);
\draw[fill=black] (0,10) circle (3pt);
\draw[fill=black] (19,0) circle (3pt);
\draw[fill=black] (19,2) circle (3pt);
\draw[fill=black] (19,4) circle (3pt);
\draw[fill=black] (19,6) circle (3pt);
\draw[fill=black] (19,8) circle (3pt);
\draw[fill=black] (19,10) circle (3pt);
 
\node at (-0.5,0) {$e$};
\node at (-0.5,2) {$d$};
\node at (-0.5,4) {$w$};
\node at (-0.5,6) {$b$};
\node at (-0.5,8) {$a$};
\node at (-0.5,10) {$h$};
\node at (19.5,0) {$d$};
\node at (19.5,2) {$h$};
\node at (19.5,4) {$e$};
\node at (19.5,6) {$w$};
\node at (19.5,8) {$a$};
\node at (19.5,10) {$b$};

\node at (16,1) {$(d,e)$};
\node at (10,7) {$(h,a)$};
\node at (4,7) {$(a,b)$};
\node at (13,5) {$(h,w)$};
\node at (7,9) {$(h,b)$};

\draw[thick] (0,0) -- (1,0)--(13,0)--(15,2) --(16,2)--(18,4)--(19,4);
\draw[thick] (0,2)--(13,2)--(15,0)--(19,0);
\draw[ultra thick] (0,4)--(10,4)--(12,6)--(13,6)--(15,6)--(19,6);
\draw[thick] (0,6)--(1,6)--(3,8)--(4,8)--(6,10)--(19,10);
\draw[thick] ((0,8)--(1,8)--(3,6)--(7,6)--(9,8)--(13,8)--(15,8)--(19,8);
\draw[thick] (0,10)--(4,10)--(6,8)--(7,8)--(9,6)--(10,6)--(12,4)--(16,4)--(18,2)--(19,2);

\end{tikzpicture}
\caption{$(d,e)$ appears after $(h,w)$ in $\cg'$.}\label{}
 \end{subfigure}
     \caption{A case that switching pairs can be swapped for equivalent geodesics. } \label{2longswap2}
\end{figure}
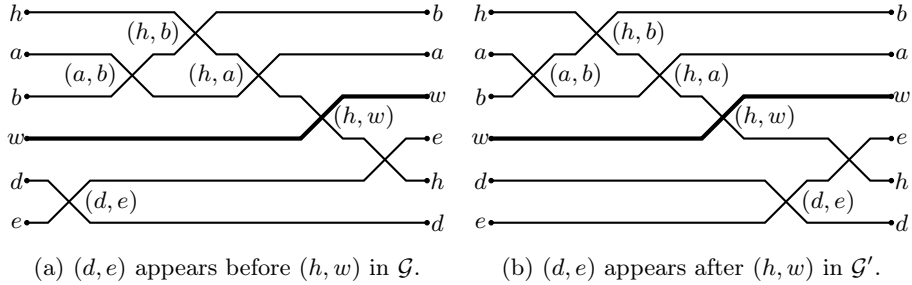 

Here, both $d$ and $e$ are not swapped with $w$. We need a lemma to generalise the idea that $(d,e)$ would be disjoint with each of those switching pairs $(h,a)$, $(h,b)$ and $(a,b)$. 

Let $\cg$ be a geodesic on $\cl(A)$ and $T(\cg)$ be {\bf the set of switching pairs in $S(\cg)$}.  Let $w\in A$ be an alternative that swaps in $\cg$ with at least one alternative in $A$. Let the {\bf swap set of $w$} on $\cg$ be defined by
$$H=\{h\in A\setminus\{w\}: (h,w)\in T(\cg)\}.$$
Define the {\bf swap closure of $w$} on $\cg$ by
\begin{align*}
    K_1=\{(a,b)\in T(\cg):\;&\text{there exist}\; h\in H\; \text{and}\; u\in A\setminus{\{h\}}\;\text{such that}\;\\ &(h,u)\in T(\cg)\;\text{and} \\
    & (a,b)\trianglelefteq (h,u)\trianglelefteq (h,w)\;\text{and}\;  \\
    & \{a,b\}\cap\{h,u\}\ne\emptyset\}.
\end{align*}

The swap closure of an alternative $w$ includes all switching pairs that either are not disjoint from any switching pair involving $w$, or are not disjoint from any switching pairs that are themselves not disjoint from a switching pair involving $w$.
For example, the switching pairs $(h,b), (h,a)$ and $(h,w)$, as in Figure~\ref{2longswap2}(a), are not disjoint from $(h,w)$. They are all in the swap closure of $w$.
Moreover, the switching pair $(a,b)$, as in Figure~\ref{2longswap2}(a), are not disjoint from $(h,a)$ that is itself not disjoint from $(h,w)$. So $(a,b)$ is also in the swap closure of $w$.

\begin{lemma}\label{conpropl1}
Let $\cg$ be a geodesic on $\cl(A)$ and $T(\cg)$ be the set of switching pairs in $S(\cg)$.  Let $w\in A$ be an alternative that swaps in $\cg$ with at least one alternative in $A$. Let $K_1$ be the swap closure of $w$ on $\cg$ and $K_2=T(\cg)\setminus K_1.$
Then each switching pair in $K_1$ is disjoint from each switching pair in $K_2$ which appears before it in $S(\cg)$.
\end{lemma}

\begin{proof}
Let $H$ be the swap set of $w$. Note that $H$ is not empty by the choice of $w$, Also, if $h\in H$, then $(h,w)\in K_1$ by choosing $u=w$ in the definition of $K_1$, and so $K_1$ is not empty.
Now if $K_2$ is empty, then we are done. 
Suppose $K_2\ne \emptyset$.
    Let $(a_1,b_1)\in K_1$ and $(a_2,b_2)\in K_2$ be such that $(a_2,b_2)\vartriangleleft(a_1,b_1)$ . 
We shall show that $\{a_1,b_1\}\cap\{a_2,b_2\}=\emptyset$. Suppose not.
Since the switching pair $(a_1, b_1)$ is the same as the switching pair $(b_1, a_1)$ and similarly for $(a_2, b_2)$, we can assume without loss of generality that $a_1=a_2$. 
Now since $(a_1,b_1)\in K_1$, there exist $h\in H$ and $u\in  A\setminus{\{h\}}$ such that $(h,u)\in T(\cg)$, $(a_1,b_1)\trianglelefteq (h,u)\trianglelefteq (h,w)$, and $\{a_1,b_1\}\cap \{h,u\}\ne\emptyset$.
So we have the ordering
\begin{equation}\label{4order}    (a_2,b_2)\vartriangleleft(a_1,b_1)\trianglelefteq(h,u)\trianglelefteq(h,w),
\end{equation}
for the four switching pairs.

We next show that $b_1=u$. If $a_1\in\{h,u\}$, then $\{a_2,b_2\}\cap \{h,u\}\supset \{a_1\}\cap \{h,u\}
\ne\emptyset$.
Moreover, $(a_2,b_2)\vartriangleleft(a_1,b_1)\trianglelefteq(h,u)$, so $(a_2,b_2)\in K_1$. This is a contradiction. Therefore, $a_1\notin \{h,u\}$. Hence, $b_1\in \{h,u\}$.
We claim that the ordering
\begin{equation}\label{4order2}
(a_2,b_2)\vartriangleleft(a_2,h)\vartriangleleft(h,w)
\end{equation}
is not possible. For a contradiction suppose one has the ordering (\ref{4order2}). We know that $(a_1,b_1)\trianglelefteq(h,w)$. 
Choose $\tilde{h}=h$ and $\tilde{u}=a_2$. 
Then $(\tilde{h},\tilde{u})=(a_2,h)\vartriangleleft(h,w)=(\tilde{h},w)$. 
Furthermore, $(a_2,b_2)\vartriangleleft(a_2,h)=(\tilde{h},\tilde{u})$ and $\{a_2,b_2\}\cap\{\tilde{h},\tilde{u}\}\supset\{a_2\}\ne\emptyset$.
Hence, $(a_2,b_2)\in K_1$, which is a contradiction.
Therefore, the ordering (\ref{4order2}) is not possible. In particular, (\ref{4order}) implies that $b_1\ne h$. Since $b_1\in\{h,u\}$, we conclude that $b_1=u$.

We claim that the ordering
\begin{equation}\label{4order4}
    (a_2,b_2)\vartriangleleft(a_2,w)
\end{equation}
is not possible.
If one has the ordering (\ref{4order4}), then $a_2\ne w$ and hence $a_2\in H$.
Choose $\tilde{h}=a_2$ and $\tilde{u}=b_2$. 
Then $(\tilde{h},\tilde{u})=(a_2,b_2)\vartriangleleft(a_2,w)=(\tilde{h},w)$. 
Furthermore, $(a_2,b_2)=(\tilde{h},\tilde{u})$ and obviously $\{a_2,b_2\}\cap\{\tilde{h},\tilde{u}\}\ne\emptyset$.
Hence, $(a_2,b_2)\in K_1$, which is a contradiction.
Therefore, the ordering (\ref{4order4}) is not possible.
By the ordering (\ref{4order}) and using $(a_1,b_1)=(a_2,u)$, it follows that $u\ne w$.

Therefore, $(a_1,b_1)=(a_2,u)$ and $u\ne w$. 
We proved the ordering
\begin{equation}\label{4orders}
(a_2,b_2)\vartriangleleft(a_2,u)\vartriangleleft(h,u)\vartriangleleft(h,w),
\end{equation}
in $S(\cg)$. Moreover, all four are different, where we recall $a_2=a_1\notin\{h,u\}$. 

We next show that the ordering
\begin{equation}\label{4order3}
(a_2,b_2)\vartriangleleft(a_2,u)\vartriangleleft(u,w)
\end{equation}
is not possible. Suppose one has the ordering (\ref{4order3}). We know that $u\ne w$. Now $(u,w)\in T(\cg)$, hence $u\in H$. 
Choose $\tilde{h}=h$ and $\tilde{u}=a_2$. 
Then $(\tilde{h},\tilde{u})=(a_2,u)\vartriangleleft(u,w)=(\tilde{h},w)$. 
Furthermore, $(a_2,b_2)\vartriangleleft(a_2,u)=(\tilde{h},\tilde{u})$. Also, $\{a_2,b_2\}\cap\{\tilde{h},\tilde{u}\}\supset\{a_2\}\ne\emptyset$.
So $(a_2,b_2)\in K_1$, which is a contradiction.
Hence, the ordering (\ref{4order3}) is not possible.

Now let $L$ be the linear order in $\cg$ immediately after the swap $(a_2,b_2)$. 
Obviously, $L_{\{a_2,h,u\}}\in \{hua_2,a_2uh,a_2hu, uha_2,ha_2u, ua_2h\}$.
We shall show that each of the six possible linear orders for $L_{\{a_2,h,u\}}$ gives a contradiction.

{\bf Case 1.} Suppose $L_{\{a_2,h,u\}}=hua_2$. 
Since $(a_2,u)\vartriangleleft(h,u)$ by (\ref{4orders}).
we have $(a_2,h)$ is a switching pair in $S(\cg)$ and
$(a_2,u)\vartriangleleft(a_2,h)\vartriangleleft(h,u)$ by Lemma~\ref{2geodesics3}\ref{2geodesics3d}.
Hence, $(a_2,b_2)\vartriangleleft(a_2,u)\vartriangleleft(a_2,h)\vartriangleleft(h,u)\vartriangleleft(h,w)$ by (\ref{4orders}). This is not possible by (\ref{4order2}).

{\bf Case 2.} Suppose $L_{\{a_2,h,u\}}=a_2uh$. This is the same as Case 1.

{\bf Case 3.} Suppose $L_{\{a_2,h,u\}}=a_2hu$.
Since $(a_2,u)\vartriangleleft(h,u)$ by (\ref{4orders}).
we have $(a_2,h)$ is a switching pair in $S(\cg)$ and
$(a_2,h)\vartriangleleft(a_2,u)\vartriangleleft(h,u)$ by Lemma~\ref{2geodesics3}\ref{2geodesics3e}.
Now if $(a_2,h)\vartriangleleft(a_2,b_2)$, then $a_2$ and $u$ are not adjacent in $L_{\{a_2,h,u\}}$ and it is not possible that $(a_2,u)$ swaps after $(a_2,b_2)$.
Hence, $(a_2,b_2)\vartriangleleft (a_2,h)$.
This gives the order $(a_2,b_2)\vartriangleleft(a_2,h)\vartriangleleft(a_2,u)\vartriangleleft(h,u)\vartriangleleft(h,w)$, which contradicts (\ref{4order2}).

{\bf Case 4.} Suppose $L_{\{a_2,h,u\}}=uha_2$. This is the same as Case 3.

{\bf Case 5.} Suppose $L_{\{a_2,h,u\}}=ha_2u$. There are four possibilities for $L_{\{a_2,h,u,w\}}$.

{\bf Case 5.1.} Suppose $L_{\{a_2,h,u,w\}}=wha_2u$. 
Then $L_{\{h,u,w\}}=whu$. Consider the switching pair $(u,w)$. Arguing as before, it must be that $(h,u)\vartriangleleft (u,w) \vartriangleleft(h,w)$.
This gives the order $(a_2,b_2)\vartriangleleft(a_2,u)\vartriangleleft(h,u)\vartriangleleft(u,w)\vartriangleleft(h,w)$, which contradicts
(\ref{4order3}).

{\bf Case 5.2.} Suppose $L^*_{\{a_2,h,u,w\}}=hwa_2u$. Then $L_{\{h,u,w\}}=hwu$. Since $(h,u)\vartriangleleft(h,w)$ by (\ref{4orders}), we need to swap $u$ needs to swap with $w$ before swapping $h$ and $u$. Moreover, since $L^*_{\{a_2,u,w\}}=wa_2u$, before we can swap $u$ and $w$, we need to swap $a_2$ and $w$ first or swap $a_2$ and $u$ first.
If $(a_2, w)\vartriangleleft (u,w)$, we have the impossible order (\ref{4order4}). If $(a_2, u)\vartriangleleft (u, w)$, we have the impossible order (\ref{4order3}).


{\bf Case 5.3.} Suppose $L_{\{a_2,h,u,w\}}=ha_2wu$. Then $L_{\{a_2,h,w\}}=ha_2w$, and so one has to swap first $h$ and $a_2$, or $a_2$ and $w$ before swapping $h$ and $w$. In the former case, we have the impossible order (\ref{4order2}). In the latter case, we have (\ref{4order4}).

{\bf Case 5.4.} Suppose $L_{\{a_2,h,u,w\}}=ha_2uw$. Then $L_{\{a_2,h,w\}}=ha_2w$ as Case 5.3, which is impossible.

{\bf Case 6.} Suppose $L_{\{a_2,h,u\}}=ua_2h$. In the arguments in Case 5, we considered whether or not alternatives are adjacent. Since all alternatives are now in reversed order, the argument is the same.

Therefore, $\{a_1,b_1\}\cap \{a_2,b_2\}=\emptyset$ and hence each switching pair in $K_1$ is disjoint from each switching pair in $K_2$ which  appears before it in $S(\cg)$. 
\end{proof}

When a geodesic meets a peak-pit Condorcet domain with a certain property, its sequence of switching pairs behaves in a certain way. 
Recall that given a Condorcet domain $\cd$, we use $N_p(\cd)$ to denote the set of peak-pit conditions satisfied by $\cd$.

\begin{lemma}\label{conpropl2}
    Let $\cd$ be a peak-pit Condorcet domain on $\{a,b,v,z\}$ with  $R=abvz\in \cd$ and $T=vzba\in\cd$ and such that $N_p(\cd_{\{a,b,z\}})=\{bN_{\{a,b,z\}}1\}$.
    Then $N_p(\cd_{\{a,b,v\}})=\{bN_{\{a,b,v\}}1\}$.
    Consequently, if $\cg=(G_1,\ldots, G_k)$ is a geodesic on $\cl(\{a,b,v\})$ connecting $R_{\{a,b,v\}}$ and $T_{\{a,b,v\}}$ such that $\cd_{\{a,b,v\}}\cup\{G_1,\ldots, G_k\}$ is a peak-pit Condorcet domain, then the sequence of switching pairs $S(\cg_{\{a,b,v\}})$ is $(b,v)\vartriangleleft (a,v)\vartriangleleft (a,b)$.
\end{lemma}
\begin{proof}
    Since $R=abvz\in \cd$ and $T=vzba\in\cd$, we have $abv,vba\in\cd_{\{a,b,v\}}$.
    Hence, $\emptyset\ne N_p(\cd_{\{a,b,v\}})\subseteq \{bN_{\{a,b,v\}}1$, $bN_{\{a,b,v\}}3\}$ due to Lemma~\ref{wex0}\ref{wex0a}. Suppose $N_p(\cd_{\{a,b,v\}})\ne\{bN_{\{a,b,v\}}1\}$ for a contradiction. We have two cases for $N_p(\cd_{\{a,b,v\}})$.

{\bf Case 1.} Suppose $N_p(\cd_{\{a,b,v\}})=\{bN_{\{a,b,v\}}3\}$.
Since $abv, vba\in\cd_{\{a,b,v\}}$ and $\cd_{\{a,b,v\}}$ does not satisfy $bN_{\{a,b,v\}}1$, one deduces that $bav\in \cd_{\{a,b,v\}}$ or $bva\in \cd_{\{a,b,v\}}$. Hence, there is a linear order $S\in\cd$ such that $S_{\{a,b,v\}}$ is either $bav$ or $bva$. Nevertheless, note that $S$ always ranks $b$ before $a$.
On the other hand, since  $N_p(\cd_{\{a,b,z\}})=\{bN_{\{a,b,z\}}1\}$ is given, we have $\cd_{\{a,b,z\}}\subseteq\{abz,azb,zab,zba\}$.
Hence, the only linear order in $\cd_{\{a,b,z\}}$ that ranks $b$ before $a$ is $zba$ and so $S_{\{a,b,z\}}=zba$. 
Recall that $S_{\{a,b,v\}}=bav$ or $S_{\{d,e,z\}}=bva$. 
Hence, $S$ is either $zbav$ or $zbva$. 
Therefore, in both cases, $zbv=S_{\{b,v,z\}}\in \cd_{\{b,v,z\}}$.
In particular, $\{bvz,vzb,zbv\}=\{R,T,S\}_{\{b,v,z\}}\subseteq \cd_{\{b,v,z\}}$.
Thus, $\cd_{\{b,v,z\}}$ does not satisfy any never-condition, which contradicts that $\cd$ is a peak-pit Condorcet domain. This case is not possible.

{\bf Case 2.} Suppose $N_p(\cd_{\{a,b,v\}})=\{bN_{\{a,b,v\}}3,bN_{\{a,b,v\}}1\}$.
Then $\cd_{\{d,e,w\}}$ = $\{abv,vba\}$ due to Lemma~\ref{wex0}\ref{wex0b}. However, since $N_p(\cd_{\{a,b,z\}})=\{bN_{\{a,b,z\}}1\}$ and $abz, zba\in\cd_{\{a,b,z\}}$ 
one deduces that $azb\in \cd_{\{a,b,z\}}$ or $zab\in \cd_{\{a,b,z\}}$. 
Hence, there is a linear order $S\in\cd$ such that $S_{\{a,b,z\}}$ is either $azb$ or $zab$.
In particular, $S_{\{b,z\}}=zb$ and $S_{\{a,b\}}=ab$.
Since the only linear order in $\cd_{\{a,b,v\}}$ that ranks $a$ before $b$ is $abv$, we have $S_{\{a,b,v\}}=abv$. Recall that $S_{\{b,z\}}=zb$.
Hence, $S$ is either $azbv$ or $zabv$, and so $zbv=S_{\{b,v,z\}}\in \cd_{\{b,v,z\}}$. 
In particular, $\{bvz,vzb,zbv\}=\{R,T,S\}_{\{b,v,z\}}\subseteq \cd_{\{b,v,z\}}$.
Thus, $\cd_{\{b,v,z\}}$ does not satisfy any never-condition, which contradicts that $\cd$ is a peak-pit Condorcet domain. This case is also not possible.

Therefore, $N_p(\cd_{\{a,b,v\}})=\{bN_{\{a,b,v\}}1\}$.

    Moreover, let $\cg=(G_1,\ldots, G_k)$ be a geodesic on $\cl(\{a,b,v\})$ connecting $R_{\{a,b,v\}}$ and $T_{\{a,b,v\}}$ such that $\cd_{\{a,b,v\}}\cup\{G_1,\ldots, G_k\}$ is a peak-pit Condorcet domain. 
    Since $R_{\{a,b,v\}}=abv$ and $T_{\{a,b,v\}}=vba$, the set  $\{G_1,\ldots, G_k\}_{\{a,b,v\}}$ satisfies a unique never-condition which is either $bN_{\{a,b,v\}}3$ or $bN_{\{a,b,v\}}1$ due to Lemma~\ref{2geodesics}. Since $N_p(\cd_{\{a,b,v\}})=\{bN_{\{a,b,v\}}1\}$,  the set  $\{G_1,\ldots, G_k\}_{\{a,b,v\}}$ satisfies a unique never-condition $bN_{\{a,b,v\}}1$.
    Therefore, the sequence of switching pairs $S(\cg_{\{a,b,v\}})$ is $(b,v)\vartriangleleft (a,v)\vartriangleleft(a,b)$ due to Lemma~\ref{2geodesics}\ref{2geodesicsb}.
\end{proof}

The last two lemmas above can be used for the following lemma. Furthermore, Lemmas~\ref{conpropl1}, \ref{conpropl2} and \ref{conpropl3} help to prove Lemma~\ref{conex2.1} later. 

\begin{lemma}\label{conpropl3}
Let $\cd\subseteq\cl(A)$ be a peak-pit Condorcet domain. Let $L,T\in\cd$ be distinct linear orders and $z\in A$.
Let $\cg=(G_1,\ldots, G_k)$ be a geodesic on $\cl(A\setminus\{z\})$ connecting $L_{A\setminus\{z\}}$ and $T_{A\setminus\{z\}}$ such that $\cd_{A\setminus\{z\}}\cup\{G_1,\ldots, G_k\}$ is a peak-pit Condorcet domain.
Let $T(\cg)$ be the set of switching pairs in $S(\cg)$.  Let $w\in A\setminus\{z\}$ be an alternative that swaps in $\cg$ with at least one alternative in $A$. Let $K_1$ be the swap closure of $w$.
Let $(a,b)\in K_1$ be such that $L_{\{a,b,z\}}=abz$, $L_{\{w,z\}}=wz$
and $T_{\{a,b,w,z\}}=wzba$. Then $N_p(\cd_{\{a,b,z\}})\ne\{bN_{\{a,b,z\}}1\}$.
\end{lemma}
\begin{proof}
Let $H$ be the swap set of $w$.
    Suppose $N_p(\cd_{\{a,b,z\}})=\{bN_{\{a,b,z\}}1\}$ for a contradiction. We consider three cases for $R_{\{a,b,w\}}$. We will show that each case leads to $(a,b)\notin K_1$.

{\bf Case 1.} Suppose $L_{\{a,b,w\}}=abw$. Since $L_{\{w,z\}}=wz$, we have $L_{\{a,b,w\}}=abwz$. 
Hence, the sequence of switching pairs in $S(\cg_{\{a,b,w\}})$ is $(b,w)$, $(a,w)$, $(a,b)$ due to Lemma~\ref{conpropl2}.
Since $(a,b)\in K_1$, there exists $h\in H\setminus\{a,b\}$ and $u\in A\setminus\{h\}$ such that $(a,b)\trianglelefteq(h,u)\trianglelefteq(h,w)$, and $\{a,b\}\cap\{h,u\}\ne\emptyset$. Since $h\in H\setminus\{a,b\}$, we have $u\in \{a,b\}$.

{\bf Case 1.1.} Suppose $L_{\{a,b,h\}}=hab$. Then $L_{\{a,b,h,w\}}=habw$. In particular, $L_{\{a,h,w\}}=haw$.
Note that all switching pairs $(a,h)$, $(h,w)$ and $(a,w)$ are elements of $T(\cg)$.
Since $(h,a)=(h,u)\trianglelefteq (h,w)$, it follows from Lemma~\ref{2geodesics3}\ref{2geodesics3c} that $(h,a)\trianglelefteq (a,w)$. However, $(a,w)\trianglelefteq(a,b)$. So $(h,u)=(h,a)\trianglelefteq(a,b)\trianglelefteq(h,u)$.
Hence, $h\in\{h,u\}=\{a,b\}$. This is a contradiction. The argument is the same for the case when $u=b$.

{\bf Case 1.2.} Suppose $L_{\{a,b,h\}}=ahb$. Then $L_{\{a,b,h,w\}}=ahbw$.
If $u=a$, then $L_{\{a,h,w\}}=ahw$ and all switching pairs $(a,h)$,$(h,w)$ and $(a,w)$ are elements of $T(\cg)$.
Now $(h,a)\trianglelefteq(h,w)$, so by Lemma~\ref{2geodesics3}\ref{2geodesics3c}, one deduces that $(h,a)\trianglelefteq(a,w)\trianglelefteq(a,b\trianglelefteq(h,u)=(h,a)$. Hence, $h\in\{h,u\}=\{a,b\}$. This is a contradiction.
Hence, $u=b$. Then $L_{\{b,h,w\}}=hbw$ and this is similar to Case 1.1.

{\bf Case 1.3.} Suppose $L_{\{a,b,h\}}=abh$. 
Then $L_{\{a,b,h,w\}}=abhw$ or $L_{\{a,b,h,w\}}=abwh$.
If $L_{\{a,b,h,w\}}=abhw$, then $L_{\{a,h,w\}}=ahw$, which is not possible by the argument in Case 1.2.
Hence, $L_{\{a,b,h,w\}}=abwh$. Then $L_{\{h,w\}}=wh$. Since $(h,w)\in T(\cg)$, we have $T_{\{h,w\}}=hw$. Since $T_{\{a,b,w\}}=wba$, one has $T_{\{a,b,h\}}=hba$. 
Also, because  $T_{\{h,w\}}=hw$ and one has $T_{\{a,b,h\}}=hwba$. 
Since $T_{\{a,b,w,z\}}=wzba$ and $L_{\{w,z\}}=wz$, one has $T_{\{a,b,h,z\}}=hzba$ and $L_{\{a,b,z\}}=hzba$.
It follows that  the sequence of switching pairs in $S(\cg_{\{a,b,h\}})$ is $(h,b)$, $(h,a)$, $(a,b)$ due to Lemma~\ref{conpropl2}. Hence, whether $u=a$ or $u=b$, the switching pair $(h,u)$ always occurs before $(a,b)$, a contradiction. 

{\bf Case 2.} Suppose $L_{\{a,b,w\}}=awb$. Since  $T_{\{a,b,w\}}=wba$, the switching pair $(a,w)\in T(\cg)$ appears before $(a,b)$, and $(b,w)\notin T(\cg)$.
Now since $(a,b)\in K_1$, there exists $h\in H\setminus\{a,b\}$ and $u\in A\setminus\{h\}$ such that $(a,b)\trianglelefteq(h,u)\trianglelefteq(h,w)$, and $\{a,b\}\cap\{h,u\}\ne\emptyset$. We consider three cases.

{\bf Case 2.1.} Suppose $L_{\{a,b,h\}}=hab$. Then $L_{\{a,b,h,w\}}=hawb$ and so $L_{\{a,h,w\}}=haw$. This is the same as Case 1.1.

{\bf Case 2.2.} Suppose $L_{\{a,b,h\}}=ahb$. Then $L_{\{a,b,h,w\}}=ahwb$ or $L_{\{a,b,h,w\}}=awhb$.
Suppose $L_{\{a,b,h,w\}}=ahwb$. If $u=a$, then $L_{\{a,h,w\}}=ahw$ which is not possible by the argument in Case 1.2.
If $u=b$, then $L_{\{b,h,w\}}=hwb$. Recall that $(b,w)\notin T(\cg)$, but $(h,b)=(h,u)\in T(\cg)$ and $(h,w)\in T(\cg)$. Hence, $(h,w)\trianglelefteq(h,b)$, a contradiction.

Therefore, $L_{\{a,b,h,w\}}=awhb$. If $u=a$, then $L_{\{a,b,h\}}=ahb$.
Now $(h,w)\in T(\cg)$, so $T_{\{h,w\}}=hw$ and $T_{\{a,b,h,w\}}=hwba$. Then $L_{\{b,h\}}=hb$ and $T_{\{b,h\}}=hb$, so $(b,h)\notin T(\cg)$.
Furthermore, $(h,a)=(h,u)\in T(\cg)$ and $(a,b)\in T(\cg)$.
Since $L_{\{a,b,h\}}=ahb$, if follows that $(h,u)=(h,a)\trianglelefteq (a,b)$, a contradiction. The argument for $u=b$ is similar.

{\bf Case 2.3.} Suppose $L_{\{a,b,h\}}=abh$. This is the same as Case 1.3.

{\bf Case 3.} Suppose $L_{\{a,b,w\}}=wab$. 
Since $T_{\{a,b,w\}}=wba$, the switching pairs $(a,w)$ and $(b,w)$ are not in $T(\cg)$. Since $(a,b)\in K_1$, there exists $h\in H$ and $u\in A\setminus\{h,z\}$ such that $(a,b)\trianglelefteq(h,u)\trianglelefteq(h,w)$, and $\{a,b\}\cap\{h,u\}\ne\emptyset$. 
We consider three cases. The argument is similar to Case 2.

{\bf Case 3.1.} Suppose $L_{\{a,b,h\}}=hab$. 
Then $L_{\{a,b,h,w\}}=hwab$ or $L_{\{a,b,h,w\}}=whba$. 
Suppose $L_{\{a,b,h,w\}}=hwab$. If $u=a$, then $L_{\{a,h,w\}}=hwa$. Since $(w,a)\notin T(\cg)$, we have $(h,w)\trianglelefteq(h,u)=(h,a)$, a contradiction. If $u=b$, then $L_{\{b,h,w\}}=hwb$ and this is the same to the case if $u=a$.

{\bf Case 3.2.} Suppose $L_{\{a,b,h\}}=ahb$. 
Then $L_{\{a,b,h,w\}}=wahb$. If $u=a$ and $(h,b)\notin T(\cg)$, then $(h,a)=(h,u)\trianglelefteq(a,b)$, a contradiction. Hence, $(h,b)\in T(\cg)$. Then $T_{\{a,b,h\}}=bha$.
Since $T_{\{a,b,w\}}=wba$, we have $T_{\{a,b,h,w\}}=wbha$ and so $(h,w)\notin T(\cg)$, a contradiction.
If $u=b$, then $L_{\{b,h,w\}}=whb$, this is the same as Case 3.1.

{\bf Case 3.3.} Suppose $L_{\{a,b,h\}}=abh$. 
Then $L_{\{a,b,h,w\}}=wabh$. Then $L_{\{a,h,w\}}=wah$ and $L_{\{b,h,w\}}=wbh$.
This is the same as Case 3.2.
Therefore, $L_{\{a,b,h,w\}}=whba$. If $u=a$, then $L_{\{a,h,w\}}=wha$. Recall that $(a,w)\notin T(\cg)$ but both $(h,a)=(h,u)$ and $(h,w)$ are in $T(\cg)$. Since $(h,a)=(h,u)\trianglelefteq(h,w)$. After swapping $(h,a)$, one obtains the linear order $wah$ from $L_{\{a,h,w\}}=wha$. However, the alternatives $h$ and $w$ are not adjacent in $wah$ and so they cannot be swapped, a contradiction. If $u=b$, then $L_{\{b,h,w\}}=whb$ and this is the same as the case where $u=a$.

Therefore, it is not possible for such a switching pair $(a,b)$ to exist in $K_1$ and so $N_p(\cd_{\{a,b,z\}})\ne\{bN_{\{a,b,z\}}1\}$.
\end{proof}



\section{Lemmas for Peak-Pit conditions with additional alternatives}

The proof of Proposition~\ref{conex2} involves 
an inductive construction, where we add an alternative $z$ at the bottom of the first linear order in a geodesic, and at the top of the last linear order. 
Hence, we will need to ensure that the result satisfies certain peak-pit conditions, and we will be particularly interested in never-top conditions. 
In this section, we prove some technical lemmas to ensure that if the first switching pair of alternatives in the geodesic, together with $z$, satisfies a never-top condition, then any other pair of alternatives, together with $z$, will also satisfy the never-top condition, making the addition of $z$ feasible.

We will first examine two cases where a domain for a triple of alternatives satisfies only a never-top condition, then for a relevant triple, it can also satisfy a never-top condition.

\begin{lemma}\label{fgz2}
    Let $\cd$ be a peak-pit Condorcet domain on $\{f,g,s,z\}$ with  $R=fgsz\in \cd$ and $T\in\cd$ such that $T\in\{zgsf,zsgf\}$ and
    $N_p(\cd_{\{f,g,z\}})=\{gN_{\{f,g,z\}}1\}$.
Let $\cg=(G_1,\ldots, G_k)$ be a geodesic on $\cl(\{f,g,s\})$ connecting $R_{\{f,g,s\}}$ and $T_{\{f,g,s\}}$ such that $\cd_{\{f,g,s\}}\cup\{G_1,\ldots, G_k\}$ is a peak-pit Condorcet domain and $(f,g)$ is the first switching pair in $S(\cg)$. Then  $sN_{\{f,s,z\}}1\in N_p(\cd_{\{f,s,z\}})$.
\end{lemma}
\begin{proof}
Suppose for contradiction that $sN_{\{f,s,z\}}1\notin N_p(\cd_{\{f,s,z\}})$. 
Since $R_{\{f,s,z\}}=fsz$ and $T_{\{f,s,z\}}\in\{zgsf,zsgf\}_{\{f,s,z\}}=\{zsf\}$, we have
$\emptyset\ne N_p(\cd_{\{f,s,z\}})\subseteq \{sN_{\{f,s,z\}}1$, $sN_{\{f,s,z\}}3\}$ due to Lemma~\ref{wex0}\ref{wex0a}.
Since $sN_{\{f,s,z\}}1\notin N_p(\cd_{\{f,s,z\}})$, we have $N_p(\cd_{\{f,s,z\}})=\{sN_{\{f,s,z\}}3\}$ and so $\cd_{\{f,s,z\}}\subseteq\{fsz,zsf,sfz,szf\}$.

Now since $N_p(\cd_{\{f,g,z\}})=\{gN_{\{f,g,z\}}1\}$ and $\cd_{\{f,g,z\}}$ does not satisfy $gN_{\{f,g,z\}}3$, there is a linear order $S\in\cd$ such that 
    either $S_{\{f,g,z\}}=fzg$ or $S_{\{f,g,z\}}=zfg$. 
Suppose $S_{\{f,g,z\}}=fzg$. Since only $fsz$ and $sfz$ in $\cd_{\{f,s,z\}}$ rank $z$ after $f$, we have $S$ equals $fszg$ or $sfzg$.
Alternatively, suppose $S_{\{f,g,z\}}=zfg$.
Now only $zsf$ and $szf$ are elements of $\cd_{\{f,s,z\}}$ that rank $z$ before $f$. So 
$S$ equals either $zsfg$ or $szfg$.
Together, $S\in\{zsfg,szfg,fszg,sfzg\}$.
Note that $S_{\{f,g,s\}}\in\{sfg,fsg\}$ always ranks $g$ last. We still have two cases for $T$.

{\bf Case 1.} Suppose $T=zgsf$. 
Note that if $S$ = $fszg$, then $\{gsz,zgs,szg\}=\{R,T,S\}_{\{g,s,z\}}\subseteq \cd_{\{g,s,z\}}$ which does not satisfy any never-condition. This contradicts that $\cd$ is a peak-pit Condorcet domain.
If $S\in\{zsfg,szfg,sfzg\}$, then $S_{\{f,g,s\}}=sfg$.
Hence, $\{fgs,gsf,sfg\}$= $\{R,T,S\}_{\{f,g,s\}}\subseteq \cd_{\{f,g,s\}}$ which does not satisfy any never-condition.

{\bf Case 2.} Suppose $T=zsgf$. Then $T_{\{f,g,s\}}=sgf$.
Since $R_{\{f,g,s\}}=fgs$ and $(f,g)$ appears first in $S(\cg)$ by assumption, we have $N_p(\{G_1,\ldots, G_k\})=\{gN_{\{f,g,s\}}3\}$ by Lemma~~\ref{2geodesics}\ref{2geodesicsa}.
However, $S_{\{f,g,s\}}\in\{sfg,fsg\}$ always rank $g$ last. 
So $S_{\{f,g,s\}}\in\cd_{\{f,g,s\}}$ does not satisfy $gN_{\{f,g,s\}}3$.
This contradicts that $\cd_{\{f,g,s\}}\cup\{G_1$, $\ldots$, $G_k\}$ is a peak-pit Condorcet domain.

Therefore, $sN_{\{f,s,z\}}1\in N_p(\cd_{\{f,s,z\}})$.
\end{proof}

\begin{lemma}\label{fgu}
Let $\cd$ be a peak-pit Condorcet domain on $\{f,g,u,z\}$ with  $R=fguz\in \cd$ and $T=zugf\in\cd$ such that
    $N_p(\cd_{\{g,u,z\}})=\{uN_{\{g,u,z\}}1\}$.
Let $\cg=(G_1,\ldots, G_k)$ be a geodesic on $\cl(\{f,g,u\})$ connecting $R_{\{f,g,u\}}$ and $T_{\{f,g,u\}}$ such that $\cd_{\{f,g,u\}}\cup\{G_1,\ldots, G_k\}$ is a peak-pit Condorcet domain and $(g,u)$ is the first switching pair in $S(\cg)$. Then $N_p(\cd_{\{f,u,z\}})=\{uN_{\{f,u,z\}}1\}$.
\end{lemma}
\begin{proof}
Since $R=fguz$ and $T=zugf$, we have $R_{\{f,u,z\}}=fuz$ and $T_{\{f,u,z\}}=zuf$.
Hence, $\emptyset\ne N_p(\cd_{\{f,u,z\}})\subseteq \{uN_{\{f,u,z\}}1$, $uN_{\{f,u,z\}}3\}$ due to Lemma~\ref{wex0}\ref{wex0a}. 
Suppose for a contradiction that $N_p(\cd_{\{f,u,z\}})\ne\{uN_{\{f,u,z\}}1\}$.
Then we have two cases for $N_p(\cd_{\{f,u,z\}})$.
If $N_p(\cd_{\{f,u,z\}})=\{uN_{\{f,u,z\}}3\}$, then $\cd_{\{f,u,z\}}\subseteq\{fuz,ufz,uzf,zuf\}$.
If $N_p(\cd_{\{f,u,z\}})=\{uN_{\{f,u,z\}}3,uN_{\{f,u,z\}}1\}$.
Then $\cd_{\{f,u,z\}}$ = $\{fuz$, $zuf\}\subset\{fuz$, $ufz$, $uzf$, $zuf\}$ by Lemma~\ref{wex0}\ref{wex0b}.
So it either case, $\cd_{\{f,u,z\}}\subseteq\{fuz,ufz,uzf,zuf\}$.

Now $N_p(\cd_{\{g,u,z\}})=\{uN_{\{g,u,z\}}1\}$ by assumption and therefore $\cd_{\{g,u,z\}}$ does not satisfy $uN_{\{g,u,z\}}3$. 
Hence, there is a linear order $S\in\cd$ such that $S_{\{g,u,z\}}$ is either $gzu$ or $zgu$. 
In particular, $S_{\{g,u\}}=gu$ and $S_{\{z,u\}}=zu$.
Since $\cd_{\{f,u,z\}}\subseteq\{fuz,ufz,uzf,zuf\}$, the only linear order that can possibly be in $\cd_{\{z,u,f\}}$ and ranks $z$ before $u$ is $zuf$ and so $S_{\{f,u,z\}}=zuf$. 
Recall that $S_{\{g,u\}}=gu$.
Hence, $S_{\{f,g,u,z\}}$ is either $zguf$ or $gzuf$. 
Therefore, in both cases, $guf= S_{\{f,g,u\}}\in \cd_{\{f,g,u\}}$.

Recall $T=zugf$, so $T_{\{f,g,u\}}=ugf$.
Since $(g,u)$ appears first in $S(\cg)$, we have a never-top geodesic $\cg=(fgu, fug,ufg,ugf)$ due to Lemma~\ref{2geodesics}\ref{2geodesicsb}.
Since $guf=S_{\{f,g,u\}}\in\cd_{\{f,g,u\}}$, we have a contradiction that $\{guf\}\cup \{ fgu, fug,ufg,ugf\}\subset\cd_{\{f,g,u\}}\cup\{G_1,\ldots, G_k\}$ is a peak-pit Condorcet domain.
Thus,  $N_p(\cd_{\{f,u,z\}})=\{uN_{\{f,u,z\}}1\}$.
\end{proof}

We next examine two weaker cases where a domain for a triple of alternatives satisfies a never-top condition, then for a relevant triple, it can also satisfy a never-top condition.

\begin{lemma}\label{c211}
    Let $\cd$ be a peak-pit Condorcet domain on $\{f,g,s,z\}$ with  $R=sfgz\in \cd$ and $T\in\cd$ such that $T\in\{zgsf,zgfs\}$ and
    $gN_{\{f,g,z\}}1\in N_p(\cd_{\{f,g,z\}})$.
Let $\cg=(G_1,\ldots, G_k)$ be a geodesic on $\cl(\{f,g,s\})$ connecting $R_{\{f,g,s\}}$ and $T_{\{f,g,s\}}$ such that $\cd_{\{f,g,s\}}\cup\{G_1,\ldots, G_k\}$ is a peak-pit Condorcet domain and $(f,g)$ is the first switching pair in $S(\cg)$. Then $gN_{\{g,s,z\}}1\in N_p(\cd_{\{g,s,z\}})$.
\end{lemma}
\begin{proof}
    Suppose for a contradiction that $gN_{\{g,s,z\}}1\notin N_p(\cd_{\{g,s,z\}})$.
Then there is a linear order $S\in\cd$ such that 
    either $S_{\{g,s,z\}}=gsz$   or $S_{\{g,s,z\}}=gzs$.
In particular, $S_{\{g,s\}}=gs$ and $S_{\{g,z\}}=gz$.
On the other hand, since $R_{\{f,g,z\}}=fgz$ and $T_{\{f,g,z\}}\in\{zgsf,zgfs\}_{\{f,g,z\}}=\{zgf\}$, we have $\emptyset\ne N_p(\cd_{\{f,g,z\}})\subseteq \{gN_{\{f,g,z\}}1$, $gN_{\{f,g,z\}}3\}$ due to Lemma~\ref{wex0}\ref{wex0a}. 
Since $gN_{\{f,g,z\}}1\in N_p(\cd_{\{f,g,z\}})$ by assumption, there are two cases for $N_p(\cd_{\{f,g,z\}})$.
If $N_p(\cd_{\{f,g,z\}})=\{gN_{\{f,g,z\}}1\}$, then $\cd_{\{f,g,z\}}\subseteq\{fgz,fzg,zfg,zgf\}$.
If $N_p(\cd_{\{f,g,z\}})=\{gN_{\{f,g,z\}}3,gN_{\{f,g,z\}}1\}$, Lemma~\ref{wex0}\ref{wex0b} gives $\cd_{\{f,g,z\}}$ = $\{fgz,zgf\}\subset\{fgz,fzg,zfg,zgf\}$.
So in either case, $\cd_{\{f,g,z\}}\subseteq\{fgz$, $fzg$, $zfg$, $zgf\}$.
The only linear order that can possibly be in $\cd_{\{f,g,z\}}$ and ranks $g$ before $z$ is $fgz$ and so $S_{\{f,g,z\}}=fgz$. 
Recall that $S_{\{g,s\}}=gs$.
Hence, $S$ is either $fgzs$ or $fgsz$. 
Therefore, in both cases, $fgs= S_{\{f,g,s\}}\in \cd_{\{f,g,s\}}$.
We still have two cases for $T$.

{\bf Case 1.} Suppose $T=zgsf$. Then $\{sfg,gsf,fgs\}=\{R,T,S\}_{\{f,g,s\}}\subseteq \cd_{\{f,g,s\}}$ which does not satisfy any never-condition, contradicting that $\cd$ is a peak-pit Condorcet domain.

{\bf Case 2.} Suppose $T=zgfs$. Then $T_{\{f,g,s\}}=gfs$.
Since $(f,g)$ appears first in $S(\cg)$, we have a never-top geodesic $\cg=(sfg,sgf,gsf,gfs)$ due to Lemma~\ref{2geodesics}\ref{2geodesicsb}.
Since $fgs=S_{\{f,g,s\}}\in\cd_{\{f,g,s\}}$, we have a contradiction that $\{fgs\}\cup \{ sfg,sgf,gsf,gfs\}\subset\cd_{\{f,g,s\}}\cup\{G_1,\ldots, G_k\}$ is a peak-pit Condorcet domain.

Therefore, $gN_{\{g,s,z\}}1\in N_p(\cd_{\{g,s,z\}})$.
\end{proof}

\begin{lemma}\label{c22}
    Let $\cd$ be a peak-pit Condorcet domain on $\{f,g,s,z\}$ with  $R=fsgz\in \cd$ and $T=zgsf\in\cd$ such that $gN_{\{f,g,z\}}1\in N_p(\cd_{\{f,g,z\}})$.
Let $\cg=(G_1,\ldots, G_k)$ be a geodesic on $\cl(\{f,g,s\})$ connecting $R_{\{f,g,s\}}$ and $T_{\{f,g,s\}}$ such that $\cd_{\{f,g,s\}}\cup\{G_1,\ldots, G_k\}$ is a peak-pit Condorcet domain. 
Then
\begin{tabel}
\item\label{c221} If $(f,s)\vartriangleleft(f,g)\vartriangleleft(s,g)$ in $S(\cg)$, then $gN_{\{g,s,z\}}1\in N_p(\cd_{\{g,s,z\}})$.
\item\label{c222} If $(s,g)\vartriangleleft(f,g)\vartriangleleft(f,s)$ in $S(\cg)$, then $sN_{\{f,s,z\}}1\in N_p(\cd_{\{f,s,z\}})$.
\end{tabel}
\end{lemma}
\begin{proof}
Since  $R=fsgz$ and $T=zgsf$, we have $R_{\{f,g,z\}}=fgz$ and $T_{\{f,g,z\}}=zgf$. 
Then $\emptyset\ne N_p(\cd_{\{f,g,z\}})\subseteq \{gN_{\{f,g,z\}}1$, $gN_{\{f,g,z\}}3\}$ by Lemma~\ref{wex0}\ref{wex0a}. 
Since $gN_{\{f,g,z\}}1\in N_p(\cd_{\{f,g,z\}})$ is given, two cases for $N_p(\cd_{\{f,g,z\}})$.
If $N_p(\cd_{\{f,g,z\}})=\{gN_{\{f,g,z\}}1\}$, then $\cd_{\{f,g,z\}}\subseteq\{fgz$, $fzg$, $zfg$, $zgf\}$.
If $N_p(\cd_{\{f,g,z\}})=\{gN_{\{f,g,z\}}3$, $gN_{\{f,g,z\}}1\}$, then Lemma~\ref{wex0}\ref{wex0b} gives $\cd_{\{f,g,z\}}$ = $\{fgz,zgf\}\subset\{fgz$, $fzg$, $zfg$, $zgf\}$.
So in either case, $\cd_{\{f,g,z\}}\subseteq\{fgz$, $fzg$, $zfg$, $zgf\}$.

(a) Suppose for a contradiction that $gN_{\{g,s,z\}}1\notin N_p(\cd_{\{g,s,z\}})$.
Then there is a linear order $S\in\cd$ such that 
    either $S_{\{g,s,z\}}=gsz$   or $S_{\{g,s,z\}}=gzs$.
In particular, $S_{\{g,s\}}=gs$ and $S_{\{g,z\}}=gz$.
Since $\cd_{\{f,g,z\}}\subseteq\{fgz,fzg,zfg,zgf\}$, the only linear order that can possibly be in $\cd_{\{f,g,z\}}$ and ranks $g$ before $z$ is $fgz$ and so $S_{\{f,g,z\}}=fgz$. 
Recall that $S_{\{g,s\}}=gs$.
Hence, $S$ is either $fgzs$ or $fgsz$. 
Therefore, in both cases, $fgs= S_{\{f,g,s\}}\in \cd_{\{f,g,s\}}$.
Moreover, since $(f,s)\vartriangleleft(f,g)\vartriangleleft(s,g)$ in $S(\cg)$, as shown in Figure~\ref{ecp3}, we have a never-bottom geodesic $\cg_{\{f,g,s\}}=(fsg,sfg,sgf,gsf)$ due to Lemma~\ref{2geodesics}\ref{2geodesicsa}.
Then we have a contradiction that $\{fgs\}\cup \{fsg,sfg,sgf,gsf\}\subset\cd_{\{f,g,s\}}\cup\{G_1,\ldots, G_k\}$ is a peak-pit Condorcet domain.

\begin{figure}[H]
     \centering
\begin{tikzpicture}[scale=0.60]

\draw[fill=black] (3,0) circle (3pt);
\draw[fill=black] (3,2) circle (3pt);
\draw[fill=black] (3,4) circle (3pt);
\draw[fill=black] (3,6) circle (3pt);
\draw[fill=black] (22,0) circle (3pt);
\draw[fill=black] (22,2) circle (3pt);
\draw[fill=black] (22,4) circle (3pt);
\draw[fill=black] (22,6) circle (3pt);
 
\node at (2.5,0) {$z$};
\node at (2.5,2) {$g$};
\node at (2.5,4) {$s$};
\node at (2.5,6) {$f$};
\node at (22.5,0) {$f$};
\node at (22.5,2) {$s$};
\node at (22.5,4) {$g$};
\node at (22.5,6) {$z$};

\node at (9.5,1) {$(g,z)$};
\node at (15.5,5) {$(s,z)$};
\node at (12.5,3) {$(f,z)$};
\node at (6.5,5) {$(f,s)$};
\node at (18.5,1) {$(f,g)$};
\node at (21.5,3) {$(s,g)$};

\draw[thick,dashed] (3,0) -- (7,0)--(9,2)--(10,2) --(12,4)--(13,4)--(15,6)--(19,6)--(22,6);
\draw[thick] (3,2)--(7,2) --(9,0)--(16,0)--(18,2)--(19,2)--(21,4)--(22,4);
\draw[thick] (3,4)--(4,4)--(6,6)--(13,6)--(15,4)--(19,4)--(21,2)--(22,2);
\draw[thick] (3,6)--(4,6)--(6,4)--(10,4)-- (12,2) --(16,2)--(18,0)-- (19,0)--(22,0);

\end{tikzpicture}
     \caption{Case (a).} \label{ecp3}
\end{figure} 

(b) Suppose for a contradiction that $sN_{\{f,s,z\}}1\notin N_p(\cd_{\{f,s,z\}})$.
Then there is a linear order $S\in\cd$ such that 
    either $S_{\{f,s,z\}}=sfz$ or $S_{\{f,s,z\}}=szf$.
In particular, $S_{\{f,s\}}=sf$ and $S_{\{s,z\}}=sz$.
Since $\cd_{\{f,g,z\}}\subseteq\{fgz,fzg,zfg,zgf\}$ and the top alternative in every linear order in $\cd_{\{f,g,z\}}$ is $f$ or $z$, the linear order $S$ always ranks $s$ first.
So $S_{\{f,g,s\}}\in\cd_{\{f,g,s\}}$ does not satisfy $sN_{\{f,g,s\}}1$.
However, since $(s,g)\vartriangleleft(f,g)\vartriangleleft(f,s)$ in $S(\cg)$, as shown in Figure~\ref{ecp3.1}, we have 
$N_p(\{G_1,\ldots, G_k\})=\{sN_{\{f,g,s\}}1\}$ by Lemma~\ref{2geodesics}\ref{2geodesicsb}.
This contradicts that $\cd_{\{f,g,s\}}\cup\{G_1$, $\ldots$, $G_k\}$ is a peak-pit Condorcet domain.
\end{proof}
\begin{figure}[H]
     \centering
\begin{tikzpicture}[scale=0.60]

\draw[fill=black] (3,0) circle (3pt);
\draw[fill=black] (3,2) circle (3pt);
\draw[fill=black] (3,4) circle (3pt);
\draw[fill=black] (3,6) circle (3pt);
\draw[fill=black] (22,0) circle (3pt);
\draw[fill=black] (22,2) circle (3pt);
\draw[fill=black] (22,4) circle (3pt);
\draw[fill=black] (22,6) circle (3pt);
 
\node at (2.5,0) {$z$};
\node at (2.5,2) {$g$};
\node at (2.5,4) {$s$};
\node at (2.5,6) {$f$};
\node at (22.5,0) {$f$};
\node at (22.5,2) {$s$};
\node at (22.5,4) {$g$};
\node at (22.5,6) {$z$};

\node at (12.5,3) {$(g,z)$};
\node at (9.5,1) {$(s,z)$};
\node at (15.5,5) {$(f,z)$};
\node at (6.5,3) {$(s,g)$};
\node at (18.5,3) {$(f,g)$};
\node at (21.5,1) {$(f,s)$};

\draw[thick,dashed] (3,0) -- (7,0)--(9,2)--(10,2) --(12,4)--(13,4)--(15,6)--(19,6)--(22,6);
\draw[thick] (3,2)--(4,2) --(6,4)--(10,4)--(12,2)--(16,2)--(18,4)--(19,4)--(22,4);
\draw[thick] (3,4)--(4,4)--(6,2)--(7,2)--(9,0)--(19,0)--(21,2)--(22,2);
\draw[thick] (3,6)--(13,6)--(15,4)-- (16,4)--(18,2)-- (19,2)--(21,0)--(22,0);

\end{tikzpicture}
     \caption{Case (b).} \label{ecp3.1}
\end{figure} 

At last, we need another technical lemma that involves with a never-bottom condition, which will be used as a contradiction for a case analysis in the next section.
Lemmas~\ref{fgz2} and \ref{fgu}, 
serve the following lemma.

\begin{lemma}\label{c232.1}
    Let $\cd\subseteq\cl(B)$ be a peak-pit Condorcet domain. Let $R,T\in\cd$ be distinct linear orders, and $a,b,f,s,z\in B$ be distinct alternatives. 
    Suppose $R_{\{a,b,f,s,z\}}=fabsz$ and $T_{\{a,b,f,z\}}=zbaf$. 
    Suppose $sN_{\{f,s,z\}}1\notin N_p(\cd_{\{f,s,z\}})$.
Let $\cg=(G_1,\ldots, G_k)$ be a geodesic on $\cl(B\setminus\{z\})$ connecting $R_{B\setminus\{z\}}$ and $T_{B\setminus\{z\}}$ such that $\cd_{B\setminus\{z\}}\cup\{G_1,\ldots, G_k\}$ is a peak-pit Condorcet domain.
Suppose  $(a,b)\vartriangleleft(f,b)\vartriangleleft(f,a)\vartriangleleft(f,s)$ in $S(\cg)$.
     Then $bN_{\{a,b,z\}}3\in N_p(\cd_{\{a,b,z\}})$.
     \end{lemma}
\begin{proof}
Note that $R_{\{a,b,f,s,z\}}=fabsz$ and $T_{\{a,b,f,z\}}=zbaf$.
Since $T_{\{f,s,z\}}=zsf$, we have $T_{\{b,f,s,z\}}\in\{zbsf, zsbf\}$.
Now $R_{\{a,b,z\}}=abz$ and $T_{\{a,b,z\}}=zba$, we have
$\emptyset\ne N_p(\cd_{\{a,b,z\}})\subseteq \{bN_{\{a,b,z\}}1$, $bN_{\{a,b,z\}}3\}$ due to Lemma~\ref{wex0}\ref{wex0a}.
Suppose for a contradiction that $bN_{\{a,b,z\}}3\notin N_p(\cd_{\{a,b,z\}})$.
Then $N_p(\cd_{\{a,b,z\}})=\{bN_{\{a,b,z\}}1\}$.
Hence, $N_p(\cd_{\{b,f,z\}})=\{bN_{\{b,f,z\}}1\}$ due to  Lemma~\ref{fgu}.
Next suppose that the switching pair $(b,s)$ occurs in $S(\cg_{\{b,f,s\}})$. Now $R_{\{b,f,s\}}=fbs$.
Since $(f,b)\vartriangleleft(f,s)$, the situation of Lemma~\ref{2geodesics}\ref{2geodesicsb} does not occur and hence by Lemma~\ref{2geodesics}\ref{2geodesicsa}, one has $(f,b)\vartriangleleft(f,s)\vartriangleleft(b,s)$.
Hence, $(b,s)$ is after $(f,b)$ if it occurs.
Therefore, in any case, $(f,b)$ is the first switching pair in $S(\cg_{\{b,f,s\}})$.
Then $sN_{\{f,s,z\}}1\in N_p(\cd_{\{f,s,z\}})$ due to Lemma~\ref{fgz2}. This contradicts the assumption.
Therefore, $bN_{\{a,b,z\}}3\in N_p(\cd_{\{a,b,z\}})$.
\end{proof}

\section{Proof for Proposition \ref{conex2}}

We now synthesise the ideas from the previous sections to prove  Proposition \ref{conex2}: {\em Let $\cd\subseteq\cl(A)$ be a peak-pit Condorcet domain. Let $R, T\in \cd$ be linear orders and $B\subseteq A$ be a subset such that $|B|\geq 3$. Then there exists a geodesic $\ca$  connecting $R_{B}$ and $T_{B}$ such that $\cd_{B}\cup K(\ca)$ is a peak-pit Condorcet domain}.

\begin{proof}
     Let $R, T\in \cd\subseteq \cl(A)$ be linear orders. Let $B\subseteq A$ be such that $|B|\geq 3$. We will use induction on the cardinality of $B$ to show that there exists a geodesic $\ca$  connecting $R_{B}$ and $T_{B}$ such that $\cd_{B}\cup K(\ca)$ is a peak-pit Condorcet domain. 
Suppose $|B|=3$. Then there exists a desired geodesic connecting $R_B$ and $T_B$ due to Lemma~\ref{wex1.2}\ref{wex1.2d}.

For the induction step, let $n\geq 4$ and assume that whenever $B\subseteq A$ and $|B|=n-1$, then there exists a geodesic $\ca$ connecting $R_B$ and $T_B$ such that $\cd_{B}\cup K(\ca)$ is a peak-pit Condorcet domain, with $k\geq 1$.
Now fix $B\subseteq A$ with $|B|=n$ and let $z$ be the alternative ranked last in $R_B$. Write $$T_B=\cdots z t_1\cdots t_q,$$ with $q\geq 0$.




Define the set $C$ of pairs in $\{t_1,\ldots, t_q\}^2$ by
 $$C=\{(a,b): a,b\in \{t_1,\ldots, t_q\},a\ne b, R_{\{a,b,z\}}=abz,\;\text{and}\; T_{\{a,b,z\}}=zba\}.$$
Note that $C=\emptyset$ when $q\leq 1$.
We want to emphasise that whenever $\cg$ is a geodesic connecting $R_{B\setminus\{z\}}$ and $T_{B\setminus\{z\}}$, then
\( C \) is a collection of pairs in \(\{t_1,\ldots, t_q\}^2\) that are switched by \(\cg\). The alternative \( z \) is included in the definition of \( C \) primarily for our later case analysis, as the positions of \( z \) in \( R_B \) and \( T_B \) are known.
Note that if $(a,b)\in C$, then $R_{\{a,b,z\}}=abz$ and $T_{\{a,b,z\}}=zba$ are in $\cd_{\{a,b,z\}}$. Hence, the set of peak-pit conditions satisfied by $\cd_{\{a,b,z\}}$ is a non-empty subset of $\{bN_{\{a,b,z\}}1, bN_{\{a,b,z\}}3\}$ due to Lemma~\ref{wex0}\ref{wex0a}. 
Further define the subset,
$$C_{NT}=\{(a,b)\in C: N_p(\cd_{\{a,b,z\}})=\{bN_{\{a,b,z\}}1\}.$$
Note that if $(a,b)\in C\setminus C_{NT}$, then $\cd_{\{a,b,z\}}$ satisfies $bN_{\{a,b,z\}}3$ due to Lemma~\ref{wex0}\ref{wex0a}, and possibly also $bN_{\{a,b,z\}}1$. Also note that the definitions of $C$ and $C_{NT}$ do not involve any geodesic $\cg$.
There are two cases depending on whether or not $C_{NT}$ is empty.


{\bf Case 1.} Suppose $|C_{NT}|=0$. By hypothesis there exists a geodesic $\cg=(G_1,\ldots, G_k)$ connecting $R_{B\setminus\{z\}}$ and $T_{B\setminus\{z\}}$ such that $\cd_{B\setminus\{z\}}\cup\{G_1,\ldots, G_k\}$ is a peak-pit Condorcet domain.
Suppose the sequence of switching pairs of alternatives for $\cg$ is
$$S(\cg): (r_1,s_1),\ldots, (r_{k-1}, s_{k-1}).$$
We write the switching pairs such that $R_{\{r_i,s_i\}}=r_is_i$ for all $i\in[k-1]$.
Note that the alternative $z$ does not occur in any of the switching pairs in $S(\cg)$.
Let $\cg^*=(G^*_1,\ldots,G^*_{k'})$ be characterised by 
$$S(\cg^*):(r_1,s_1),\ldots, (r_{k-1}, s_{k-1}),  (t_q,z), \ldots , (t_2,z), (t_1,z),$$
where $k'=k+q$. Note that if $q=0$, then $S(\cg^*)=S(\cg)$.
Since $\cg$ is a geodesic and each $(z,t_p)$ only occurs once in $\cg^*$, Lemma~\ref{wex1.1}\ref{wex1.1b} gives that $\cg^*$ is a geodesic. 
It remains to show that $\cd_{B}\cup\{G^*_1,\ldots,G^*_{k'}\}$ is a peak-pit Condorcet domain. Let $a,b,c\in B$ be distinct alternatives. If $z\notin\{a,b,c\}$, then again $(\cd_B\cup\{G^*_1,\ldots,G^*_{k'}\})_{\{a,b,c\}}=(\cd_{B\setminus\{z\}}\cup\{G_1,\ldots,G_k\})_{\{a,b,c\}}$ is a peak-pit Condorcet domain by hypothesis. 
Suppose $z\in\{a,b,c\}$, say $z=c$. 
Suppose 
$(a,b)\notin C$ and $(b,a)\notin C$, then $R_{\{a,b,z\}}$ and $T_{\{a,b,z\}}$ differ by at most two swaps of adjacent alternatives. Hence, by Lemma~\ref{wex1.2}\ref{wex1.2b}, $(\cd_{B}\cup\{G^*_1,\ldots,G^*_{k'}\})_{\{a,b,c\}}$ is a peak-pit Condorcet domain. 

Otherwise, without loss of generality, suppose $(a,b)\in C$. 
Then $\cg^*_{\{a,b,z\}}$ is a geodesic connecting $R_{\{a,b,z\}}=abz$ and $T_{\{a,b,z\}}=zba$. 
Note that there are only two geodesics connecting $abz$ and $zba$ due to Lemma~\ref{2geodesics}, which are given by the sequences of switching pairs $(a,b)\vartriangleleft(a,z)\vartriangleleft(b,z)$ and $(b,z)\vartriangleleft(a,z)\vartriangleleft(a,b)$, respectively.
Now the switching pairs involving the alternative $z$ are at the end of $S(\cg^*)$ by the definition of $S(\cg^*)$. Hence, $S(\cg^*_{\{a,b,z\}})$ is $(a,b)\vartriangleleft(a,z)\vartriangleleft(b,z)$.
Therefore, we have a never-bottom geodesic $\cg^*_{\{a,b,z\}}$ that satisfies the never-bottom condition $bN_{\{a,b,z\}}3$ due to Lemma~\ref{2geodesics}\ref{2geodesicsa}.
Since $C_{NT}=\emptyset$, the peak-pit Condorcet domain $\cd_{\{a,b,z\}}$ 
satisfies $bN_{\{a,b,z\}}3$ by Lemma~\ref{wex0}\ref{wex0a}.
Consequently, 
$(\cd_{B}\cup\{G^*_1,\ldots,G^*_{k'}\})_{\{a,b,z\}}$ satisfies $bN_{\{a,b,z\}}3$ and it is a peak-pit Condorcet domain. 
Thus, $\cg^*$ is a geodesic connecting $R_B$ and $T_B$ such that $\cd_{B}\cup\{G^*_1,\ldots,G^*_{k'}\}$ is a peak-pit Condorcet domain.

{\bf Case 2.} Suppose $|C_{NT}|\geq 1$.
Then there exists at least one element in $C_{NT}$.
If $\cg$ is a geodesic connecting $R_{B\setminus\{z\}}$ and $T_{B\setminus\{z\}}$ and if $(a,b)\in C_{NT}$, then the switching pair $(a,b)$ is in $S(\cg)$ since $R_{\{a,b\}}=ab$ and $T_{\{a,b\}}=ba$.
We need three lemmas to find a well-behaved geodesic $\cg=(G_1,\ldots, G_k)$ on $B\setminus\{z\}$ connecting $R_{B\setminus\{z\}}$ and $T_{B\setminus\{z\}}$ satisfying the following four conditions $A_1-A_4$, before continuing the proof of Proposition~\ref{conex2}.

\begin{itemize}
    \item \(\text{Condition}\; A_1 \): $\cd_{B\setminus\{z\}}\cup\{G_1,\ldots, G_k\}$ is a peak-pit Condorcet domain.
    \item \(\text{Condition}\;  A_2 \): Let $(d,e)$ be the first switching pair in $S(\cg)$ among all switching pairs in $C_{NT}$. Let $L=\cdots t'_1\cdots t'_{q}\in\cl(B\setminus\{z\})$ be the linear order in $\cg$ which is immediately before the swap $(d,e)$ in $\cg$, then $\{t'_1,\ldots, t'_{q}\}=\{t_1,\ldots, t_{q}\}$. 
    \item \(\text{Condition}\;  A_3 \): All switching pairs in $\cg'$ which occur after $(d,e)$ are in $C$.
    \item \(\text{Condition}\;  A_4 \): All switching pairs in $\cg'$, which are after $(d,e)$, are in $C_{NT}$ or not disjoint with at least one switching pair between $(d,e)$ and the given switching pair, including $(d,e)$.
\end{itemize}

\begin{lemma}\label{conex2.1}
    Adopt the above notation and assumption. 
    There exists a geodesic $\cg=(G_1,\ldots, G_k)$ on $B\setminus\{z\}$ connecting $R_{B\setminus\{z\}}$ and $T_{B\setminus\{z\}}$ satisfying Conditions $A_1-A_2$.
\end{lemma}
\begin{proof}
By hypothesis there exists a geodesic $\cg=(G_1,\ldots, G_k)$ connecting $R_{B\setminus\{z\}}$ and $T_{B\setminus\{z\}}$ such that $\cd_{B\setminus\{z\}}\cup\{G_1,\ldots, G_k\}$ is a peak-pit Condorcet domain.
Suppose the sequence of switching pairs of alternatives for $\cg$ is
$$S(\cg): (r_1,s_1),\ldots, (r_{k-1}, s_{k-1}).$$
We write the switching pairs such that $R_{\{r_i,s_i\}}=r_is_i$ for all $i\in[k-1]$.
Note that the alternative $z$ does not occur in any of the switching pairs in $S(\cg)$.
Let $(d,e)$ be the first switching pair in $S(\cg)$ among all switching pairs in $C_{NT}$.
Let $L=\cdots t^*_1\cdots t^*_{q}$ be the linear order that is immediately before the swap $(d,e)$ in $\cg$.
If $\{t^*_1,\ldots, t^*_{q}\}=\{t_1,\ldots, t_{q}\}$, then we are done. Suppose $\{t^*_1,\ldots, t^*_{q}\}\ne \{t_1,\ldots, t_{q}\}$.
We only consider the switching pairs in $S(\cg)$ that occur after $(d,e)$. Define the set
$$K=\{(a,b): (d,e)\trianglelefteq (a,b)\}.$$ 

Write $\overline{B}=B\setminus\{z\}$ and let $w\in \overline{B}\setminus\{t_1\ldots, t_q\}$ be such that $$T_{\overline{B}}=\cdots wt_1\cdots t_q.$$ 
We will show that there is an equivalent geodesic $\cg'$ of $\cg$ such that if $L'$ is the linear order immediately before $(d,e)$ in $\cg'$, then for all $h\in \{t_1,\ldots, t_q\}$, we have $L'_{\{h,w\}}=wh$, and for all $h\in \overline{B}\setminus \{w,t_1,\ldots, t_q\}$, we have $L'_{\{h,w\}}=hw$. Hence, it is useful to define the set of alternatives that swap with $w$ after $(d,e)$ in $\cg$. 
Let $H$ be the swap set of $w$. In particular, 
$$H=\{h\in \overline{B}\setminus\{w\}: (h,w)\in K\}.$$

We first show that $H\ne\emptyset$.
Suppose $H=\emptyset$. If $p\in [q]$, then $T_{\{t_p,w\}}=wt_p$ and $(t_p,w)\notin T(\cg)$, so $L_{\{t_p,w\}}=wt_p$. Hence, $t_p$ is ranked after $w$ in $L$. If $v\in \overline{B}\setminus\{w,t_1,\ldots, t_q\}$, then $T_{\{v,w\}}=vw$ and $(v,w)\notin G$, so $L_{\{v,w\}}=vw$. Hence, $v$ is ranked before $w$ in $L$.
Therefore, the set of the last $q$ ranked alternatives in $L$ is $\{t_1,\ldots t_q\}$. 
However, we assumed that $\{t^*_1,\ldots, t^*_{q}\}\ne\{t_1,\ldots t_q\}$, so this gives a contradiction. Hence, $H\ne \emptyset$.

Let $K_1$ be the swap closure of $w$, in particular, 
\begin{align*}
    K_1=\{(a,b)\in K:\;&\text{there exist}\; h\in H\; \text{and}\; u\in \overline{B}\setminus{\{h\}}\;\text{such that}\;\\ &(h,u)\in K\;\text{and} \\
    & (a,b)\trianglelefteq (h,u)\trianglelefteq (h,w)\;\text{and}\;  \\
    & \{a,b\}\cap\{h,u\}\ne\emptyset\},
\end{align*}
and $$K_2=K\setminus K_1.$$ Note that if $h\in H$, then $(h,w)\in K_1$ by choosing $u=w$ in the definition of $K_1$.
Note that each switching pair in $K_1$ is disjoint from each switching pair in $K_2$ which appears before it in $S(\cg)$ due to Lemma~\ref{conpropl1}.

We will now show that $(d,e)\in K_2$. For this we first show that all the alternatives ranked after $e$ in $L$ are in $\{t_1,\ldots, t_q\}$.
Let $v$ be an alternative ranked after $e$ in $L$. If $v\in \overline{B}\setminus\{t_1,\ldots, t_q\}$, then  $L_{\{d,e,w\}}=dev$. and $T_{\{d,e,w\}}=ved$.
Since $(d,e)\in C_{NT}$, the sequence of switching pairs in $S(\cg_{\{d,e,v\}})$ is $(e,v)$, $(d,v)$, $(d,e)$ due to Lemma~\ref{conpropl2}, which contradicts the fact that $L$ is immediately before $(d,e)$. Therefore, all alternatives ranked after $e$ in $L$ are in $\{t_1,\ldots, t_q\}$.

Now suppose for a contradiction that $(d,e)\in K_1$. Then there exist $h\in H$ and $u\in  \overline{B}\setminus{\{h\}}$ such that $(h,u)\in K$ and $(d,e)\trianglelefteq (h,u)\trianglelefteq (h,w)$, and $\{d,e\}\cap \{h,u\}\ne\emptyset$.
Since $w\notin \{t_1,\ldots, t_q\}$, we have $L_{\{d,e,w\}}=wde$. Since $T_{\{d,e,w\}}=wed$, both $(d,w)$ and $(e,w)$ are not switching pairs in $K$ and so $d\ne h$ and $e\ne h$. Hence, $u\in \{d,e\}$.
Since $(h,w)\in K$, if $L_{\{d,e,h,w\}}=wdeh$, then $h\in \{t_1,\ldots, t_q\}$ and so $(h,w)$ is not a switching pair in $K$, a contradiction. Hence, $L_{\{d,e,h,w\}}=hwde$ or $L_{\{d,e,h,w\}}=whde$.
If $L_{\{d,e,h,w\}}=hwde$, then $h$ needs to swap with $w$ before it can swap with $d$ or $e$, so that $(h,w)\vartriangleleft (h,u)$, a contradiction.
Similarly, if $L_{\{d,e,h,w\}}=whde$, then after $h$ swaps with $d$ or $e$, it cannot swap with $w$ since $\cg$ is a geodesic. Hence, it is not possible that $(h,u)\vartriangleleft (h,w)$. Therefore, $(d,e)\in K_2$.

We now 
define a geodesic $\cg'$ that is equivalent to $\cg$ such that if $L'=\cdots t'_1\cdots t'_q$  is the linear order in $\cg'$ immediately before the swap $(d,e)$, then $\{t'_1,\ldots, t'_q\}=\{t_1, \ldots, t_q\}$.
List all switching pairs in \( K_1 \) as \((a_1, b_1)\vartriangleleft (a_2, b_2)\vartriangleleft \cdots \vartriangleleft (a_m, b_m)\). The construction of $\cg'$ is done by $m$ steps of repeatedly swapping between the switching pairs in $K$  in the following sense.
The first step is to swap \((a_1, b_1)\) with each of the switching pairs in \( K_2 \) which appear before \((a_1, b_1)\) in $S(\cg)$, so that \((a_1, b_1)\) eventually occupies the position immediately before \((d, e)\). 
Subsequently, swap \((a_2, b_2)\) with each of the switching pairs in \( K_2 \) which appear before \((a_2, b_2)\) in $S(\cg)$, so that \((a_2, b_2)\) eventually occupies the position between \((a_{1}, b_{1})\) and \((d, e)\). 
Continue this process for \((a_3, b_3)\), $\dots$, \((a_m, b_m)\) end up in the order immediately before $(d,e)$ and we obtained the geodesic \( \cg' \) which is equivalent to \(\cg\).

Let $L'=\cdots t'_1\cdots t'_q$  be the linear order in $\cg'$ immediately before the swap $(d,e)$. Let $h\in \overline{B}\setminus\{w\}$.

{\bf Case i.} Suppose $h\in\{t_1,\ldots, t_q\}\cap H$. Then $(h,w)\in K_1$. Hence in $\cg'$, the switching pair $(h,w)$ occurs before $(d,e)$. Since $T_{\{h,w\}}=wh$, one obtains that $L'_{\{h,w\}}=wh$.

{\bf Case ii.} Suppose $h\in\{t_1,\ldots, t_q\}\setminus H$. Then $(h,w)$ is not in $K$ by the definition of $H$. Since the set of switching pairs after $(d,e)$ in $\cg'$ is a subset of $K$, the switching pair $(h,w)$ does not occur after $(d,e)$ in $\cg'$. Since $T_{\{h,w\}}=wh$,
one obtains that $L'_{\{h,w\}}=wh$.

{\bf Case iii.} Suppose $h\in H\setminus\{t_1,\ldots, t_q\}$. Then $(h,w)\in K_1$. Hence in $\cg'$, the switching pair $(h,w)$ occurs before $(d,e)$. Since $T_{\{h,w\}}=hw$, one obtains that $L'_{\{h,w\}}=hw$.

{\bf Case iv.} Suppose $h\notin \{t_1,\ldots, t_q\}\cup H$. Then $(h,w)$ is not in $K$. Since the set of switching pairs after $(d,e)$ in $\cg'$ is a subset of $K$, the switching pair $(h,w)$ does not occur after $(d,e)$ in $\cg'$. Since $T_{\{h,w\}}=hw$,
one obtains that $L'_{\{h,w\}}=hw$.

Therefore, for all $h\in \{t_1,\ldots, t_q\}$, we have $L'_{\{h,w\}}=wh$, and for all $h\in \overline{B}\setminus \{w, t_1,\ldots, t_q\}$, we have $L'_{\{h,w\}}=hw$.
Thus, $\{t'_1,\ldots, t'_q\}=\{t_1, \ldots, t_q\}$.
Hence, there is a geodesic $\cg'$ that is equivalent to $\cg$ such that if $L'=\cdots t'_1\cdots t'_q$  is the linear order in $\cg'$ immediately before the swap $(d,e)$, then $\{t'_1,\ldots, t'_q\}=\{t_1, \ldots, t_q\}$.

We next show that $(d,e)$ remains the switching pair which appears first in $S(\cg')$ among all switching pairs in $C_{NT}$. Since only the switching pairs in $K_1$ moved before $(d,e)$, this is equivalent to showing $C_{NT}\cap K_1=\emptyset$.
Let $(a,b)\in C_{NT}\cap K_1$. 
Since $(a,b)\in C_{NT}\subseteq C$, we have $T_{\{a,b,z\}}=zba$.
Since $w\notin\{t_1,\ldots, t_{q},z\}$, we have $T_{\{a,b,w,z\}}=wzba$.
Therefore,  $T_{\{a,b,v\}}=wba$.
Since $(a,b)\in C_{NT}\subseteq C$, we have $R_{\{a,b\}}=ab$ and $T_{\{a,b\}}=ba$. So the switching pair $(a,b)$ occurs in $S(\cg)$.
Since $(a,b)\in C_{NT}$, we have $(d,e)\vartriangleleft (a,b)$. 
Since $L$ is the linear order immediately before $(d,e)$, we have $L_{\{a,b\}}=R_{\{a,b\}}=ab$.
Therefore, $N_p(\cd_{\{a,b,z\}})\ne\{bN_{\{a,b,z\}}1\}$ due to Lemma~\ref{conpropl3}.
Therefore, it is not possible for such a switching pair $(a,b)$ to exist and so $C_{NT}\cap K_1=\emptyset$. 

Thus, there is a geodesic $\cg'$ equivalent to $\cg$ such that $(d,e)$ is the first switching pair in $S(\cg')$ among all switching pairs in $C_{NT}$, and the linear order $L'=\cdots t'_1\cdots t'_q$ in $\cg'$ that immediately before the swap $(d,e)$ has that $\{t'_1,\ldots, t'_{q}\}=\{t_1,\ldots, t_{q}\}$. This is Condition $A_2$ for $\cg'$.
Since $\cg'=(G'_1,\ldots,G'_k)$ is equivalent with $\cg=(G_1,\ldots,G_k)$, the set $\cd_{B\setminus\{z\}}\cup\{G'_1,\ldots, G'_k\}$ satisfies the same set of never-conditions as $\cd_{B\setminus\{z\}}\cup\{G_1,\ldots, G_k\}$ due to Lemma~\ref{wex1.3}\ref{wex1.3c}. In particular, it is a peak-pit Condorcet domain, which is Condition $A_1$ for $\cg'$.
\end{proof}

We next improve Lemma~\ref{conex2.1}.

\begin{lemma}\label{conex2.2}
    There exists a geodesic $\cg=(G_1,\ldots, G_k)$ on $B\setminus\{z\}$ connecting $R_{B\setminus\{z\}}$ and $T_{B\setminus\{z\}}$ satisfying Conditions $A_1-A_3$.
\end{lemma}
\begin{proof}
Let $\cg$ be as in Lemma~\ref{conex2.1} that satisfies Conditions $A_1-A_2$. Let $(d,e)$ be the first switching pair in $S(\cg)$ among all switching pairs in $C_{NT}$.
In particular, the linear order $L=\cdots t'_1\cdots t'_{q}$ in $\cg$ that is immediately before the swap $(d,e)$ gives $\{t'_1,\ldots, t'_{q}\}=\{t_1,\ldots, t_{q}\}$. 
Let $j\in[k-1]$ be such that $(d,e)=(r_{j},s_{j})$. Write the sequence $S(\cg)$ of switching pairs in $\cg$ as
$$S(\cg):(r_1,s_1),\ldots, (r_{j-1},s_{j-1}), (d,e),(r_{j+1},s_{j+1}), \ldots, (r_{k-1}, s_{k-1}).$$
Again we write the switching pairs such that $R_{\{r_i,s_i\}}=r_is_i$ for all $i\in[k-1]$.

We next 
construct a geodesics $\cg'$ that is equivalent with $\cg$. Then construction involves moving switching pairs after $(d,e)$ but not in $C$ to the positions before $(d,e)$.
Let $\ell, m\in\{j,\cdots, k-1\}$ and suppose $(r_\ell,s_\ell)\in C$ and $(r_m,s_m)\notin C$.
We claim that $\{r_\ell,s_\ell\}\cap\{r_m,s_m\}=\emptyset$. Suppose not. Without loss of generality, suppose $r_m\in\{r_\ell,s_\ell\}$.
Since $(r_\ell,s_\ell)\in C\subseteq\{t_1,\ldots,t_q\}^2$, it follows that $r_m\in\{t_1,\ldots,t_q\}$. Suppose $s_m\in\{t_1,\ldots,t_q\}$ for a contradiction.
Since $z$ is ranked last in $R_B$ and $T_B=\cdots zt_1\cdots t_q$ and also $(r_m,s_m)$ is in $S(\cg)$, we have $R_{\{r_m,s_m,z\}}$ and $T_{\{r_m,s_m,z\}}$ differ by three swaps of adjacent alternatives. Hence, $(r_m,s_m)\in C$, a contradiction. So $s_m\notin\{t_1,\ldots,t_q\}$.
Now $L_{\{r_\ell,s_\ell,s_m\}}= s_mr_\ell s_\ell$ by Lemma~\ref{conex2.1}.
Note that $(r_\ell,s_\ell)\in C$, so $T_{\{r_\ell,s_\ell\}}=s_\ell r_\ell$.
Since $(r_m,s_m)$ appear after $(d,e)$ in $\cg$ and $r_m\in\{r_\ell,s_\ell\}$, we have $(r_\ell,s_m)$ or $(s_\ell,s_m)$ is in $S(\cg)$.
If $(r_\ell,s_m)$ is in $S(\cg)$, then $T_{\{r_\ell,s_m\}}=r_\ell s_m$ and so $T_{\{r_\ell,s_\ell,s_m\}}=s_\ell r_\ell s_m$.
If $(s_\ell,s_m)$ is in $S(\cg)$, then $T_{\{s_\ell,s_m\}}=s_\ell s_m$. In both cases, $T_{\{s_\ell,s_m\}}=s_\ell s_m$. Since $T_B=\cdots zt_1\cdots t_q$ and $s_\ell\in\{t_1,\ldots, t_q\}$, we have $s_m\in\{t_1,\ldots, t_q\}$, a contradiction.
Therefore, we proved that $\{r_\ell,s_\ell\}\cap\{r_m,s_m\}=\emptyset$.
Now, we can obtain the geodesic $\cg'$ equivalent to $\cg$ by repeatedly moving the switching pairs after $(d,e)$ in $\cg$, which are not in $C$, to the position immediately before $(d,e)$ (in the same way as the construction of $\cg'$ as in the proof of Lemma~\ref{conex2.1}). 
Write
\begin{align*}
    S(\cg'):\; &(r'_1,s'_1),\ldots, (r'_{j'-1},s'_{j'-1}), (d,e),(r'_{j'+1},s'_{j'+1}), \ldots, (r'_{k-1}, s'_{k-1}).
\end{align*}
Then $(r'_{j'+1},s'_{j'+1}), \ldots, (r'_{k-1}, s'_{k-1})\in C$.

Since we only move switching pairs that are not in $C$ to the position before $(d,e)$, the switching pair $(d,e)$ remains the first switching pair in $S(\cg')$ among all switching pairs in $C_{NT}$. 
Let $L'=\cdots t''_1\cdots t''_{q}$ be the linear order that is immediately before the swap $(d,e)$ in $\cg'$. If $\{t''_1,\ldots,t''_q\}\ne\{t_1,\ldots,t_q\}$, then there are $t\in\{t_1,\ldots,t_q\}$ and $h\in \{t''_1,\ldots,t''_q\}$ such that $t\notin \{t''_1,\ldots,t''_q\}$  and $h\notin \{t_1,\ldots,t_q\}$. Then $L'_{\{t,h\}}=th$. 
Since $T_B=\cdots zt_1\cdots t_q$, we have $T_{\{t,h\}}=ht$ and so $(t,h)$ is a switching pair in $S(\cg)$. 
Since $\cg'$ is obtained from $\cg$ only by moving the switching pairs that appear after $(d,e)$ in $\cg$, the switching pair $(t,h)$ appears after $(d,e)$ in $\cg$. Hence, $L_{\{t,h\}}=th$, which contradicts $\{t'_1,\ldots,t'_q\}=\{t_1,\ldots,t_q\}$.
Therefore, $\{t''_1,\ldots,t''_q\}=\{t_1,\ldots,t_q\}$.
So we are in the same situation as in Lemma~\ref{conex2.1}, but in addition all switching pairs in $\cg'$ which occur after $(d,e)$ are in $C$.

Moreover, since $\cg'=(G'_1,\ldots,G'_k)$ is equivalent with $\cg=(G_1,\ldots,G_k)$, the set $\cd_{B\setminus\{z\}}\cup\{G'_1,\ldots, G'_k\}$ satisfies the same set of never-conditions as $\cd_{B\setminus\{z\}}\cup\{G_1,\ldots, G_k\}$ due to Lemma~\ref{wex1.3}\ref{wex1.3c}. In particular, it is a peak-pit Condorcet domain.
Therefore, $\cg'$ satisfies the requirements of the lemma.
\end{proof}

We improve once more the geodesic.

\begin{lemma}\label{conex2.3}
There exists a geodesic $\cg=(G_1,\ldots, G_k)$ on $B\setminus\{z\}$ connecting $R_{B\setminus\{z\}}$ and $T_{B\setminus\{z\}}$ satisfying Conditions $A_1-A_4$.
\end{lemma}
\begin{proof}
Let $\cg$ be as in Lemma~\ref{conex2.2} that satisfies Conditions $A_1-A_3$. Let $(d,e)$ be the first switching pair in $S(\cg)$ among all switching pairs in $C_{NT}$.
In particular, the linear order $L=\cdots t'_1\cdots t'_{q}$ in $\cg$ that is immediately before the swap $(d,e)$ gives $\{t'_1,\ldots, t'_{q}\}=\{t_1,\ldots, t_{q}\}$.
Write the sequence $S(\cg)$ of switching pairs in $\cg$ as
$$S(\cg):(r_1,s_1),\ldots, (r_{j-1},s_{j-1}), (d,e),(r_{j+1},s_{j+1}), \ldots, (r_{k-1}, s_{k-1}).$$
Again we write the switching pairs such that $R_{\{r_i,s_i\}}=r_is_i$ for all $i\in[k-1]$.
We will construct $\cg'$ such that for then in addition all switching pairs in $\cg'$, which are after $(d,e)$, are in $C_{NT}$ or not disjoint with at least one switching pair between $(d,e)$ and the given switching pair, including $(d,e)$. The construction is by induction.
To construct $\cg'$, we construct a set $M\subseteq \{1,\ldots, k-1-j\}$ such that for all $m\in M$, we shall move the switching pair $(r_{j+m},s_{j+m})$ before $(d,e)$. Write $(r_{j},s_{j})=(d,e)$. Define $M_0=\emptyset$. If $m\in\{1,\ldots, k-1-j\}$ and $M_{m-1}$ is defined, then $M_m=M_{m-1}\cup\{m\}$ if
\begin{align}\label{prop24e1}
   &(r_{j+m},s_{j+m})\notin C_{NT}\;\text{and}\;
   (r_{j+m},s_{j+m})\;\text{is disjoint from}\; (r_{j+\ell},s_{j+\ell})\;\nonumber\\&
   \text{for all}\; \ell\in\{0,\ldots,m-1\}\setminus M_{m-1}.
\end{align}
If (\ref{prop24e1}) is not satisfied, then define $M_m=M_{m-1}$. Inductively set $M=M_{k-1-j}$. Now
obtain the geodesic $\cg'$ equivalent to $\cg$ by repeatedly moving each $(r_{j+m}$,  $s_{j+m})$ with $m\in M$, starting with the lowest element of $M$, to the position immediately before $(d,e)$ (in the same way as the construction of $\cg'$ as in the proof of Lemma~\ref{conex2.1}). 
Write
\begin{align*}
    S(\cg'):\; &(r'_1,s'_1),\ldots, (r'_{j'-1},s'_{j'-1}),(d,e), (r'_{j'+1},s'_{j'+1}), \ldots, (r'_{k-1}, s'_{k-1})
\end{align*}
with $(r'_{j'+1},s'_{j'+1}), \ldots, (r'_{k-1}, s'_{k-1}) \in C$.
Then for all $m\in\{1,\ldots,k-1-j\}$ one has $(r'_{j'+m},s'_{j'+m})\in C_{NT}$ or $(r'_{j'+m},s'_{j'+m})$ is not disjoint with at least one element in $\{(d,e)$, $(r'_{j'+1},s'_{j'+1})$, $\ldots, (r'_{j'+m-1}$, $s'_{j'+m-1})\}$.

Since the switching pairs in $C_{NT}$ were not moved, the switching pair $(d,e)$ remains the first switching pair in $S(\cg')$ among all switching pairs in $C_{NT}$. 
Let $L'=\cdots t''_1\cdots t''_{q}$ be the linear order that is immediately before the swap $(d,e)$ in $\cg'$. If $\{t''_1,\ldots,t''_q\}\ne\{t_1,\ldots,t_q\}$, then there are $t\in\{t_1,\ldots,t_q\}$ and $h\in \{t'_1,\ldots,t'_q\}$ such that $t\notin \{t'_1,\ldots,t'_q\}$  and $h\notin \{t_1,\ldots,t_q\}$. Then $L'_{\{t,h\}}=th$.
Since $T_B=\cdots zt_1\cdots t_q$, we have $T_{\{t,h\}}=ht$ and so $(t,h)$ is a switching pair in $S(\cg')$. 
Since $\cg'$ is obtained from $\cg$ only by moving the switching pairs that appear after $(d,e)$ in $\cg$, the switching pair $(t,h)$ appears after $(d,e)$ in $\cg'$. Hence, $L_{\{t,h\}}=th$, which contradicts $\{t'_1,\ldots,t'_q\}=\{t_1,\ldots,t_q\}$ in Lemma~\ref{conex2.2}.
Therefore, $\{t''_1,\ldots,t''_q\}=\{t_1,\ldots,t_q\}$.

Moreover, since $\cg'=(G'_1,\ldots,G'_k)$ is equivalent with $\cg=(G_1,\ldots,G_k)$, the set $\cd_{B\setminus\{z\}}\cup\{G'_1,\ldots, G'_k\}$ satisfies the same set of never-conditions as $\cd_{B\setminus\{z\}}\cup\{G_1,\ldots, G_k\}$ due to Lemma~\ref{wex1.3}\ref{wex1.3c}. In particular, it is a peak-pit Condorcet domain.
Therefore, $\cg'$ satisfies the requirements of the lemma.
\end{proof}

We now continue the proof for Proposition~\ref{conex2}.
Let $\cg$ be as in Lemma~\ref{conex2.3}.
Let $(d,e)$ be the first switching pair in $S(\cg)$ among all switching pairs in $C_{NT}$.
Then the linear order $L=\cdots t'_1\cdots t'_{q}$ in $\cg=(G_1,\ldots, G_k)$ that is immediately before the swap $(d,e)$ gives $\{t'_1,\ldots, t'_{q}\}=\{t_1,\ldots, t_{q}\}$ 
by Condition $A_2$. 
Write
$$S(\cg): (r_1,s_1),\ldots, (r_{j-1},s_{j-1}),(d,e), (r_{j+1},s_{j+1}),\ldots, (r_{k-1}, s_{k-1}).$$
Define the path of alike linear orders $\cg^*$ by adding for all $p\in[q]$, the switching pair $(t_p,z)$ into the positions immediately before $(d,e)$ in $S(\cg)$. Then $\cg^*$ is characterised by 
\begin{align*}
    S(\cg^*):\; &(r_1,s_1),\ldots, (r_{j-1},s_{j-1}),(t'_q,z), \ldots, (t'_2,z), (t'_1,z),\\
    &(d,e),(r_{j+1},s_{j+1}), \ldots, (r_{k-1}, s_{k-1}).
\end{align*}
Recall that $(r_{j+1},s_{j+1}), \ldots, (r_{k-1}, s_{k-1}) \in C$ by Condition $A_3$.
Then for all $m\in\{1, \ldots, k-1-j\}$ one has $(r_{j+m},s_{j+m})\in C_{NT}$ or $(r_{j+m},s_{j+m})$ is not disjoint with at least one element in $\{(d,e), (r_{j+1},s_{j+1}), \ldots, (r_{j+m-1}, s_{j+m-1})\}$ by Condition $A_4$.
Also recall that we write the switching pairs such that $R_{\{r_i,s_i\}}=r_is_i$ for all $i\in[k-1]$.

Since $\{t'_1,\ldots, t'_{q}\}=\{t_1,\ldots, t_{q}\}$ and $T_B=\cdots zt_1\cdots t_q$, we have that $\cg^*$ is a path of alike linear orders by construction. 
Since all switching pairs in $\cg^*$ occur exactly once, $\cg^*$ is a geodesic connecting $R_B$ and $T_B$ due to Lemma~\ref{wex1.1}\ref{wex1.1b}.
It remains to show that $\cd_B\cup\{G^*_1,\ldots,G^*_{k'}\}$ is a peak-pit Condorcet domain, where $\cg^*=\{G^*_1,\ldots,G^*_{k'}\}$.

Let $a,b,c\in B$ be distinct alternatives. If $z\notin\{a,b,c\}$, then we have $(\cd_B\cup\{G^*_1,\ldots,G^*_{k'}\})_{\{a,b,c\}}$ is a peak-pit Condorcet domain by Condition $A_1$. 
Suppose $z\in\{a,b,c\}$, say $z=c$. If $(a,b)\notin C$ and $(b,a)\notin C$, then $R_{\{a,b,z\}}$ and $T_{\{a,b,z\}}$ differ by at most two swaps of adjacent alternatives. Then by Lemma~\ref{wex1.2}\ref{wex1.2b}, $(\cd_{B}\cup\{G^*_1,\ldots,G^*_{k'}\})_{\{a,b,z\}}$ is a peak-pit Condorcet domain.
Suppose, without loss of generality, that $(a,b)\in C$. 
If $(a,b)\vartriangleleft(d,e)$ in $S(\cg)$, then  $(a,b)\vartriangleleft(z,a)$ in $S(\cg^*)$. The dichotomy of Lemma~\ref{2geodesics} gives that the situation of 
Lemma~\ref{2geodesics}\ref{2geodesics3c} occurs. Hence, $\{G^*_1,\ldots,G^*_{k'}\}_{\{a,b,z\}}$ satisfies $bN_{\{a,b,z\}}3$. 
Since $(d,e)$ is the first element of $C_{NT}$ by condition $A_2$, it follows that $(a,b)\notin C_{NT}$. 
In particular, $N_p(\cd_{\{a,b,z\}})\ne\{bN_{\{a,b,z\}}3\}$.
Since $(a,b)\in C$, we have  $R_{\{a,b,z\}}=abz$ and $T_{\{a,b,z\}}=zba$.
Hence, $\emptyset\ne N_p(\cd_{\{a,b,z\}})\subseteq \{bN_{\{a,b,z\}}1$, $bN_{\{a,b,z\}}3\}$ due to Lemma~\ref{wex0}\ref{wex0a}.
Since $(a,b)\notin C_{NT}$, one has $N_p(\cd_{\{a,b,z\}})\ne \{bN_{\{a,b,z\}}1$ and so $bN_{\{a,b,z\}}\in N_p(\cd_{\{a,b,z\}})$. In particular, $\cd_{\{a,b,z\}}$ satisfies $bN_{\{a,b,z\}}3$.
Hence, $(\cd_B\cup\{G^*_1,\ldots,G^*_{k'}\})_{\{a,b,z\}}=(\cd_{\{a,b,z\}}\cup\{G^*_1,\ldots,G^*_{k'}\})_{\{a,b,z\}}$ satisfies $bN_{\{a,b,z\}}3$.

If $(d,e)\trianglelefteq(a,b)$ in $S(\cg)$, then $(z,a)\vartriangleleft(a,b)$ in $S(\cg^*)$. So now the dichotomy of Lemma~\ref{2geodesics} gives that the situation of 
Lemma~\ref{2geodesics}\ref{2geodesics3d} occurs. Hence, $\{G^*_1,\ldots,G^*_{k'}\}_{\{a,b,z\}}$ satisfies $bN_{\{a,b,z\}}1$.
We will next show that $\cd_{\{a,b,z\}}$ satisfies $bN_{\{a,b,z\}}1$ and therefore $(\cd_B\cup\{G^*_1,\ldots,G^*_{k'}\})_{\{a,b,z\}}$ satisfies $bN_{\{a,b,z\}}1$ and it is a peak-pit Condorcet domain, which will finish the induction step of the proof of Proposition~\ref{conex2}.
So the proof of Proposition~\ref{conex2} is complete once we have shown
that for every switching pair $(a,b)$ after $(d,e)$, the domain $\cd_{\{a,b,z\}}$ satisfies $bN_{\{a,b,z\}}1$.

Let $(a,b)$ be a switching pair after $(d,e)$ in $S(\cg)$.
For a contradiction suppose that $\cd_{\{a,b,z\}}$ does {\bf not} satisfy $bN_{\{a,b,z\}}1$.
Without loss of generality, we may assume that $(a,b)$ is the first switching pair after $(d,e)$ such that 
$\cd_{\{a,b,z\}}$ does not satisfy $bN_{\{a,b,z\}}1$.
If $(a,b)\in C_{NT}$, then $N_p(\cd_{\{a,b,z\}})=\{bN_{\{a,b,z\}}1\}$ by definition of $C_{NT}$. So $(a,b)\notin C_{NT}$.
Hence, there exists a switching pair \((f, g)\) such that \((d, e)\trianglelefteq(f,g)\vartriangleleft(a,b)\), and \((f, g)\in C\) by Condition $A_3$, and \(\{f, g\}\cap\{a,b\}\ne \emptyset\)  by Condition $A_4$. 
Because $(a,b)$ is the first switching pair after $(d,e)$ such that $bN_{\{a,b,z\}}1\notin N_p(\cd_{\{a,b,z\}})$, 
we have $gN_{\{f,g,z\}}1\in N_p(\cd_{\{f,g,z\}})$.

Now $(f,g)\in C$ and so $fgz=R_{\{f,g,z\}}\in \cd_{\{f,g,z\}}$ and $zgf=T_{\{f,g,z\}}\in \cd_{\{f,g,z\}}$. 
Hence, $\emptyset\ne N_p(\cd_{\{f,g,z\}})\subseteq \{gN_{\{f,g,z\}}1$, $gN_{\{f,g,z\}}3\}$ due to Lemma~\ref{wex0}\ref{wex0a}. 
Since $gN_{\{f,g,z\}}1\in N_p(\cd_{\{f,g,z\}})$, we have 
\begin{equation}\label{fgz}
    \cd_{\{f,g,z\}}\subseteq\{fgz,fzg,zfg,zgf\}.
\end{equation}

On the other hand, since $bN_{\{a,b,z\}}1\notin N_p(\cd_{\{a,b,z\}})$, there is a linear order $W\in\cd$ such that 
\begin{equation}\label{abz}
    W_{\{a,b,z\}}=baz\;\text{or}\; W_{\{a,b,z\}}=bza.
\end{equation}
Note that in either case, $W$ ranks $b$ before both $a$ and $z$.

Recall that $(f,g)\vartriangleleft(a,b)$ and so $\{f,g\}\ne\{a,b\}$. 
In particular, $|\{f,g\}\cap\{a,b\}|=1$. Write $\{f,g\}\cup\{a,b\}=\{f,g,s\}$. 
Now since $z$ is ranked last in $R_B$, there are only three possibilities for $R_{\{f,g,s,z\}}$: $sfgz, fsgz$ and $fgsz$. 
We will show that all cases lead to a contradiction so that $bN_{\{a,b,z\}}1\in N_p(\cd_{\{a,b,z\}})$.

{\bf Case 2.1.} Suppose $R_{\{f,g,s,z\}}=sfgz$. 
Then $R_{\{f,g,s\}}=sfg$. 
There is at least the swap $(f,g)$ in $S(\cg_{\{f,g,s\}})$. 
Note that the swap $(s,f)$ may or may not appear in $S(\cg_{\{f,g,s\}})$.
Since $s$ and $f$ are adjacent and $f$ and $g$ are adjacent in $R_{\{f,g,s\}}$, either $(s,f)$ appears first, if at all, in $S(\cg_{\{f,g,s\}})$ or $(f,g)$ appears first.
If $(s,f)$ appears first, then $g$ needs to swap with $s$ before it can swap with $f$.
So we have $(s,f)\vartriangleleft(s,g)\vartriangleleft(f,g)$ in $S(\cg_{\{f,g,s\}})$. Since $(a,b)=(s,f)$ or $(a,b)=(s,g)$ as swaps, this
contradicts the assumption that $(f,g)$ appears before $(a,b)$.
Hence, $(f,g)$ appears first in $S(\cg_{\{f,g,s\}})$. 
There are two cases for $(a,b)$.

{\bf Case 2.1.1.} Suppose $(a,b)=(s,g)$. 
Since $(f,g)\in C$, we have  $T_{\{f,g,z\}}=zgf$. Since $(s,g)$ is in $\cg$ and $R_{\{g,s\}}=sg$, we have $T_{\{g,s\}}=gs$.
Therefore, $T_{\{f,g,s,z\}}$ is $zgsf$ or $zgfs$. 
The difference is whether or not $(s,f)$ is a switching pair in $\cg$, as shown in Figure~\ref{ecp2}. 
Note that $(f,g)\vartriangleleft(s,g)$ and so $(f,g)$ appears first in $S(\cg_{\{f,g,s\}})$ by Lemma~\ref{2geodesics3}\ref{2geodesics3f}.
Since $gN_{\{f,g,z\}}1\in N_p(\cd_{\{f,g,z\}})$, we have $gN_{\{g,s,z\}}1\in N_p(\cd_{\{g,s,z\}})$ by Lemma~\ref{c211}.

\begin{figure}[H]
 \centering
\begin{tikzpicture}[scale=0.60]

\draw[fill=black] (3,0) circle (3pt);
\draw[fill=black] (3,2) circle (3pt);
\draw[fill=black] (3,4) circle (3pt);
\draw[fill=black] (3,6) circle (3pt);
\draw[fill=black] (19,0) circle (3pt);
\draw[fill=black] (19,2) circle (3pt);
\draw[fill=black] (19,4) circle (3pt);
\draw[fill=black] (19,6) circle (3pt);
\draw[fill=black] (22,0) circle (3pt);
\draw[fill=black] (22,2) circle (3pt);
 
\node at (2.5,0) {$z$};
\node at (2.5,2) {$g$};
\node at (2.5,4) {$f$};
\node at (2.5,6) {$s$};
\node at (19.5,0) {$f$};
\node at (19.5,2) {$s$};
\node at (19.5,4) {$g$};
\node at (19.5,6) {$z$};
\node at (22.5,0) {$s$};
\node at (22.5,2) {$f$};

\node at (6.5,1) {$(g,z)$};
\node at (9.5,3) {$(f,z)$};
\node at (12.5,5) {$(s,z)$};
\node at (15.5,1) {$(f,g)$};
\node at (18.5,3) {$(s,g)$};
\node at (21.5,1) {$(s,f)$};

\draw[thick,dashed] (3,0) -- (4,0)--(6,2)--(7,2) --(9,4)--(10,4)--(12,6)--(19,6);
\draw[thick] (3,2)--(4,2) --(6,0)--(13,0)--(15,2)--(16,2)--(18,4)--(19,4);
\draw[thick] (3,4)--(7,4)--(9,2)--(13,2)--(15,0)--(19,0);
\draw[thick,dotted] (19,0)--(21,2)--(22,2);
\draw[thick,dotted] (19,2)--(21,0)--(22,0);
\draw[thick] (3,6)--(10,6)--(12,4)--(16,4)-- (18,2) -- (19,2);

\end{tikzpicture}
    \caption{Case 2.1.} \label{ecp2}
\end{figure}
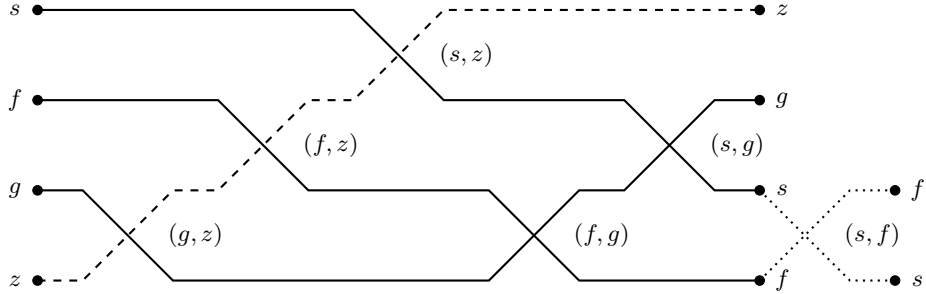

{\bf Case 2.1.2.} Suppose $(a,b)=(s,f)$.
Since $(s,f)$ appears after $(f,g)$, we have $(f,g)\vartriangleleft(s,g)\vartriangleleft(s,f)$ in $S(\cg_{\{f,g,s\}})$ due to Lemma~\ref{2geodesics3}\ref{2geodesics3d}.
Since $R_{\{f,g,s\}}=sfg$, we have $T_{\{f,g,s\}}=gfs$. Since $(f,g)\in C$, we have $T_{\{f,g,z\}}=zgf$.
Therefore, $T_{\{f,g,s,z\}}=zgfs$. 

Now (\ref{abz}) gives that $W_{\{f,s,z\}}$ is either $fsz$ or $fzs$, where we recall that $(a,b)=(s,f)$ and $W\in \cd$.
In particular, $W_{\{f,s\}}=fs$ and $W_{\{f,z\}}=fz$.
Since $\cd_{\{f,g,z\}}\subseteq\{fgz,fzg,zfg,zgf\}$ by (\ref{fgz}), the only linear orders that can possibly be in $\cd_{\{f,g,z\}}$ and ranks $f$ before $z$ are $fgz$ and $fzg$. So $W_{\{f,g,z\}}=fgz$ or $W_{\{f,g,z\}}=fzg$. Nevertheless, $W$ always ranks $f$ before $g$.
Recall that $W$ always ranks $f$ before $s$.
We have $W_{\{f,g,s\}}$ is either $fgs$ or $fsg$. 

Recall $R_{\{f,g,s\}}=sfg$, $T_{\{f,g,s\}}=gfs$ and $(f,g)$ appears first in $S(\cg_{\{f,g,s\}})$.
So we have a never-top geodesic $\cg_{\{f,g,s\}}=(sfg,sgf,gsf,gfs)$ by Lemma~\ref{2geodesics}\ref{2geodesicsb}.
However, $W_{\{f,g,s\}}\in\{fgs,fsg\}$ and both $\{fgs\}\cup \{sfg$, $sgf$, $gsf$, $gfs\}$ and $\{fsg\}\cup \{sfg$, $sgf$, $gsf$, $gfs\}$ do not satisfy any never-condition.
Since $W_{\{f,g,s\}}\in\cd_{\{f,g,s\}}$, we have a contradiction that $\{W_{\{f,g,s\}}\}\cup \{ sfg$, $sgf$, $gsf$, $gfs\}\subset(\cd_{B\setminus\{z\}}\cup\{G_1,\ldots, G_k\})_{\{f,g,s\}}$ is a peak-pit Condorcet domain.



{\bf Case 2.2.} Suppose $R_{\{f,g,s,z\}}=fsgz$. Then $R_{\{f,g,s\}}=fsg$. 
There is at least the swaps $(f,g)$ and $(a,b)$ in $S(\cg_{\{f,g,s\}})$ with $(f,g)\vartriangleleft(a,b)$.
Since $s$ is ranked between $f$ and $g$ in $R$,
there is another swap before the swap $(f,g)$. So the geodesic $\cg_{\{f,g,s\}}$ contains at least three swaps.
By Lemma~\ref{2geodesics3}, one has either $(f,s)\vartriangleleft(f,g)\vartriangleleft(s,g)$ or $(s,g)\vartriangleleft(f,g)\vartriangleleft(f,s)$
in $S(\cg_{\{f,g,s\}})$.
If $(f,s)\vartriangleleft(f,g)\vartriangleleft(s,g)$ in $S(\cg_{\{f,g,s\}})$, then $(a,b)=(s,g)$ and we have $gN_{\{g,s,z\}}1\in N_p(\cd_{\{g,s,z\}})$ by Lemma~\ref{c22}\ref{c221}.
If $(s,g)\vartriangleleft(f,g)\vartriangleleft(f,s)$ in
$S(\cg_{\{f,g,s\}})$, then $(a,b)=(f,s)$ and we have $sN_{\{f,s,z\}}1\in N_p(\cd_{\{f,s,z\}})$ by Lemma~\ref{c22}\ref{c222}.


{\bf Case 2.3.} Suppose $R_{\{d,e,s,z\}}=fgsz$. 
Then $R_{\{f,g,s\}}=fgs$. There is at least the swap $(f,g)$ in $S(\cg_{\{f,g,s\}})$.
Note that the swap $(g,s)$ may or may not appear in $S(\cg_{\{f,g,s\}})$.
Since $s$ and $g$ are adjacent and $f$ and $g$ are adjacent in $R_{\{f,g,s\}}$, either $(g,s)$ appears first, if at all, in $S(\cg_{\{f,g,s\}})$ or $(f,g)$ appears first.
If $(g,s)$ appears first, then $f$ needs to swap with $s$ before it can swap with $g$. So we have $(g,s)\vartriangleleft(f,s)\vartriangleleft(f,g)$ in $S(\cg_{\{f,g,s\}})$. Since $(a,b)=(f,s)$ or $(a,b)=(g,s)$ as swaps, this contradicts the assumption that $(f,g)$ appears before $(a,b)$.
Hence, $(f,g)$ appears first in $S(\cg_{\{f,g,s\}})$, as shown in Figure~\ref{ecp4}.
There are two cases for $(a,b)$.

\begin{figure}[H]
 \centering
\begin{tikzpicture}[scale=0.60]

\draw[fill=black] (3,0) circle (3pt);
\draw[fill=black] (3,2) circle (3pt);
\draw[fill=black] (3,4) circle (3pt);
\draw[fill=black] (3,6) circle (3pt);
\draw[fill=black] (19,0) circle (3pt);
\draw[fill=black] (19,2) circle (3pt);
\draw[fill=black] (19,4) circle (3pt);
\draw[fill=black] (19,6) circle (3pt);
\draw[fill=black] (22,2) circle (3pt);
\draw[fill=black] (22,4) circle (3pt);
 
\node at (2.5,0) {$z$};
\node at (2.5,2) {$s$};
\node at (2.5,4) {$g$};
\node at (2.5,6) {$f$};
\node at (19.5,0) {$f$};
\node at (19.5,2) {$s$};
\node at (19.5,4) {$g$};
\node at (19.5,6) {$z$};
\node at (22.5,2) {$g$};
\node at (22.5,4) {$s$};

\node at (6.5,1) {$(s,z)$};
\node at (9.5,3) {$(g,z)$};
\node at (12.5,5) {$(f,z)$};
\node at (18.5,1) {$(f,s)$};
\node at (15.5,3) {$(f,g)$};
\node at (21.5,3) {$(g,s)$};

\draw[thick,dashed] (3,0) -- (4,0)--(6,2)--(7,2) --(9,4)--(10,4)--(12,6)--(19,6);
\draw[thick] (3,2)--(4,2) --(6,0)--(16,0)--(18,2)--(19,2);
\draw[thick] (3,4)--(7,4)--(9,2)--(13,2)--(15,4)--(19,4);
\draw[thick,dotted] (19,4)--(21,2)--(22,2);
\draw[thick,dotted] (19,2)--(21,4)--(22,4);
\draw[thick] (3,6)--(10,6)--(12,4)--(13,4)--(15,2)--(16,2)-- (18,0) -- (19,0);

\end{tikzpicture}
    \caption{Case 2.3.} \label{ecp4}
\end{figure}
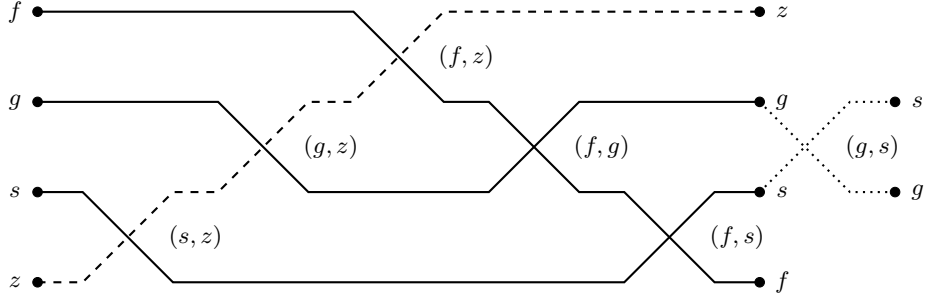 

{\bf Case 2.3.1.} Suppose $(a,b)=(g,s)$.
Since $(f,s)$ appears after $(f,g)$, we have $(f,g)\vartriangleleft(f,s)\vartriangleleft(g,s)$ in $S(\cg_{\{f,g,s\}})$ due to Lemma~\ref{2geodesics3}\ref{2geodesics3d}.
Moreover, $R_{\{f,g,s\}}=fgs$ and so $T_{\{f,g,s\}}=sgf$. Because $(f,g)\in C$, we have $T_{\{f,g,z\}}=zgf$.
Therefore, $T_{\{f,g,s,z\}}=zsgf$. 

Now (\ref{abz}) gives that $W_{\{g,s,z\}}$ is either $sgz$ or $szg$, where we recall that $(a,b)=(g,s)$ and $W\in \cd$.
In particular, $W_{\{g,s\}}=sg$ and $W_{\{s,z\}}=sz$.
On the other hand, recall that $(g,s)=(a,b)$ is the first switching pair after $(d,e)$ such that $\cd_{\{a,b,z\}}$ does not satisfy $bN_{\{a,b,z\}}1$.
Since $(f,s)$ appears between $(f,g)$ and $(g,s)$, it appears between $(d,e)$ and $(a,b)$. So $sN_{\{f,s,z\}}1\in N_p(\cd_{\{f,s,z\}})$.

Recall $(f,s)\in C$ by $A_3$ and so $fsz=R_{\{f,s,z\}}\in \cd_{\{f,s,z\}}$ and $zsf=T_{\{f,s,z\}}\in \cd_{\{f,s,z\}}$. 
Since $sN_{\{f,s,z\}}1\in N_p(\cd_{\{f,s,z\}})$, 
we have $\cd_{\{f,s,z\}}\subseteq\{fsz$, $fzs$, $zfs$, $zsf\}$. 
The only linear order that can possibly be in $\cd_{\{f,s,z\}}$ and ranks $s$ before $z$ is $fsz$ and so $W_{\{f,s,z\}}=fsz$.
Recall that $W_{\{g,s,z\}}$ is either $sgz$ or $szg$.
We have $W_{\{f,g,s,z\}}$ is either $fsgz$ or $fszg$.
In both cases, $W_{\{f,g,s\}}=fsg$. 

Recall $R_{\{f,g,s\}}=fgs$, $T_{\{f,g,s\}}=sgf$ and $(f,g)$ appears first in $S(\cg_{\{f,g,s\}})$.
So we have a never-top geodesic $\cg_{\{f,g,s\}}=(fgs,gfs,gsf,sgf)$ by Lemma~\ref{2geodesics}\ref{2geodesicsb}.
Since $fsg=W_{\{f,g,s\}}\in\cd_{\{f,g,s\}}$, we have a contradiction that $\{fsg\}\cup \{ fgs$, $gfs$, $gsf$, $sgf\}\subset(\cd_{B\setminus\{z\}}\cup\{G_1,\ldots, G_k\})_{\{f,g,s\}}$ is a peak-pit Condorcet domain.

{\bf Case 2.3.2.} Suppose $(a,b)=(f,s)$. 
Since $(f,g)\in C$, we have  $T_{\{f,g,z\}}=zgf$. Since $(f,s)$ is in $\cg$ and $R_{\{f,s\}}=fs$, we have $T_{\{f,s\}}=sf$.
Moreover, $s\in\{a,b\}\in\{t_1,\ldots t_q\}^2$ and $T_B=\cdots zt_1\cdots t_q$, so that $T_{\{s,z\}}=zs$.
Therefore, $T_{\{f,g,s,z\}}$ is $zgsf$ or $zsgf$. 
The difference is whether or not $(g,s)$ is a switching pair in $\cg$, as shown in Figure~\ref{ecp4}.
Note that $(f,g)\vartriangleleft(s,g)$ and so $(f,g)$ appears first in $S(\cg_{\{f,g,s\}})$ by Lemma~\ref{2geodesics3}\ref{2geodesics3a}.


Write $R_B=\cdots fg_1\cdots g_m s\cdots$, where $g_1,\ldots, g_m$ are the alternatives that are ranked between $f$ and $s$ in $R_B$, for some $m$.
Note that $g$ is ranked between $f$ and $s$ and so $m\geq1$.
Denote $$F=\{f,g_1,\ldots,g_m\}.$$ 
It can be proved that all switching pairs of alternatives from $F$ between $(d,e)$ and $(f,s)$, and which are not disjoint from $(f,s)$ will be elements of the set $G$ that we next define.
For the proof of Case 2.3.2, we do not need this characterisation of $G$, but we will use the set $G$ to provide a relatively short proof for Case 2.3.2.

Define two sets of switching pairs from $S(\cg)$ as follows:
\begin{align*}
    G_1=&\;\{(f,k):k\in F\setminus\{f\},\; (d,e)\trianglelefteq(f,k)\vartriangleleft(f,s)\;\text{in}\;S(\cg)\},
\end{align*}
and
\begin{align*}
    G_2=&\;\{(h,k):h,k\in F\setminus\{f\},\; h\ne k,\; \\
&(d,e)\trianglelefteq(h,k)\vartriangleleft(f,k)\vartriangleleft(f,h)\vartriangleleft(f,s)\;\text{in}\;S(\cg)\}.
\end{align*}
Moreover, let $G=G_1\cup G_2$. Note that $g$ is ranked between $f$ and $s$ in $R_B$, and
$(d,e)\trianglelefteq(f,g)\vartriangleleft(f,s)$ in $S(\cg)$.
Recall that $R_{\{f,g,z\}}=fgz$ and $T_{\{f,g,z\}}=zgf$ and so $(f,g)\in G_1\subseteq G$. Hence, $G\ne\emptyset$.
Note that for all $(f,k)\in G_1$, where $k\in F\setminus\{f\}$, we have $(f,k)\in C$ by Condition $A_3$.
Hence, $R_{\{f,k,s,z\}}=fksz$ and $T_{\{f,k,z\}}=zkf$.
Moreover, for all $(h,k)\in G_2$, the switching pairs $(f,k)$, $(f,h)$ are in $S(\cg)$. So $R_{\{f,h,k,s,z\}}=fhksz$ and $T_{\{f,h,k,z\}}=zkhf$.
Further, suppose $(h,k)\in G$ and let $t\in\{h,k\}$ with $t\ne f$. Then note that $(f,t)$ occurs in $S(\cg)$. Moreover, $(h,k)\trianglelefteq(f,t)\vartriangleleft(f,s)$ in $S(\cg)$.

We shall show that $kN_{\{h,k,z\}}3\in N_p(D_{\{h,k,z\}})$ for all $(h,k)\in G$. Further, we show that there is an element $(h',k')\in G$ strictly before $(h,k)$ such that also $k'N_{\{h',k',z\}}3\in N_p(D_{\{h',k',z\}})$. This is impossible, because there cannot be a first element in $G$.
Let $(h,k)\in G$, we will show that $kN_{\{h,k,z\}}3\in N_p(D_{\{h,k,z\}})$.
Note that $kN_{\{h,k,z\}}1\in N_p(D_{\{h,k,z\}})$ by the definition of $(a,b)$.
Suppose $(h,k)\in G_1$. Then $h=f$.
Suppose $kN_{\{f,k,z\}}3\notin N_p(D_{\{f,k,z\}})$ for a contradiction.
Since $R_{\{f,k,s,z\}}=fksz$ and $T_{\{f,k,z\}}=zkf$, we have $\emptyset\ne N_p(\cd_{\{f,k,z\}})\subseteq \{kN_{\{f,k,z\}}1$, $kN_{\{f,k,z\}}3\}$ by Lemma~\ref{wex0}\ref{wex0a}. 
Because $kN_{\{f,k,z\}}3\notin N_p(D_{\{f,k,z\}})$, therefore we have $N_p(\cd_{\{f,k,z\}})=\{kN_{\{f,k,z\}}1\}$.
Since $T_{\{f,k,z\}}=zkf$ and $T_{\{f,s,z\}}=zsf$, we have
$T_{\{f,k,s,z\}}\in\{zksf,zskf\}$.
Moreover, $(f,k)\vartriangleleft(f,s)$ in $S(\cg_{\{f,k,s\}})$. If $(f,k)$ is not the first switching pair in $S(\cg_{\{f,k,s\}})$, then only available option for the first switching pair is $(k,s)$.
Then, we have $(k,s)\vartriangleleft(f,k)\trianglelefteq(f,s)$ in $S(\cg_{\{f,k,s\}})$, which contradicts the dichotomy of Lemma~\ref{2geodesics} for $\cg_{\{f,k,s\}}$.
Therefore, $(f,k)$ is the first switching pair in $S(\cg_{\{f,k,s\}})$.
Hence, $sN_{\{f,s,z\}}1\in N_p(\cd_{\{f,s,z\}})$ by Lemma~\ref{fgz2}.
But $bN_{\{a,b,z\}}1\notin N_p(\cd_{\{a,b,z\}})$ and $(a,b)=(f,s)$. This is a contradiction.

On the other hand, suppose $(h,k)\in G_2$. Then $(d,e)\trianglelefteq(h,k)\vartriangleleft(f,k)\vartriangleleft(f,h)\vartriangleleft(f,s)$  in $S(\cg)$ with
$R_{\{f,h,k,s,z\}}=fhksz$ and $T_{\{f,h,k,z\}}=zkhf$.
Hence, we have $kN_{\{h,k,z\}}3\in N_p(D_{\{h,k,z\}})$ by Lemma~\ref{c232.1}.
We proved that $kN_{\{h,k,z\}}3\in N_p(D_{\{h,k,z\}})$ for all $(h,k)\in G$. 

Now fix $(h,k)\in G$ be the switching pair that appears first in $S(\cg)$ among all elements in $G$.
Then $kN_{\{h,k,z\}}3\in N_p(D_{\{h,k,z\}})$.
In particular, $(h,k)\notin C_{NT}$. Hence by Condition $A_4$, there exists a switching pair $(h',k')$ in $S(\cg)$ such that $(d,e)\trianglelefteq(h',k')\vartriangleleft(h,k)$ and $|\{h',k'\}\cap\{h,k\}|=1$.
Recall the convention that we write the switching pairs such that $(h',k')$ such that $R_{\{h',k'\}}=h'k'$.
Now $(h,k)\in G$, so $(h,k)\vartriangleleft(f,s)=(a,b)$.
Therefore by the definition of $(a,b)$, we have $k'N_{\{h',k',z\}}1\in N_p(D_{\{h',k',z\}})$.
Moreover, $(h',k')\in C$ by Condition $A_3$ and so $R_{\{h',k',z\}}=h'k'z$ and $T_{\{h',k',z\}}=zk'h'$.
Write $\{h',k'\}\cup\{h,k\}=\{h,k,u\}$.  We first show that $u\notin F$. For a contradiction suppose $u\in F$.
Since $h,k\in F$, we have $\{h',k'\}\subseteq\{h,k,u\}\subseteq F$.
There are two cases.
 
{\bf Case A.} Suppose $f\in \{h',k'\}$. 
Recall that $R_{\{h',k'\}}=h'k'$. If $k'=f$, then $h'\notin F$ which contradicts the assumption. Hence, $h'=f$. Moreover, $R_{\{k',s\}}=k's$ since $k'\in F$.
In particular, $R_{\{f,k',s,z\}}=R_{\{h',k',s,z\}}=h'k'sz=fk'sz$ and $T_{\{f,k',z\}}=T_{\{h',k',z\}}=zk'h'=zk'f$.
Now $(f,k')=(h',k')\vartriangleleft(h,k)\vartriangleleft(f,s)$ and so $(f,k')\in G_1$. This contradicts the definition of $(h,k)$ being the first switching pair in $S(\cg)$ among all elements of $G$.

{\bf Case B.} Suppose $f\notin \{h',k'\}$. Since $h',k'\in F$, we have $R_{\{f,h',k',s\}}=fh'k's$. We will first show that $(h',k')$ is the first switching pair in $S(\cg_{\{f,h',k',s\}})$.
Note that there are three pairs of adjacent alternatives in $R_{\{f,h',k',s\}}$.
Hence, the first switching pair in $S(\cg_{\{f,h',k',s\}})$ can only be $(f,h')$, $(h',k')$ or $(k',s)$.

Suppose $(f,h')$ is the first switching pair in $S(\cg_{\{f,h',k',s\}})$. Then $(f,h')\vartriangleleft(h',k')$.
Hence, $(f,h')\vartriangleleft(f,k')\vartriangleleft(h',k')$
by Lemma~\ref{2geodesics3}\ref{2geodesics3c}.
Now if $(h,k)\in G_1$, then $f=h$ and $(h',k')\vartriangleleft(h,k)=(f,k)$.
Since $f\notin \{h',k'\}$ and $|\{h',k'\}\cap\{h,k\}|=1$, either $(f,h')$ or $(f,k')$ equals $(f,k)$.
We have a contradiction with the strict ordering $(f,h')\vartriangleleft(f,k')\vartriangleleft(h',k')\vartriangleleft(f,k)$.
Likewise, if $(h,k)\in G_2$, then $(f,h')\vartriangleleft(f,k')\vartriangleleft(h',k')\vartriangleleft(h,k)\vartriangleleft(f,k)\vartriangleleft(f,h)$. But $|\{h',k'\}\cap\{h,k\}|=1$.
This gives a contradiction with the strict ordering.
Hence, $(f,h')$ cannot be the first switching pair in $S(\cg_{\{f,h',k',s\}})$.

Next suppose $(k',s)$ is the first switching pair in $S(\cg_{\{f,h',k',s\}})$.
Then $(k',s)\vartriangleleft(h',k')$.
Hence, $(k',s)\vartriangleleft(h',s)\vartriangleleft(h',k')$
by Lemma~\ref{2geodesics3}\ref{2geodesics3d}.
Consequently, $(k',s)\vartriangleleft(h',s)\vartriangleleft(h',k')\vartriangleleft(h,k)\vartriangleleft(f,s)$.
Now $|\{h',k'\}\cap\{h,k\}|=1$. Write $\{t\}=\{h',k'\}\cap\{h,k\}$. Recall $f\notin\{h',k'\}$, so $f\ne t$. Therefore, $(f,t)$ occurs in $S(\cg)$. Moreover, $(h,k)\trianglelefteq(f,t)\vartriangleleft(f,s)$.
Suppose $t=h'$. Then $(f,h')$ occurs in $S(\cg)$ and $(h,k)\trianglelefteq(f,h')\vartriangleleft(f,s)$.
Then $(h',s)\vartriangleleft(h,k)\trianglelefteq(f,h')\vartriangleleft(f,s)$ which contradicts the dichotomy of Lemma~\ref{2geodesics} for $\cg_{\{f,h',s\}}$.
Suppose $t=k'$. Then $(f,k')$ occurs in $S(\cg)$ and $(f,k')\vartriangleleft(f,s)$ in $S(\cg_{\{f,k',s\}})$.
Then $(f,k')$ is the first switching pair in $S(\cg_{\{f,k',s\}})$ by Lemma~\ref{2geodesics3}\ref{2geodesics3a}.
This contradicts that $(h',s)$ is the first switching pair.
Hence, $(k',s)$ cannot be the first switching pair in $S(\cg_{\{f,h',k',s\}})$.

Therefore, $(h',k')$ is the first switching pair.
Now $R_{\{f,h',s\}}=fh's$ and note that $(f,s)$ is in $S(\cg_{\{f,h',s\}})$. The first switching pair in $S(\cg_{\{f,h',s\}})$ can only be $(f,h')$ or $(h',s)$. 
Suppose $(h',s)$ is the first switching pair in $S(\cg_{\{f,h',s\}})$.
Then $(h',s)\vartriangleleft(f,s)$ in $S(\cg_{\{f,h',s\}})$.
Because $(h',k')$ is the first switching pair in $S(\cg_{\{f,h',k',s\}})$, we have $(h',k')\vartriangleleft(h',s)\vartriangleleft(f,s)$.
In particular, $(d,e)\trianglelefteq(h',k')\vartriangleleft(h',s)\vartriangleleft(f,s)=(a,b)$, and so $sN_{\{h',s,z\}}1\in N_p(\cd_{\{h',s,z\}})$ by the definition of $(a,b)$.
Moreover, $R_{\{f,h',s,z\}}=fh'sz$. Because $(h',s)$ and $(f,s)$ are switching pairs in $S(\cg_{\{f,h',s\}})$, we have $T_{\{f,h',s\}}\in\{sh'f, sfh'\}$.
Since $T_{\{f,s,z\}}=zsf$, we have $T_{\{f,h',s,z\}}\in\{zsh'f, zsfh'\}$.
Therefore, $sN_{\{f,s,z\}}1\in N_p(\cd_{\{f,s,z\}})$ by Lemma~\ref{c211}. But $bN_{\{a,b,z\}}1\notin N_p(\cd_{\{a,b,z\}})$ and $(a,b)=(f,s)$. This is a contradiction.
Therefore, $(f,h')$ appears in $S(\cg)$ and it is the first switching pair in $S(\cg_{\{f,h',s\}})$. 
Hence, $(f,h')\vartriangleleft(f,s)$.
The previous argument also works with $h'$ replaced by $k'$. So also $(f,k')$ appears in $S(\cg)$ and it is the first switching pair in $S(\cg_{\{f,k',s\}})$. Consequently, $(f,k')\vartriangleleft(f,s)$.
Now $(h',k')$ is the first switching pair in $S(\cg_{\{f,h',k'\}})$, and so we have  $(h',k')\vartriangleleft(f,k')\vartriangleleft(f,h')$ by Lemma~\ref{2geodesics3}\ref{2geodesics3d}.
Recall $R_{\{f,h',k'\}}=fh'k'$. Therefore, $T_{\{f,h',k'\}}=k'h'f$ and then $T_{\{f,h',k',z\}}=zk'h'f$.
Together, we have $(d,e)\vartriangleleft(h',k')\vartriangleleft(f,k')\vartriangleleft(f,h')\vartriangleleft(f,s)$ in $S(\cg)$.
So $(h',k')\in G_2$. This contradicts the definition of $(h,k)$ being the first switching pair in $S(\cg)$ among all elements of $G$.

We proved that $u\notin F$. 
Write $\{t\}=\{h',k'\}\cap\{h,k\}$. Then $t\in F$ and $t\ne u$. Moreover, $t\in\{h',k'\}$ and $u\in\{h',k'\}$. 
So $\{t,u\}\in\{h',k'\}$
By definition of $F$, this means $u$ is ranked before $f$ or after $s$ in $R_B$. Since $R_{\{f,s,z\}}=fsz$ and $z$ is ranked last in $R_B$, we have $R_{\{f,s,u,z\}}=ufsz$ or $R_{\{f,s,u\}}=fsuz$.
We will show both cases are not possible.

{\bf Case 2.3.2.1.} Suppose $R_{\{f,s,u,z\}}=ufsz$. 
Suppose $t=f$. Then $(u,f)=(h',k')\vartriangleleft(h,k)\vartriangleleft(f,s)$.
Hence, Lemma~\ref{2geodesics3}\ref{2geodesics3c} gives $(u,s)$ is a switching pair in $S(\cg)$ and $(u,f)\vartriangleleft(u,s)\vartriangleleft(f,s)$ in $S(\cg_{\{f,s,u\}})$.
Moreover, $T_{\{f,s,u\}}=sfu$.
Then $T_{\{f,s,z\}}=zsf$ gives that $T_{\{f,s,u,z\}}=zsfu$.
Now $(d,e)\trianglelefteq(h',k')=(u,f)=\vartriangleleft(u,s)\vartriangleleft(f,s)=(a,b)$. We have
$sN_{\{s,u,z\}}1\in N_p(\cd_{\{s,u,z\}})$ by the definition of $(a,b)$.
Hence, $sN_{\{f,s,z\}}1\in N_p(\cd_{\{f,s,z\}})$ by Lemma~\ref{c22}\ref{c221}.
But $bN_{\{a,b,z\}}1\notin N_p(\cd_{\{a,b,z\}})$ and $(a,b)=(f,s)$. This is a contradiction.

Hence, $t\ne f$. Recall $t\in F$, so $R_{\{f,s,t\}}=fts$.
Also $R_{\{f,u\}}=uf$. Therefore, $R_{\{f,s,t,u\}}=ufts$.
Note that $\{h',k'\}\cap\{h,k\}=\{t\}$ and use that $(h,k)\in G$, we have $(h',k')\vartriangleleft(h,k)\trianglelefteq(f,t)\vartriangleleft(f,s)$.
Hence, $(u,t)=(h',k')\vartriangleleft(f,t)\vartriangleleft(f,s)$.
Therefore, $(u,f)$ and $(u,s)$ are in $S(\cg)$ and $(u,f)\vartriangleleft(u,t)\vartriangleleft(u,s)\vartriangleleft(f,s)$ by
Lemma~\ref{2geodesics4}\ref{2geodesics4a}.
Consequently, since $R_{\{f,s,u\}}=ufs$ and there are three switching pairs in $S(\cg_{\{f,s,u\}})$, we have $T_{\{f,s,u\}}=sfu$.
Then $T_{\{s,z\}}=zs$ gives that $T_{\{f,s,u,z\}}=zsfu$.
Now $(d,e)\trianglelefteq(h',k')=(u,t)\vartriangleleft(u,s)\vartriangleleft(f,s)=(a,b)$.
Therefore, $sN_{\{s,u,z\}}1\in N_p(\cd_{\{s,u,z\}})$ by the definition of $(a,b)$.
Hence, $sN_{\{f,s,z\}}1\in N_p(\cd_{\{f,s,z\}})$ by Lemma~\ref{c22}\ref{c221}.
But $bN_{\{a,b,z\}}1\notin N_p(\cd_{\{a,b,z\}})$ and $(a,b)=(f,s)$. This is a contradiction.

{\bf Case 2.3.2.2.} Suppose $R_{\{f,s,u,z\}}=fsuz$. 
Suppose $t=f$. Then $(f,u)=(h',k')\vartriangleleft(h,k)\vartriangleleft(f,s)$.
Now $R_{\{f,s,u\}}=fsu$. We have  $(s,u)$ is in $S(\cg)$ and $(s,u)\vartriangleleft(f,u)\vartriangleleft(f,s)$ in $S(\cg_{\{f,s,u\}})$ by Lemma~\ref{2geodesics3}\ref{2geodesics3b}.
Moreover, 
$T_{\{f,s,u\}}=usf$.
Together with $T_{\{u,z\}}=zu$, we have $T_{\{f,s,u,z\}}=zusf$.
Because $(f,u)=(h',k')\vartriangleleft(f,s)=(a,b)$, we have $uN_{\{f,u,z\}}1\in N_p(\cd_{\{f,u,z\}})$ by the definition of $(a,b)$.
Hence, $sN_{\{f,s,z\}}1\in N_p(\cd_{\{f,s,z\}})$ by Lemma~\ref{c22}\ref{c222}.
But $bN_{\{a,b,z\}}1\notin N_p(\cd_{\{a,b,z\}})$ and $(a,b)=(f,s)$. This is a contradiction.

Therefore $t\ne f$. Recall $t\in F$, so $R_{\{f,s,t\}}=fts$.
Also $R_{\{s,u\}}=su$.
Hence, $R_{\{f,s,t,u\}}=ftsu$.
Note that $(h,k)\in G$ and $\{h',k'\}\cap\{h,k\}=\{t\}$. Therefore, $(h',k')\vartriangleleft(h,k)\trianglelefteq(f,t)\vartriangleleft(f,s)$.
Hence, $(t,u)=(h',k')\vartriangleleft(f,t)\vartriangleleft(f,s)$.
Therefore, $(s,u)$ and $(f,u)$ are in $S(\cg)$ and $(s,u)\vartriangleleft(t,u)\vartriangleleft(f,u)\vartriangleleft(f,s)$ by
Lemma~\ref{2geodesics4}\ref{2geodesics4b}.
Since $R_{\{f,s,u\}}=fsu$ and there are three switching pairs in $S(\cg_{\{f,s,u\}})$, we have $T_{\{f,s,u\}}=usf$.
Together with $T_{\{k',z\}}=zk'=zu$, we have $T_{\{f,s,u,z\}}=zusf$.
Now $(d,e)\trianglelefteq(h',k')=(t,u)\vartriangleleft(f,u)\vartriangleleft(f,s)=(a,b)$.
Therefore, $uN_{\{f,u,z\}}1\in N_p(\cd_{\{f,u,z\}})$ by the definition of $(a,b)$.
Hence, $sN_{\{f,s,z\}}1\in N_p(\cd_{\{f,s,z\}})$ by Lemma~\ref{c22}\ref{c222}.
But $bN_{\{a,b,z\}}1\notin N_p(\cd_{\{a,b,z\}})$ and $(a,b)=(f,s)$. This is a contradiction.

This completes the proof of Proposition~\ref{conex2}.
\end{proof}

%
%
%
 \bibliographystyle{splncs04}
    \bibliography{Ref}
    
\end{document}